\documentclass[a4paper,11pt,oneside]{amsart}
\pdfoutput=1

\usepackage[utf8]{inputenc}
\usepackage{bm}
\usepackage{mathtools,amssymb}
\usepackage{hyperref}
\usepackage{xcolor,graphicx}
\usepackage{mathrsfs}
\usepackage[shortlabels]{enumitem}
\usepackage{lineno}
\usepackage{amsmath,eufrak}
\usepackage{enumitem}
\usepackage{amsthm} 
\usepackage{fullpage} 

\DeclarePairedDelimiter\floor{\lfloor}{\rfloor}

\allowdisplaybreaks
\newcommand{\para}[1]{\vspace{3mm} \noindent\textbf{#1.}} 

\mathtoolsset{showonlyrefs}

\graphicspath{{images/}}

\newtheorem{theorem}{Theorem}[section]
\newtheorem{lemma}[theorem]{Lemma}

\newtheorem{corollary}[theorem]{Corollary}
\newtheorem{question}[theorem]{Question}
\newtheorem{remark}[theorem]{Remark}

\title[UCP and Poincar\'e inequality for higher order fractional Laplacians]{Unique continuation property and Poincar\'e inequality for higher order fractional Laplacians with applications in inverse problems}

\author{Giovanni Covi}
\address{Department of Mathematics and Statistics, University of Jyv\"askyl\"a, Jyv\"askyl\"a, Finland}
\email{giovanni.g.covi@jyu.fi}
\author{Keijo M\"onkk\"onen}
\email{kematamo@student.jyu.fi}
\author{Jesse Railo}
\address{Seminar for Applied Mathematics, Department of Mathematics, ETH Zurich, Z\"urich, Switzerland}
\email{jesse.railo@math.ethz.ch}

\date{\today}

\newcommand{\R}{{\mathbb R}}
\newcommand{\Z}{{\mathbb Z}}
\newcommand{\N}{{\mathbb N}}

\newcommand{\der}{{\mathrm d}}

\newcommand{\riesz}{I_{\alpha}}
\newcommand{\dplane}{R_d}
\newcommand{\nod}{N_d}

\newcommand{\schwartz}{\mathscr{S}}
\newcommand{\cschwartz}{\mathscr{S}_0}
\newcommand{\tempered}{\mathscr{S}^{\prime}}
\newcommand{\rapidly}{\mathscr{O}_C^{\prime}}
\newcommand{\slowly}{\mathscr{O}_M}
\newcommand{\fraclaplace}{(-\Delta)^s}
\newcommand{\fourier}{\mathcal{F}}
\newcommand{\ifourier}{\mathcal{F}^{-1}}

\newcommand{\csmooth}{\mathcal{D}}
\newcommand{\smooth}{\mathcal{E}}
\newcommand{\cdistr}{\mathcal{E}'}
\newcommand{\distr}{\mathcal{D}^{\prime}}
\newcommand{\dimens}{n}
\newcommand{\kernel}{h_{\alpha}} 

\newcommand{\abs}[1]{\left\lvert #1 \right\rvert}
\newcommand{\aabs}[1]{\left\lVert #1 \right\rVert}
\newcommand{\ip}[2]{\left\langle #1,#2 \right\rangle}
\DeclareMathOperator{\spt}{spt}

\begin{document}

\maketitle
\begin{abstract} We prove a unique continuation property for the fractional Laplacian $(-\Delta)^s$ when $s \in (-n/2,\infty)\setminus \mathbb{Z}$ where $n\geq 1$. In addition, we study Poincar\'e-type inequalities for the operator $(-\Delta)^s$ when $s\geq 0$. We apply the results to show that one can uniquely recover, up to a gauge, electric and magnetic potentials from the Dirichlet-to-Neumann map associated to the higher order fractional magnetic Schr\"odinger equation. We also study the higher order fractional Schr\"odinger equation with singular electric potential. In both cases, we obtain a Runge approximation property for the equation. Furthermore, we prove a uniqueness result for a partial data problem of the $d$-plane Radon transform in low regularity. Our work extends some recent results in inverse problems for more general operators.
\end{abstract}

\section{Introduction}

The fractional Laplacian~$\fraclaplace$, $s\in (-n/2,\infty)\setminus \Z$, is a non-local operator by definition and thus differs substantially from the ordinary Laplacian~$(-\Delta)$. The non-local behaviour can be exploited when solving fractional inverse problems. In section \ref{subsec:uniquecontinuationresults}, we prove that ~$\fraclaplace$ admits a unique continuation property (UCP) for open sets, that is, if~$u$ and~$\fraclaplace u$ both vanish in a nonempty open set, then~$u$ vanishes everywhere. Clearly this property cannot hold for local operators. We give many other versions of UCPs as well.

We have also included a quite comprehensive discussion of the Poincar\'e inequality for the higher order fractional Laplacian~$\fraclaplace$, $s \geq 0$, in section~\ref{subsec:poinacareinequality}. We give many proofs for the higher order fractional Poincar\'e inequality based on various different methods in the literature. The higher order fractional Poincar\'e inequality appears earlier at least in \cite{XI-note-on-fractional-poincare} for functions in $C_c^\infty(\Omega)$ where $\Omega$ is a bounded Lipschitz domain. Also similar inequalities are proved in the book~\cite{BCD-fourier-analysis-nonlinear-pde} for homogeneous Sobolev norms but without referring to the fractional Laplacian. However, we have extended some known results, given alternative proofs, and studied a connection between the fractional and the classical Poincar\'e constants. We believe that section~\ref{subsec:poinacareinequality} will serve as a helpful reference on fractional Poincar\'e inequalities in the future.

Our main applications are fractional Schr\"odinger equations with and without a magnetic potential, and the $d$-plane Radon transforms with partial data. We apply the UCP result and the Poincar\'e inequality for higher order fractional Laplacians to show uniqueness for the associated fractional Schr\"odinger equation and the Runge approximation properties. UCPs have also applications in integral geometry since certain partial data inverse problems for the Radon transforms can be reduced to unique continuation problems of the normal operators. We remark that the normal operators of the Radon transforms are negative order fractional Laplacians (Riesz potentials) up to constant coefficients.

In this section, we introduce our models, discuss some related results and present our main theorems and corollaries. We start with the classical Calder\'on problem as a motivation.

\subsection{The Calder\'on problem} We will study a non-local version of the famous Calder\'on problem called the fractional Calder\'on problem.  A survey of the fractional Calder\'on problem is given in \cite{Sal17}. The Calder\'on problem is a classical inverse problem where one wants to determine the electrical conductivity on some sufficiently smooth domain by boundary measurements \cite{SA:calderon-problem, UH-inverse-problems-seeing-the-unseen}. Suppose that $\Omega \subset \R^n$ is a domain with regular enough boundary $\partial\Omega$. The electrical conductivity is usually represented as a bounded positive function $\gamma$, and the conductivity equation is
\begin{align}\label{eq:conductivityequation} \left\{\begin{array}{rl}
        \nabla\cdot(\gamma\nabla u)&=0 \;\;\text{in}\;  \Omega\\
        u|_{\partial\Omega} &=f
        \end{array}\right.\end{align}
where $f$ is the potential on the boundary $\partial\Omega$ and $u$ is the induced potential in $\Omega$. The data in this problem is the Dirichlet-to-Neumann (DN) map $\Lambda_{\gamma}(f)=(\gamma\partial_{\nu}u)|_{\partial\Omega}$, where $\nu$ is the outer unit normal on the boundary. The DN map basically tells how the applied voltage on the boundary induces normal currents on the boundary by the electrical properties of the interior. The inverse problem is to determine $\gamma$ from the DN map $\Lambda_{\gamma}$. One of the associated basic questions is the uniqueness problem, that is, whether $\gamma_1=\gamma_2$ follows from $\Lambda_{\gamma_1}=\Lambda_{\gamma_2}$. 

Equation~\eqref{eq:conductivityequation} can be reduced to a Schr\"odinger equation
\begin{align}\label{eq:scrodingercalderon} \left\{\begin{array}{rl}
        (-\Delta+q)u&=0  \;\;\text{in}\;  \Omega\\
        u|_{\partial\Omega} &=f
        \end{array}\right.\end{align}
where $q=(\Delta\sqrt{\gamma})/\sqrt{\gamma}$ now represents the electric potential in $\Omega$. One typically assumes that $0$ is not a Dirichlet eigenvalue of the operator $(-\Delta+q)$ to obtain unique solutions to equation~\eqref{eq:scrodingercalderon}. The inverse problem then is to know whether one can determine the electric potential $q$ uniquely from the DN map $\Lambda_q$, which can be expressed in terms of the normal derivative $\Lambda_q f=\partial_{\nu}u|_{\partial\Omega}$ for regular enough boundaries. For more details on the classical Calder\'on problem and its applications to medical, seismic and industrial imaging, see \cite{SA:calderon-problem, UH-inverse-problems-seeing-the-unseen}.

\subsection{Fractional Schr\"odinger equation}
In this article, we focus on the fractional Schr\"odinger equation and its generalization, the fractional magnetic Schr\"odinger equation. The main difference between the classical and fractional Schr\"odinger operators is that the first one is local and the second one is non-local. This can be seen since the Laplacian $(-\Delta)$ is local as a differential operator while the fractional counterpart $\fraclaplace$, $s\in\R^+\setminus\Z$, is a non-local Fourier integral operator. In other words, the value $\fraclaplace u(x)$, $s \in \R^+\setminus\Z$, depends on the values of $u$ everywhere, not just in a small neighbourhood of $x \in \R^n$. Fractional Laplacians have a close connection to Lev\'y processes and have been used in many areas of mathematics and physics, for example to model anomalous and nonlocal diffusion, and also in the formulation of fractional quantum  mechanics where the fractional Schr\"odinger equation arises naturally as a generalization of the ordinary Schr\"odinger equation~\cite{VMRTM-nonlocal-diffusion-problems, BV-nonlocal-diffusion-applications, DGLZ2012, GSU-calderon-problem-fractional-schrodinger, La00, LA-fractional-quantum-mechanics, KM-random-walks-quide-anomalous-diffusion, RO-nonlocal-elliptic-equations-bounded-domains}.

Since the fractional Laplacian is a non-local operator, it is more natural to fix exterior values for the solutions of the equation instead of just boundary values. This motivates the study of the following exterior value problem, first introduced in~\cite{GSU-calderon-problem-fractional-schrodinger},
\begin{align}\label{eq:fractionalschrodingerequation} \left\{\begin{array}{rl}
        (\fraclaplace+q)u&=0  \;\;\text{in}\;  \Omega\\
        u|_{\Omega_e} &=f
        \end{array}\right.\end{align}
where $\Omega_e=\R^\dimens\setminus\overline{\Omega}$ is the exterior of $\Omega$. The associated DN map for equation~\eqref{eq:fractionalschrodingerequation} is a bounded linear operator $\Lambda_q\colon H^s(\Omega_e)\rightarrow (H^s(\Omega_e))^*$ which, under stronger assumptions, has an expression $\Lambda_q f=\fraclaplace u|_{\Omega_e}$ \cite{GSU-calderon-problem-fractional-schrodinger}. We assume that the potential $q$ is such that the following holds:
\begin{equation}
\label{def:zeronotdirichleteigenvalue}
\text{If} \ u\in H^s(\R^\dimens) \ \text{solves} \ (\fraclaplace+q)u=0 \ \text{in} \ \Omega \ \text{and} \ u|_{\Omega_e}=0, \ \text{then} \ u=0.
\end{equation}
In other words, condition~\eqref{def:zeronotdirichleteigenvalue} requires that $0$ is not a Dirichlet eigenvalue of the operator $(\fraclaplace+q)$.

In section \ref{sec:schrodingerequation}, we will prove that, under certain assumptions, one can uniquely determine the potential $q$ in equation~\eqref{eq:fractionalschrodingerequation} from exterior measurements when $s\in\R^+\setminus\Z$, and we also prove a Runge approximation property for equation~\eqref{eq:fractionalschrodingerequation} (see also section \ref{subsec:mainresults}). These generalize the results in \cite{GSU-calderon-problem-fractional-schrodinger, RS-fractional-calderon-low-regularity-stability} to higher fractional powers of $s$. The proofs basically reduce to the fact that the operator $\fraclaplace$ has the following UCP: if $\fraclaplace u|_V=0$ and $u|_V=0$ for some nonempty open set $V\subset\R^\dimens$, then $u=0$ everywhere. This reflects the fact that $\fraclaplace$ is a non-local operator since such UCP can never hold for local operators.

Unique continuation of the fractional Laplacian has been extensively studied and used to show uniqueness results for fractional Schr\"odinger equations \cite{CO-magnetic-fractional-schrodinger, GRSU-fractional-calderon-single-measurement, GSU-calderon-problem-fractional-schrodinger, RS-fractional-calderon-low-regularity-stability}. One version was already proved by Riesz~\cite{GSU-calderon-problem-fractional-schrodinger, RI-liouville-riemann-integrals-potentials} and similar methods were used in~\cite{IM-unique-continuation-riesz-potential} to show a UCP of Riesz potentials~$\riesz$ which can be seen as fractional Laplacians with negative exponents. See also~\cite{KR-all-functions-are-s-harmonic} for a unique continuation result of Riesz potentials. UCP of~$\fraclaplace$ for functions in $H^r(\R^\dimens)$, $r\in\R$, was proved in  \cite{GSU-calderon-problem-fractional-schrodinger}  when $s\in (0, 1)$. The proof is based on Carleman estimates from \cite{RU-unique-continuation-scrodinger-rough-potentials} and on Caffarelli-Silvestre extension \cite{CS-nonlinera-equations-fractional-laplacians,CS-extension-problem-fractional-laplacian}. Using the known result for $s\in (0, 1)$, we provide an elementary proof which generalizes the UCP for all $s\in (-\dimens/2, \infty)\setminus\Z$. With the same trick we obtain several other unique continuation results. There are also strong unique continuation results for $s\in (0, 1)$ if one assumes more regularity from the function \cite{FF-unique-continuation-fractional-ellliptic-equations, RU-unique-continuation-scrodinger-rough-potentials}. In the strong UCP, one replaces the condition $u|_V=0$ by the requirement that $u$ vanishes to infinite order at some point $x_0\in V$. The higher order case $s\in\R^+\setminus (\Z\cup (0, 1))$ has been studied recently by several authors \cite{FF-unique-continuation-higher-laplacian, GR-fractional-laplacian-strong-unique-continuation, YA-higher-order-laplacian}. These results however assume some special conditions on the function~$u$, i.e. they require that $u$ is in a Sobolev space which depends on the power~$s$ of the fractional Laplacian $\fraclaplace$. We only require that $u$ is in some Sobolev space $H^r(\R^\dimens)$ where~$r\in\R$ can be an arbitrarily small (negative) number. 

See also \cite{KR-all-functions-are-s-harmonic} where the author proves a higher order Runge approximation property by $s$-harmonic functions in the unit ball when $s\in\R^+\setminus\Z$ (compare to theorem \ref{thm:schrodingerrungeapproximation}). Here $s$-harmonicity simply means that $\fraclaplace u=0$ in some domain $\Omega$. The $s$-harmonic approximation in the case $s\in (0, 1)$ was already studied in \cite{DSV-all-functions-are-s-harmonic}; similar higher regularity approximation results are proved in \cite{CLR18, GSU-calderon-problem-fractional-schrodinger} for the fractional Schr\"odinger equation.

\subsection{Fractional magnetic Schr\"odinger equation}
Section \ref{sec:magneticschrodinger} of this paper extends the study of the fractional magnetic Schr\"odinger equation (FMSE) begun in \cite{CO-magnetic-fractional-schrodinger}, expanding the uniqueness result for the related inverse problem to the cases when $s\in \mathbb R^+ \setminus \mathbb Z$. The direct problem for the classical magnetic Schr\"odinger equation (MSE) consists in finding a function $u$ satisfying  
\begin{align*} \left\{\begin{array}{rl}
        (-\Delta)_A u + qu &= -\Delta u -i \nabla\cdot(Au) -i A\cdot\nabla u + (|A|^2+q)u =0  \;\;\text{in}\;  \Omega\\
        u|_{\partial \Omega}&=f
        \end{array}\right.\end{align*}
\noindent where $\Omega\subset\mathbb R^n$ is some bounded open set with Lipschitz boundary representing a medium, $f$ is the boundary value for the solution $u$, and $A,q$ are the vector and scalar potentials of the equation. In the associated inverse problem, we are given measurements on the boundary in the form of a DN map $\Lambda_{A,q} : H^{1/2}(\partial\Omega)\rightarrow H^{-1/2}(\partial\Omega)$, and we are asked to recover $A, q$ in $\Omega$ using this information. It was shown in \cite{NSU95} that this is only possible up to a natural gauge: one can uniquely determine the potential $q$ and the magnetic \emph{field} $\text{curl} A$, but the magnetic \emph{potential} $A$ can not be determined in greater detail. The inverse problem for MSE is of great interest, because it generalizes the non-magnetic case by adding some first order terms, and shows a quite different behavior. It also possesses multiple applications in the sciences: the papers \cite{NSU95, NU94, Mc00, Es01, NT00} and \cite{HLW06} give some examples of this, treating the inverse scattering problem with a fixed energy, isotropic elasticity, the Maxwell, Schr\"odinger and Dirac equations and the Stokes system. We refer to the survey \cite{Sa07} for many more references on inverse boundary value problems related to MSE. 

We are interested in the study of a high order fractional version of the MSE. There have been many studies in this direction (see for instance \cite{LI-fractional-magnetic, LI-fractional-magnetic-2, LI-fractional-magnetic-3}). In our work, we will build upon the results from \cite{CO-magnetic-fractional-schrodinger} and generalize them to higher order. Thus, for us the direct problem for FMSE asks to find a function $u$ which satisfies 
\begin{align*} \left\{\begin{array}{rl}
        (-\Delta)^s_A u + qu &=0  \;\;\text{in}\;  \Omega\\
        u|_{\Omega_e} &=f
        \end{array}\right.\end{align*}
\noindent where $\Omega$, $f$, $A$ and $q$ play a similar role as in the local case, $s\in \mathbb R^+ \setminus\mathbb Z$ and $(-\Delta)^s_A$ is the magnetic fractional Laplacian. This is a fractional version of $(-i\nabla+A)\cdot(-i\nabla+A)$, the magnetic Laplacian from which MSE arises. In section \ref{sec:magneticschrodinger}, we will construct the fractional magnetic Laplacian based on the fractional gradient operator $\nabla^s$. The fractional gradient is based on the framework laid down in \cite{DGLZ2012,DGLZ2013}, and has been studied in the papers \cite{Co18,CO-magnetic-fractional-schrodinger}. One should keep in mind that for $s>1$ the fractional gradient is a tensor of order $\floor{s}$ rather than a vector. In the corresponding inverse problem, we assume to know the DN map $\Lambda_{A,q}^s : H^s(\Omega_e) \rightarrow (H^{s}(\Omega_e))^*$, and we wish to recover $A, q$ in $\Omega$. In the cases when $s\in(0,1)$, it has been shown that the pair $A,q$ can only be recovered up to a natural gauge \cite{CO-magnetic-fractional-schrodinger}. We generalize this result to the case $s\in \mathbb R^+\setminus \mathbb Z$. This is achieved by first proving a weak UCP and the Runge approximation property for FMSE, and then testing the Alessandrini identity for the equation with suitably chosen functions.

\begin{remark}
The case of the high order magnetic Schr\"odinger equation, that is the one in which $s\in\mathbb N$, $s\neq 1$, is still open at the time of writing to the best of the authors' knowledge. Our methods are purely nonlocal, and thus cannot be applied to the integer case. It was however showed in \cite{NSU95}, as cited above, that a uniqueness result up to a natural gauge holds when $s=1$.
\end{remark}

\subsection{Radon transforms and region of interest tomography}
Unique continuation results have also applications in integral geometry. It was proved in~\cite{IM-unique-continuation-riesz-potential} that the normal operator of the X-ray transform admits a UCP in the class of compactly supported distributions. This was done by considering the normal operator as a Riesz potential. We generalize the result for the normal operator of the $d$-plane transform~$\dplane$ where $d\in\N$ is odd such that $0<d<\dimens$. In the case $d=1$ the transform $\dplane$ corresponds to the X-ray transform and in the case $d=\dimens-1$ to the Radon transform. The UCP of the normal operator $\nod=\dplane^*\dplane$ implies uniqueness for the following partial data problem: if $f$ integrates to zero over all $d$-planes which intersect some nonempty open set $V$ and $f|_V=0$, then $f=0$. This can be seen as a complementary result to the Helgason support theorem for the $d$-plane transform \cite{HE:integral-geometry-radon-transforms}. Helgason's theorem says that if $f$ integrates to zero over all $d$-planes not intersecting a convex and compact set $K$ and $f|_K=0$, then $f=0$. The $d$-plane transform~$\dplane$ is injective on continuous functions which decay rapidly enough at infinity and also on compactly supported distributions~\cite{HE:integral-geometry-radon-transforms}. The $d$-plane transform has been recently studied in the periodic case on the flat torus~\cite{A11, I15, RA-periodic-radon-transform} but also in other settings~\cite{DE03, HH17, R15}. Weighted and limited data Radon transforms ($d = n-1$) have been studied recently for example in~\cite{FQ16, G17, GN18b, GN18a}.

When $d=1$, partial data problems as discussed above arise for example in seismology and medical imaging. In \cite{IM-unique-continuation-riesz-potential}, it is explained how one can use shear wave splitting data to uniquely determine the difference of the anisotropic perturbations in the S-wave speeds, and also how one can use local measurements of travel times of seismic waves to uniquely determine the conformal factor in the linearization. Both of these problems reduce to the following partial data result: if $f$ integrates to zero over all lines which intersect some nonempty open set $V$ and  $f|_V=0$, then $f=0$. In medical imaging, one typically wants to reconstruct a specific part of the human body. Can this be done by using only X-rays which go though our region of interest (ROI)? Generally this is not possible even for $C_c^{\infty}$-functions \cite{KEQ-wavelet-methods-ROI-tomography, NA-mathematics-computerized-tomography, SU:microlocal-analysis-integral-geometry}, but if we know some information of~$f$ in the ROI, then the reconstruction can be done. For example, if the function $f$ is piecewice constant, piecewice polynomial or analytic in the ROI, then $f$ can be uniquely determined from the X-ray data~\cite{KKW-stability-of-interior-problems, KEQ-wavelet-methods-ROI-tomography, YYJW-high-order-TV-minimization}. Also, if we know the X-ray data through the ROI and the values of $f$ in an arbitrarily small open set inside the ROI, then $f$ is uniquely determined everywhere~\cite{CNDK-solving-interior-problem-ct-with-apriori-knowledge, IM-unique-continuation-riesz-potential}. For practical applications of ROI tomography in medical imaging, see for example~\cite{YYW-interior-reconstruction-limited-angle-data, YW-compressed-interior-tomography}. See also \cite{KQ-microlocal-analysis-in-tomography, QU-singularities-x-ray-transform-limited-data,QU-artifacts-and-singularities-limited-tomography} for a discussion of the difficulties of obtaining stable reconstruction in partial data problems for the X-ray transform (visible and invisible singularities).

\subsection{Main results} 
\label{subsec:mainresults}
We briefly introduce the basic notation; more details can be found in sections \ref{sec:preliminaries}, \ref{sec:applicationstointegralgeometry}, \ref{sec:schrodingerequation} and \ref{sec:magneticschrodinger}. Let $H^r(\R^\dimens)$ be the $L^2$ Sobolev space of order $r\in\R$ and $\widetilde{H}^r(\Omega)$ the closure of $C^{\infty}_c(\Omega)$ in $H^r(\R^\dimens)$ when $\Omega$ is an open set. The $L^1$ Bessel potential space is denoted by $H^{r, 1}(\R^\dimens)$. We define $H_K^r(\R^\dimens)\subset H^r(\R^\dimens)$ to be those Sobolev functions which have support in the compact set~$K$. The fractional Laplacian is defined via the Fourier transform $\fraclaplace u=\ifourier(\abs{\cdot}^{2s}\hat{u})$. Then $\fraclaplace\colon H^r(\R^\dimens)\rightarrow H^{r-2s}(\R^\dimens)$ is a continuous operator when $s\in\R^+\setminus\Z$. The $d$-plane transform~$\dplane$ takes a function which decreases rapidly enough at infinity and integrates it over $d$-dimensional planes where $0<d<\dimens$. The normal operator of the $d$-plane transform is defined as $\nod=\dplane^*\dplane$ where~$\dplane^*$ is the adjoint operator. Further, we denote by $\distr(\R^\dimens)$ the space of all distributions, $\cdistr(\R^\dimens)$ the space of compactly supported  distributions, $\rapidly(\R^\dimens)$ the space of rapidly decreasing distributions and $C_{\infty}(\R^\dimens)$ the set of rapidly decreasing continuous functions. The space of singular potentials~$Z_0^{-s}(\R^\dimens)$ is a certain subset of distributions~$\distr(\R^\dimens)$ and can be interpreted as a set of bounded multipliers from $H^s(\R^\dimens)$ to $H^{-s}(\R^\dimens)$.

The following theorem extends a result in \cite{GSU-calderon-problem-fractional-schrodinger} and has a central role in this article. We call it the UCP of the operator $\fraclaplace$.

\begin{theorem}
\label{thm:uniquecontinuationoffractionallaplacian}
Let $\dimens \geq 1$, $s\in (-n/4,\infty)\setminus \Z$ and $u\in H^{r}(\R^\dimens)$ where $r\in\R$. If $(-\Delta)^s u|_V=0$ and $u|_V=0$ for some nonempty open set $V\subset\R^\dimens$, then $u=0$. The claim holds also for $s\in (-n/2, -n/4]\setminus\Z$ if $u\in H^{r, 1}(\R^\dimens)$ or $u\in\rapidly(\R^\dimens)$.
\end{theorem}

Theorem \ref{thm:uniquecontinuationoffractionallaplacian} is proved in section \ref{subsec:uniquecontinuationresults}. The UCP of~$\fraclaplace$ implies corresponding UCP for Riesz potentials (see corollary \ref{cor:uniquecontinuationofriespotential} and \cite[Theorem 5.2]{IM-unique-continuation-riesz-potential}). This in turn implies the following UCP for the normal operator of the $d$-plane transform~$\nod$ when~$d$ is odd; the case $d=1$ was already studied in~\cite{IM-unique-continuation-riesz-potential}.

\begin{corollary}
\label{cor:uniquecontinuationnormaloperator}
Let $\dimens \geq 2$ and let $f$ belong to either $\cdistr(\R^\dimens)$ or $ C_{\infty}(\R^\dimens)$. Let $d \in \N$ be odd such that $0<d<\dimens$. If $\nod f|_V=0$ and $f|_V=0$ for some nonempty open set $V\subset\R^\dimens$, then $f=0$.
\end{corollary}

From the UCP of~$\nod$ we obtain the next result which is in a sense complementary to the Helgason support theorem for the $d$-plane transform \cite[Theorem 6.1]{HE:integral-geometry-radon-transforms}. It extends a result in~\cite{IM-unique-continuation-riesz-potential} where the authors prove a similar uniqueness property for the X-ray transform.

\begin{corollary}
\label{cor:partialdataresult}
Let $\dimens \geq 2$, $V\subset\R^\dimens$ a nonempty open set and $f\in C_{\infty}(\R^\dimens)$. Let $d\in\N$ be odd such that $0<d<\dimens$. If $f|_V=0$ and $\dplane f=0$ for all $d$-planes intersecting $V$, then $f=0$. The claim holds also for $f\in\cdistr(\R^\dimens)$ when the assumption $\dplane f=0$ for all $d$-planes intersecting $V$ is understood in the sense of distributions. 
\end{corollary}

If $d$ is even, then $f$ is uniquely determined in $V$ by its integrals over $d$-planes which intersect~$V$, i.e. $\dplane f=0$ for all $d$-planes intersecting~$V$ implies $f|_V=0$ (see remark \ref{remark:partialdata}). The authors do not know if the result of corollary \ref{cor:partialdataresult} holds when $d$ is even. However, if $d$ is even, then the result of corollary \ref{cor:uniquecontinuationnormaloperator} cannot be true as the normal operator $N_d$ is the inverse of a local operator. See section \ref{sec:applicationstointegralgeometry} for the proofs and the definition of the $d$-plane transform of distributions. 

The following result is a general version of the Poincar\'e inequality which we need for the well-posedness of the inverse problem for the fractional Schr\"odinger equation.

\begin{theorem}
\label{thm:generalpoincare}
Let $\dimens \geq 1$, $s\geq t\geq 0$, $K\subset\R^\dimens$ a compact set and $u\in H_K^s(\R^\dimens)$. There exists a constant $\tilde{c}=\tilde{c}(n, K, s)> 0$ such that
\begin{equation}
\aabs{(-\Delta)^{t/2}u}_{L^2(\R^\dimens)}\leq \tilde{c}\aabs{(-\Delta)^{s/2}u}_{L^2(\R^\dimens)}.
\end{equation}
\end{theorem}

The constant $\tilde{c}$ can be expressed in terms of the classical Poincar\'e constant when $s\geq 1$ (see theorem \ref{thm:poincareinterpolation}. See section \ref{subsec:poinacareinequality} for several proofs of the Poincar\'e inequality. From the unique continuation of~$\fraclaplace$ we obtain results for the higher order fractional Schr\"odinger equation with singular electric potential. The following theorems generalize the results in \cite{GSU-calderon-problem-fractional-schrodinger, RS-fractional-calderon-low-regularity-stability} for higher exponents $s\in\R^+\setminus (\Z\cup (0, 1))$.

\begin{theorem}
\label{thm:schrodingeruniqueness}
Let $\dimens \geq 1$, $\Omega\subset\R^\dimens$ a bounded open set, $s\in\R^+\setminus\Z$, and $q_1, q_2\in Z_0^{-s}(\R^\dimens)$ which satisfy condition~\eqref{def:zeronotdirichleteigenvalue}. Let $W_1, W_2\subset\Omega_e$ be open sets. If the DN maps for the equations $\fraclaplace u+m_{q_i}(u)=0$ in $\Omega$ satisfy $\Lambda_{q_1}f|_{W_2}=\Lambda_{q_2}f|_{W_2}$ for all $f\in C_c^{\infty}(W_1)$, then $q_1|_{\Omega}=q_2|_{\Omega}$.
\end{theorem}

\begin{theorem}
\label{thm:schrodingerrungeapproximation}
Let $\dimens \geq 1$ and $s\in\R^+\setminus\Z$. Let $\Omega\subset\R^\dimens$ be a bounded open set and $\Omega_1\supset\Omega$ any open set such that $\text{int}(\Omega_1\setminus\Omega)\neq \varnothing$. If $q\in Z_0^{-s}(\R^\dimens)$ satisfies condition~\eqref{def:zeronotdirichleteigenvalue}, then any $g\in \widetilde{H}^s(\Omega)$ can be approximated arbitrarily well in $\widetilde{H}^s(\Omega)$ by solutions $u\in H^s(\R^\dimens)$ to the equation $\fraclaplace u+m_q(u)=0$ in $\Omega$ such that $\spt(u)\subset\overline{\Omega}_1$.
\end{theorem}

We remark that the approximation property in theorem \ref{thm:schrodingerrungeapproximation} also holds in~$L^2(\Omega)$ when one takes restrictions of the solutions (see \cite[Theorem 1.3]{GSU-calderon-problem-fractional-schrodinger}). In~\cite{DSV-all-functions-are-s-harmonic, KR-all-functions-are-s-harmonic} the authors prove similar approximation results: $C^k$-functions can be approximated (in the $C^k$-norm) in the unit ball by $s$-harmonic functions, i.e. functions~$u$ which satisfy $\fraclaplace u=0$ in $B_1(0)$ (see also \cite[Remark 7.3]{GSU-calderon-problem-fractional-schrodinger}). Theorems \ref{thm:schrodingeruniqueness} and \ref{thm:schrodingerrungeapproximation} are proved in section \ref{sec:schrodingerequation}. The proofs are almost identical to those in \cite{GSU-calderon-problem-fractional-schrodinger, RS-fractional-calderon-low-regularity-stability} and only slight changes need to be done. We will present the main ideas of the proofs for clarity and in order to make a comparison to the more complicated case of FMSE. 

We have achieved the following result on the Calder\'on problem for FMSE:

\begin{theorem}
\label{thm:FMSE}
Let $\Omega \subset \mathbb R^n, \; n\geq 2$, be a bounded open set, $s\in\mathbb R^+ \setminus \mathbb Z$, and let $A_i, q_i$ verify assumptions \ref{assumption1}-\ref{assumption5} in section \ref{sec:magneticschrodinger} for $i=1,2$. Let $W_1, W_2\subset \Omega_e$ be open sets. If the DN maps for the FMSEs in $\Omega$ relative to $(A_1,q_1)$ and $(A_2,q_2)$ satisfy   $$\Lambda^s_{A_1,q_1}[f]|_{W_2}=\Lambda^s_{A_2,q_2}[f]|_{W_2} \;\;\;\;\; \mbox{for all}\; f\in C^\infty_c(W_1),$$ \noindent then $(A_1,q_1)\sim(A_2,q_2)$, that is, the potentials coincide up to gauge.
\end{theorem}

\noindent An in-depth clarification of the assumptions and the definition of the gauge involved in the proof are presented in section \ref{sec:magneticschrodinger}.

\subsection{Organization of the article}
This article is organized as follows. Section~\ref{sec:preliminaries} is devoted to preliminaries. We introduce our notation and definitions of relevant quantities. In sections \ref{subsec:uniquecontinuationresults} and \ref{subsec:poinacareinequality} we prove the unique continuation property of~$\fraclaplace$ for $s\in (-n/2,\infty)\setminus \Z$ and give several proofs for the fractional Poincar\'e inequality. We introduce some applications in integral geometry and partial data problems of the $d$-plane transform in section \ref{sec:applicationstointegralgeometry}. In section~\ref{sec:schrodingerequation}, we show the uniqueness and the Runge approximation results for the higher order fractional Schr\"odinger equation with singular electric potential. We prove the uniqueness result up to a gauge for the higher order fractional magnetic Schr\"odinger equation in section~\ref{sec:magneticschrodinger}. Finally, in section~\ref{sec:possiblegeneralizations}, we discuss other problems that would now naturally continue our work. There are many potential recent results in inverse problems which perhaps can be generalized to higher order fractional Laplacians using our unique continuation result and fractional Poincar\'e inequality.

\subsection*{Acknowledgements} The authors wish to thank Yi-Hsuan Lin for suggesting to study higher order fractional Calder\'on problems and for his idea of reducing the UCP of higher order fractional Laplacians  to the case $s\in (0, 1)$. The authors are grateful to Mikko Salo for proposing a proof for the fractional Poincar\'e inequality for $n=1$ and $s \in (1/2,1)$, and for many other helpful discussions. We thank Joonas Ilmavirta for discussions about integral geometry. The authors wish to thank the anonymous referees for helpful comments and suggestions to improve the article. G.C. was partially supported by the European Research Council under Horizon 2020 (ERC CoG 770924). K.M. and J.R. were supported by Academy of Finland (Centre of Excellence in Inverse Modelling and Imaging, grant numbers 284715 and 309963).

\section{Preliminaries}
\label{sec:preliminaries}
In this section, we will go through our basic notations and definitions. The following theory of distributions, Fourier analysis and Sobolev spaces can be found in many books (see for example \cite{AB-psidos-and-singular-integrals, BCD-fourier-analysis-nonlinear-pde, BL-interpolation-spaces, HO:analysis-of-pdos, HO-topological-vector-spaces,  ML-strongly-elliptic-systems, MI:distribution-theory, SA:fourier-analysis-distributions, TRE:topological-vector-spaces-distributions}). We write $\abs{\cdot}$ for both the Euclidean norm of vectors and the absolute value of complex numbers. We denote by $\N_0$ the set of natural numbers including zero.

\subsection{Distributions and Fourier transform}\label{sec:distributions}
We denote by $\smooth(\R^\dimens)$ the set of smooth functions equipped with the topology of uniform convergence of derivatives of all order on compact sets. We also denote by $\csmooth(\R^\dimens)$ the set of compactly supported smooth functions with the topology of uniform convergence of derivatives of all order in a fixed compact set. The topological duals of these spaces are denoted by $\distr(\R^\dimens)$ and $\cdistr(\R^\dimens)$. Elements in the space $\cdistr(\R^\dimens)$ can be identified as distributions in $\distr(\R^\dimens)$ with compact support. 

We also use the space of rapidly decreasing smooth functions, i.e. Schwartz functions. Define the Schwartz space as
$$
\schwartz(\R^\dimens)=\left\{\varphi\in C^{\infty}(\R^\dimens): \aabs{\langle\cdot\rangle^N\partial^{\beta}\varphi}_{L^{\infty}(\R^\dimens)}<\infty \ \text{for all} \ N\in\N \ \text{and} \  \beta\in\N_0^\dimens\right\},
$$
where $\langle x\rangle=(1+\abs{x}^2)^{1/2}$, equipped with the topology induced by the seminorms $\aabs{\langle\cdot\rangle^N\partial^{\beta}\varphi}_{L^{\infty}(\R^\dimens)}$. The continuous dual of $\schwartz(\R^\dimens)$ is denoted by $\tempered(\R^\dimens)$ and its elements are called tempered distributions. We have the continuous inclusions $\cdistr(\R^\dimens)\subset\tempered(\R^\dimens)\subset\distr(\R^\dimens)$. The Fourier transform of $u\in L^1(\R^\dimens)$ is defined as
$$
(\fourier u)(\xi)=\hat{u}(\xi)=\int_{\R^\dimens}e^{-ix\cdot\xi}u(x)\der x
$$
and it is an isomorphism $\fourier\colon\schwartz(\R^\dimens)\rightarrow\schwartz(\R^\dimens)$. By duality the Fourier transform is also an isomorphism $\fourier\colon\tempered(\R^\dimens)\rightarrow\tempered(\R^\dimens)$. By density of $\schwartz(\R^\dimens)$ in $L^2(\R^\dimens)$ the Fourier transform can be extended to an isomorphism $\fourier\colon L^2(\R^\dimens)\rightarrow L^2(\R^\dimens)$. The following subset of Schwartz space
$$
\cschwartz(\R^\dimens)=\{\varphi\in\schwartz(\R^\dimens):\hat{\varphi}|_{B(0, \epsilon)}=0 \ \text{for some} \ \epsilon>0\}
$$
is used to define fractional Laplacians on homogeneous Sobolev spaces.

Finally, we denote by $\rapidly(\R^\dimens)$ the space of rapidly decreasing distributions. One has that $T\in\rapidly(\R^\dimens)$ if and only if for any $N\in\N$ there exist $M(N)\in\N$ and continuous functions $g_{\beta}$ such that
$$
T=\sum_{\abs{\beta}\leq M(N)}\partial^{\beta}g_{\beta}\;,
$$
where $\langle\cdot\rangle^N g_{\beta}$ is a bounded function for every $\abs{\beta}\leq M(N)$. Alternatively one can characterize $\rapidly(\R^\dimens)$ via the Fourier transform: it holds that $\fourier\colon\rapidly(\R^\dimens)\rightarrow\slowly(\R^\dimens)$ is a bijective map where $\slowly(\R^\dimens)$ is the space of smooth functions with polynomially bounded derivatives of all orders. We have the continuous inclusions $\cdistr(\R^\dimens)\subset\rapidly(\R^\dimens)\subset\tempered(\R^\dimens)$. For example $C_{\infty}(\R^\dimens)\subset\rapidly(\R^\dimens)$, where $f\in C_{\infty}(\R^\dimens)$ if and only if $f$ is continuous and $\langle\cdot\rangle^N f$ is bounded for every $N\in\N$. The convolution formula for the Fourier transform $\widehat{f\ast g}=\hat{f}\hat{g}$ holds in the sense of distributions when $f\in\rapidly(\R^\dimens)$ and $g\in\tempered(\R^\dimens)$. For more details on distributions, see the classic books~\cite{HO:analysis-of-pdos, HO-topological-vector-spaces, TRE:topological-vector-spaces-distributions}.

\subsection{Fractional Laplacian on Sobolev spaces}
\label{sec:sobolevspacesandfractionallaplacian}
Let $r\in\R$. We define the inhomogeneous fractional $L^2$ Sobolev space of order $r$ to be the set
$$
H^{r}(\R^\dimens)=\{u\in\tempered(\R^\dimens): \ifourier(\langle\cdot\rangle^r\hat{u})\in L^2(\R^\dimens)\}
$$
equipped with the norm
$$
\aabs{u}_{H^{r}(\R^\dimens)}=\aabs{\ifourier(\langle\cdot\rangle^r\hat{u})}_{L^2(\R^\dimens)}.
$$
The spaces $H^r(\R^\dimens)$ are Hilbert spaces for all $r\in\R$. It follows that both $\schwartz(\R^\dimens)$ and $\cschwartz(\R^\dimens)$ are dense in $H^r(\R^\dimens)$ for all $r\in\R$. Note that 
$$
\rapidly(\R^\dimens)\subset\bigcup_{r\in\R}H^r(\R^\dimens).
$$

If $s\in (0, 1)$, the fractional Laplacian can be defined in several equivalent ways~\cite{KWA-ten-definitions-fractional-laplacian}. We will take the Fourier transform approach which allows us to define it as a continuous map on Sobolev spaces for all $s\in\R^+\setminus\Z$. Define the fractional Laplacian of order $s\in\R^+\setminus\Z$ as $\fraclaplace \varphi=\ifourier(\abs{\cdot}^{2s}\hat{\varphi})$ for $\varphi\in\schwartz(\R^\dimens)$. Then $\fraclaplace\colon\schwartz(\R^\dimens)\rightarrow H^{r-2s}(\R^\dimens)$ is linear and continuous with respect to the norm $\aabs{\cdot}_{H^{r}(\R^\dimens)}$ by a simple calculation. Thus we can uniquely extend it to a continuous linear operator $\fraclaplace\colon H^r(\R^\dimens)\rightarrow H^{r-2s}(\R^\dimens)$ as $\fraclaplace u=\lim_{k\rightarrow\infty}\fraclaplace\varphi_k$, where $\varphi_k\in\schwartz(\R^\dimens)$ is such that $\varphi_k\rightarrow u$ in $H^{r}(\R^\dimens)$.

On the other hand, if $s>-\dimens/4$, one can always define $\fraclaplace u$ for $u\in H^r(\R^\dimens)$ as the tempered distribution $\fraclaplace u=\ifourier(\abs{\cdot}^{2s}\hat{u})$, note that we also allow integer values of $s$ here. This can be seen in the following way: let $\varphi_k\in\schwartz(\R^\dimens)$ such that $\varphi_k\rightarrow 0$ in $\schwartz(\R^\dimens)$. It holds that $\abs{\cdot}^{-\beta}\in L^1_{loc}(\R^\dimens)$ if and only if $\beta<\dimens$. Taking $N\in\N$ large enough and using Cauchy-Schwartz we obtain
\begin{align*}
\int_{\R^\dimens}\abs{x}^{2s}\abs{\hat{u}(x)}\abs{\varphi_k(x)}\der x&\leq \bigg(\int_{\R^\dimens}\langle x\rangle^{2r}\abs{\hat{u}(x)}^2\der x\bigg)^{1/2}\bigg(\int_{\R^\dimens}\abs{x}^{4s}\langle x\rangle^{-2r}\abs{\varphi_k(x)}^2\der x\bigg)^{1/2} \\
&\leq C\bigg(\int_{\R^\dimens}\frac{\abs{x}^{4s}}{\langle x\rangle^{2N}}\der x\bigg)^{1/2}\aabs{\langle\cdot\rangle^{N-r}\varphi_k}_{L^{\infty}(\R^\dimens)}\rightarrow 0.
\end{align*}
Hence $\abs{\cdot}^{2s}\hat{u}\in\tempered(\R^\dimens)$ and also $\fraclaplace u=\ifourier(\abs{\cdot}^{2s}\hat{u})\in\tempered(\R^\dimens)$. The definition can be relaxed to $s>-\dimens/2$ if we assume that $\langle\cdot\rangle^t\hat{u}\in L^{\infty}(\R^\dimens)$ for some $t\in\R$. This holds for example if $u\in\rapidly(\R^\dimens)$ or $u\in H^{r, 1}(\R^\dimens)$ (see the definition of Bessel potential spaces below). When $s\geq 0$, we again obtain that $\fraclaplace\colon H^r(\R^\dimens)\rightarrow H^{r-2s}(\R^\dimens)$ is continuous. It follows from the properties of the Fourier transform that $(-\Delta)^k\fraclaplace=(-\Delta)^{k+s}$ when $s>-\dimens/2$ and $k\in\N$. This relation will be used many times.

Fractional Laplacians with negative powers $s$ have a connection to Riesz potentials. Let $\alpha\in\R$ such that $0<\alpha<\dimens$. We define the Riesz potential $\riesz\colon\rapidly(\R^\dimens)\rightarrow\tempered(\R^\dimens)$ as $\riesz f=f\ast\kernel$, where the kernel is $\kernel(x)=\abs{x}^{-\alpha}$. It follows that~$\riesz$ is continuous in the distributional sense and $\riesz=(-\Delta)^{-s}$, up to a constant factor, where $s=(\dimens-\alpha)/2$. On the other hand, if $-\dimens/2<s<0$, then one can write $\fraclaplace f=f\ast\abs{\cdot}^{-2s-n}=I_{2s+n}f$, also up to a constant factor. Hence fractional Laplacians with negative powers correspond to Riesz potentials and vice versa.

Following~\cite{BCD-fourier-analysis-nonlinear-pde}, one can define fractional Laplacians and Riesz potentials on homogeneous Sobolev spaces. Let us define
$$
\dot{H}^r(\R^\dimens)=\{u\in\tempered(\R^\dimens): \hat{u}\in L^1_{loc}(\R^\dimens) \ \text{and} \  \abs{\cdot}^r\hat{u}\in L^2(\R^\dimens)\}
$$
and equip it with the norm
$$
\aabs{u}_{\dot{H}^r(\R^\dimens)}=\bigg(\int_{\R^\dimens}\abs{\xi}^{2r}\abs{\hat{u}(\xi)}^2\der\xi\bigg)^{1/2}.
$$
The norm $\aabs{u}_{\dot{H}^r(\R^\dimens)}$ is homogeneous with respect to scaling $\xi\rightarrow\lambda\xi$ in contrast to the norm $\aabs{u}_{H^r(\R^\dimens)}$. We have the inclusions $\dot{H}^r(\R^\dimens)\subsetneq H^r(\R^\dimens)$ for $r<0$ and $H^r(\R^\dimens)\subsetneq \dot{H}^r(\R^\dimens)$ for $r>0$. If $r<\dimens/2$, then $\dot{H}^r(\R^\dimens)$ is a Hilbert space and $\cschwartz(\R^\dimens)$ is dense in $\dot{H}^r(\R^\dimens)$. Let $s\geq 0$ and define $\fraclaplace\varphi=\ifourier(\abs{\cdot}^{2s}\hat{\varphi})$ for $\varphi\in\cschwartz(\R^\dimens)$. Then $\fraclaplace\colon\cschwartz(\R^\dimens)\rightarrow\dot{H}^{r-2s}(\R^\dimens)$ is an isometry with respect to the norm $\aabs{\cdot}_{\dot{H}^r(\R^\dimens)}$ and by density can be extended to a continuous map $\fraclaplace\colon\dot{H}^r(\R^\dimens)\rightarrow\dot{H}^{r-2s}(\R^\dimens)$ when $r<\dimens/2$. Similarly one obtains that $\riesz\colon\dot{H}^r(\R^\dimens)\rightarrow\dot{H}^{r+n-\alpha}(\R^\dimens)$ is a continuous map for $r<\alpha-\dimens/2$ and corresponds to fractional Laplacians with negative powers, up to a constant factor.

The fractional Laplacian can also be defined on Bessel potential spaces. Let $1\leq p<\infty$. We define 
$$
H^{r, p}(\R^\dimens)=\{u\in\tempered(\R^\dimens): \ifourier(\langle\cdot\rangle^r\hat{u})\in L^p(\R^\dimens)\}
$$
and equip it with the norm
$$
\aabs{u}_{H^{r,p}(\R^\dimens)}=\aabs{\ifourier(\langle\cdot\rangle^r\hat{u})}_{L^p(\R^\dimens)}.
$$
It follows that $H^{r, p}(\R^\dimens)$ is a Banach space and $\schwartz(\R^\dimens)$ is dense in $H^{r, p}(\R^\dimens)$ for all $r\in\R$. By the Mikhlin multiplier theorem, one obtains that the operator $\fraclaplace\colon H^{r, p}(\R^\dimens)\rightarrow H^{r-2s, p}(\R^\dimens)$ is continuous for $s\geq 0$ and $1<p<\infty$. The fractional Laplacian is also defined in the space $H^{r, 1}(\R^\dimens)$ since $H^{r, 1}(\R^\dimens)\hookrightarrow H^{\frac{2r-n-\epsilon}{2}}(\R^\dimens)$ for any $\epsilon>0$ by the continuity of the Fourier transform $\fourier\colon L^1(\R^\dimens)\rightarrow L^{\infty}(\R^\dimens)$.

One can define fractional Laplacians on more general spaces. It follows that if $s\in (-\dimens/2, 1]$, then $\fraclaplace\colon\schwartz(\R^\dimens)\rightarrow\mathscr{S}_s(\R^\dimens)$ is continuous where $\mathscr{S}_s(\R^\dimens)$ is the set
$$
\mathscr{S}_s(\R^\dimens)=\{\varphi\in C^{\infty}(\R^\dimens):\langle\cdot\rangle^{\dimens+2s}\partial^{\beta}\varphi\in L^{\infty}(\R^\dimens) \ \text{for all} \ \beta\in\N_0^{\dimens}\}
$$
equipped with the topology induced by the seminorms $\aabs{\langle\cdot\rangle^{n+2s}\partial^{\beta}\varphi}_{L^{\infty}(\R^\dimens)}$. One can then extend~$\fraclaplace$ by duality to a continuous map $\fraclaplace\colon(\mathscr{S}_s(\R^\dimens))^*\rightarrow\tempered(\R^\dimens)$. See~\cite{GSU-calderon-problem-fractional-schrodinger,SI-regularity-obstacle-problem} for more details and a characterization of the dual $(\mathscr{S}_s(\R^\dimens))^*$.

\subsection{Trace spaces and singular potentials}
Let $U, \, F\subset\R^\dimens$ be an open and a closed set. We define the following Sobolev spaces
\begin{align*}
H^r(U)&=\{u|_U: u\in H^r(\R^\dimens)\} \\
\widetilde{H}^r(U)&= \ \text{closure of} \ C_c^{\infty}(U) \ \text{in} \  H^r(\R^\dimens) \\
H_0^r(U)&= \ \text{closure of} \ C_c^{\infty}(U) \ \text{in} \ H^r(U) \\
H_F^r(\R^\dimens) &=\{u\in H^r(\R^\dimens): \spt(u)\subset F\}.
\end{align*}
It is obvious that $\widetilde{H}^r(U)\subset H^r_{\overline{U}}(\R^\dimens)$ and $\widetilde{H}^r(U)\subset H_0^r(U)$. In nonlocal problems, we impose exterior values for the equation instead of boundary values. Therefore exterior values are considered to be the same if their difference is in the space $\widetilde{H}^r(U)$. For example, in equation \eqref{eq:fractionalschrodingerequation} the condition $u|_{\Omega_e}=f$ means that $u-f\in\widetilde{H}^s(\Omega)$, i.e. $u$ and $f$ are equal outside $\overline{\Omega}$, where~$\Omega$ is bounded open set. This motivates the definition of the abstract trace space $X=H^r(\R^\dimens)/\widetilde{H}^r(\Omega)$ which identifies functions in $\Omega_e$. If~$\Omega$ is a Lipschitz domain, then we have $H_0^r(\Omega)=H
^r_{\overline{\Omega}}(\R^\dimens)$ when $r>-1/2$, $r\notin \{1/2, 3/2, \dotso\}$,  $\widetilde{H}^r(\Omega)=H_{\overline{\Omega}}^r(\R^\dimens)$, $X=H^r(\Omega_e)$ and $X^*=H^{-r}_{\overline{\Omega}_e}(\R^\dimens)$. Thus for more regular domains it could be more convenient to work with the spaces~$H_{\overline{\Omega}}^r(\R^\dimens)$, but in this article we do not assume any regularity of the set~$\Omega$. For more theory of Sobolev spaces on (non-Lipschitz) domains and their properties, see \cite{CWHM-sobolev-spaces-on-non-lipchtiz-domains, ML-strongly-elliptic-systems}. 

We also use some properties of singular potentials which were introduced in \cite{RS-fractional-calderon-low-regularity-stability}. Let $t\geq 0$ and define $Z^{-t}(U)$ as a subspace of distributions $\distr(U)$ equipped with the norm
$$
\aabs{f}_{Z^{-t}(U)}=\sup\{\abs{\ip{f}{u_1u_2}_U}: u_i\in C_c^{\infty}(U), \ \aabs{u_i}_{H^t(\R^\dimens)}=1\}\;,
$$                              
where $\ip{\cdot}{\cdot}_{U}$ is the dual pairing. We denote by $Z_0^{-t}(U)$ the closure of $C_c^{\infty}(U)$ in $Z^{-t}(U)$.
Elements in $Z^{-t}(\R^\dimens)$ can be seen as multipliers: every $f\in Z^{-t}(\R^\dimens)$ induces a map $m_f\colon H^t(\R^\dimens)\rightarrow H^{-t}(\R^\dimens)$ defined as $\ip{m_f(u)}{ v}_{\R^\dimens}=\ip{f}{uv}_{\R^\dimens}$. Also $\abs{\ip{f}{uv}_{\R^\dimens}}\leq\aabs{f}_{Z^{-t}(\R^\dimens)}\aabs{u}_{H^t(\R^\dimens)}\aabs{v}_{H^{t}(\R^\dimens)}$, and this inequality can be seen as a motivation for the definition of the space $Z^{-t}(\R^\dimens)$. Clearly we have $Z_0^{-t}(\R^\dimens)\subset Z^{-t}(\R^\dimens)$. If $U$ is bounded, then $L^{\frac{\dimens}{2t}}(U)\subset Z_0^{-t}(\R^\dimens)$ for $0<t<\dimens/2$ and $L^{\infty}(U)\subset Z^{-t}_0(\R^\dimens)$ in the sense of zero extensions. Further, it holds that $L^p(U)\subset Z_0^{-t}(\R^\dimens)$ when $p>\max\{1, \dimens/2t\}$ (see section \ref{sec:magneticschrodinger}). We will only need these basic inclusions. For a more detailed treatment of the space of singular potentials~$Z^{-t}(U)$, see~\cite{MS-theory-of-sobolev-multipliers, RS-fractional-calderon-low-regularity-stability}.

\section{Unique continuation property and Poincar\'e inequality}
\subsection{Unique continuation results}
\label{subsec:uniquecontinuationresults}
In this section, we prove theorem \ref{thm:uniquecontinuationoffractionallaplacian} and give several other unique continuation results for fractional Laplacians and Riesz potentials in inhomogeneous and homogeneous Sobolev spaces. Even though we do not need all the results to solve the inverse problems considered in this article, we still state those variants since they are not given in earlier literature to the best of our knowledge. The strategy to prove results in this chapter is straightforward: if something is true for $\fraclaplace$ when $s\in (0, 1)$, then by the splitting $\fraclaplace=(-\Delta)^k(-\Delta)^{s-k}$ it should also be true for all powers~$s$ whenever the operations and claims are meaningful.

First we need a basic lemma for polyharmonic distributions, i.e. distributions which satisfy $(-\Delta)^k g=0$ for some integer $k \in \N$. We sketch the proof since it reflects the method of reduction we repeatedly use in this section.

\begin{lemma}
\label{lemma:polyharmoniccontinuation}
Let $V\subset \R^\dimens$ be any nonempty open set. If $g\in\distr(\R^\dimens)$ satisfies $(-\Delta)^kg=0$ and $g|_V=0$ for some $k\in\N$, then $g=0$.
\end{lemma}

\begin{proof}
The proof is by induction. The case $k=1$ is true since harmonic distributions are harmonic functions and therefore analytic \cite{MI:distribution-theory}. Assume that the lemma holds for some $k=m\in\N$. If $(-\Delta)^{m+1}g=0$ and $g|_V=0$, then $(-\Delta)^{m}((-\Delta)g)=0$ and $(-\Delta)g|_V=0$ since $(-\Delta)$ is a local operator. The induction assumption implies $(-\Delta)g=0$, and since also $g|_V=0$, we obtain $g=0$ by harmonicity. This implies the claim. Alternatively one could use the fact that polyharmonic distributions are analytic \cite[Theorem 7.30]{MI:distribution-theory}.
\end{proof}

Now we can prove theorem \ref{thm:uniquecontinuationoffractionallaplacian}. The idea is to reduce the general case back to the one where $s\in (0, 1)$ and use the UCP proved in \cite{GSU-calderon-problem-fractional-schrodinger}. Note that the corresponding UCP cannot hold for local operators such as $(-\Delta)^k$ when $k\in\N$. Therefore we have to assume that $s\in\R\setminus\Z$. For the proof of the case $s\in (0, 1)$, see \cite[Theorem 1.2]{GSU-calderon-problem-fractional-schrodinger}.

\begin{proof}[Proof of theorem \ref{thm:uniquecontinuationoffractionallaplacian}.]
Because of our assumptions for~$u$, the fractional Laplacian $\fraclaplace u$ for $s\in (-\dimens/2, \infty)\setminus\Z$ is well-defined, see section \ref{sec:sobolevspacesandfractionallaplacian}. Assume that $k-1<s<k$ for some $k\in\N$. Now we can split $(-\Delta)^s u=(-\Delta)^{s-(k-1)}((-\Delta)^{k-1}u)$ where $s-(k-1)\in (0, 1)$. Since the operator $(-\Delta)^{k-1}$ is local, we obtain $(-\Delta)^{s-(k-1)}((-\Delta)^{k-1}u)|_V=0$ and $(-\Delta)^{k-1}u|_V=0$ where $(-\Delta)^{k-1}u\in H^{r-2(k-1)}(\R^\dimens)$. By the UCP of $(-\Delta)^{s-(k-1)}$, we have $(-\Delta)^{k-1}u=0$. Since $u$ is polyharmonic and $u|_V=0$, lemma \ref{lemma:polyharmoniccontinuation} implies $u=0$.

If $-\dimens/2<s<0$, $s\not\in\Z$, choose $k\in\N$ such that $k+s>0$. Then by the locality of $(-\Delta)^k$ we obtain $(-\Delta)^{k+s}u|_V=0$ and $u|_V=0$. The first part of the proof implies the claim.
\end{proof}

Note that theorem \ref{thm:uniquecontinuationoffractionallaplacian} implies UCP for equations of the type $\fraclaplace u+ Lu=0$ where $L$ is any local operator. Especially, this holds if $L=P(x, D)$ where
$$
P(x, D)=\sum_{|\alpha|\leq m}a_{\alpha}(x)D^{\alpha}
$$
is a differential operator of order $m$.

The following unique continuation result of Riesz potentials was presented in \cite{IM-unique-continuation-riesz-potential}. We use it to show uniqueness for partial data problems of the $d$-plane transform in section~\ref{sec:applicationstointegralgeometry}. We recall the short proof since it relies on the UCP of the fractional Laplacian.

\begin{corollary}
\label{cor:uniquecontinuationofriespotential}
Let $\alpha\in\R$ such that $0<\alpha<\dimens$ and $(\alpha-\dimens)/2\in\R\setminus\Z$. Let $f\in \rapidly(\R^\dimens)$ and $V\subset\R^\dimens$ some nonempty open set. If $\riesz f|_V=0$ and $f|_V=0$, then $f=0$.
\end{corollary}

\begin{proof}
Recall that $f\in H^r(\R^\dimens)$ for some $r\in\R$. We can write $\riesz f=(-\Delta)^{-s} f$ where $s=(\dimens-\alpha)/2$. Choose $k\in\N$ such that $k-s>0$. By locality of $(-\Delta)^k$ we obtain the conditions $(-\Delta)^{k-s}f|_V=0$ and $f|_V=0$.
Theorem~\ref{thm:uniquecontinuationoffractionallaplacian} implies $f=0$.
\end{proof}

It is also independently proved in \cite{IM-unique-continuation-riesz-potential}, without using the UCP of~$\fraclaplace$, that if $f\in\cdistr(\R^\dimens)$, then one can replace the condition $\riesz f|_V=0$ by the requirement $\partial^{\beta}(\riesz f)(x_0)=0$ for some $x_0\in V$ and all $\beta\in\N_0^\dimens$. In fact, this can be used to prove a slightly stronger result for $\fraclaplace$ in the case of compact support.

\begin{corollary}
\label{cor:stronguniquecontinuation}
Let $u\in\cdistr(\R^\dimens)$, $V\subset\R^\dimens$ some nonempty open set and $s\in (-\dimens/2, \infty)\setminus\Z$. If $\partial^{\beta}(\fraclaplace u)(x_0)=0$ and $u|_V=0$ for some $x_0\in V$ and all $\beta\in\N_0^\dimens$, then $u=0$.
\end{corollary}
\begin{proof}
Let $k-1<s<k$ where $k\in\N$. Now $\fraclaplace=(-\Delta)^k(-\Delta)^{s-k}=(-\Delta)^k\riesz$ where $\alpha=\dimens+2s-2k\in (n-2, n)$. Furthermore, $\partial^{\beta}\fraclaplace u=\partial^{\beta}\riesz(-\Delta)^k u$ since the Riesz potential commutes with derivatives. By the locality of $(-\Delta)^k$ we obtain the conditions $\partial^{\beta}(\riesz(-\Delta)^k u)(x_0)=0$ and $(-\Delta)^k u|_V=0$ where $(-\Delta)^k u\in\cdistr(\R^\dimens)$. By \cite[Theorem 1.1]{IM-unique-continuation-riesz-potential}, we must have $(-\Delta)^k u=0$. Since also $u|_V=0$, we obtain $u=0$ by lemma \ref{lemma:polyharmoniccontinuation}. 

Let then $s\in (-\dimens/2, 0)$, $s\not\in\Z$, and pick $k\in\N$ such that $s+k>0$. All the derivatives $\partial^{\beta}(\fraclaplace u)(x_0)$ vanish, and hence $((-\Delta)^k\partial^{\beta})(\fraclaplace u)(x_0)=0$. Now $((-\Delta)^k\partial^{\beta})(\fraclaplace u)=\partial^{\beta}((-\Delta)^{s+k} u)$ and we get the conditions $\partial^{\beta}((-\Delta)^{s+k} u)(x_0)=0$ and $u|_V=0$. The first part of the proof gives the claim.
\end{proof}

The UCP of $\fraclaplace$ also extends to homogeneous Sobolev spaces. The following result is a simple consequence of theorem \ref{thm:uniquecontinuationoffractionallaplacian}. See \cite{FF-unique-continuation-fractional-ellliptic-equations, FF-unique-continuation-higher-laplacian} for related results (strong UCP and measurable UCP in some special cases).

\begin{corollary}
Let $s\in\R^+\setminus\Z$ and $u\in \dot{H}^{r}(\R^\dimens)$, $r<\dimens/2$. If $(-\Delta)^s u|_V=0$ and $u|_V=0$ for some nonempty open set $V\subset\R^\dimens$, then $u=0$.
\end{corollary}

\begin{proof}
If $r<0$, then $u\in H^r(\R^\dimens)$ and the claim follows from theorem \ref{thm:uniquecontinuationoffractionallaplacian}. Let $r>0$ and choose $k\in\N$ such that $r-2k<0$. Now $(-\Delta)^k\fraclaplace=\fraclaplace (-\Delta)^k$ holds in $\cschwartz(\R^\dimens)$ so by the density of $\cschwartz(\R^\dimens)$ and the locality of $(-\Delta)^k$ we obtain $\fraclaplace((-\Delta)^k u)|_V=0$ and $(-\Delta)^k u|_V=0$, where $(-\Delta)^k u\in\dot{H}^{r-2k}(\R^\dimens)\subset H^{r-2k}(\R^\dimens)$. Hence $(-\Delta)^k u=0$ by theorem \ref{thm:uniquecontinuationoffractionallaplacian} and since $u|_V=0$ we obtain $u=0$ by lemma \ref{lemma:polyharmoniccontinuation}. 
\end{proof}

Since $(-\Delta)^k(-\Delta)^{-s}=(-\Delta)^{k-s}$ also holds by the density of $\cschwartz(\R^\dimens)$, one can reduce the case of negative exponents to the case of positive exponents. Thus one obtains the corresponding UCP for the Riesz potential~$\riesz$ in $\dot{H}^r(\R^\dimens)$ where $r<\alpha-\dimens/2$. By the Sobolev embedding theorem we obtain the following unique continuation result for Bessel potential spaces when $1\leq p\leq 2$.

\begin{corollary}
\label{cor:besselucp}
Let $s\in\R^+\setminus\Z$, $1\leq p\leq 2$ and $u\in H^{r, p}(\R^\dimens)$, $r\in\R$. If $(-\Delta)^s u|_V=0$ and $u|_V=0$ for some nonempty open set $V\subset\R^\dimens$, then $u=0$.
\end{corollary}

\begin{proof}
If $p=1$, then $\ifourier(\langle\cdot\rangle^r \hat{u})\in L^1(\R^\dimens)$ which implies $\langle\cdot\rangle^r \hat{u}\in L^{\infty}(\R^\dimens)$ since $\fourier\colon L^1(\R^\dimens)\rightarrow L^{\infty}(\R^\dimens)$ is continuous. Hence $u\in H^t(\R^\dimens)$ for some $t\in\R$ and the claim follows from theorem~\ref{thm:uniquecontinuationoffractionallaplacian}. Let then $1<p\leq 2$. By the Sobolev embedding theorem (see e.g. \cite[Theorem 6.5.1]{BL-interpolation-spaces}) $H^{r, p}(\R^\dimens)\hookrightarrow H^{r_1, p_1}(\R^\dimens)$ when $r_1\leq r$, $1< p\leq p_1<\infty$ and
$$
r-\frac{\dimens}{p}= r_1-\frac{\dimens}{p_1}.
$$
Choose $p_1=2$. Then for any $1<p\leq 2$ the previous equality holds when
$$
r_1=\frac{2rp+\dimens(p-2)}{2p}\leq r.
$$
Hence $u\in H^{r_1, 2}(\R^\dimens)=H^{r_1}(\R^\dimens)$ and by theorem \ref{thm:uniquecontinuationoffractionallaplacian} we obtain $u=0$.
\end{proof}

For higher exponents $p$, we can prove the following version of unique continuation considering the Fourier transform.

\begin{corollary}
Let $r\geq 0$, $2\leq p<\infty$ and $s\in\R^+\setminus\Z$. Let $u\in H^{r, p}(\R^\dimens)$ and $V\subset\R^\dimens$ some nonempty open set. If $\fraclaplace\hat{u}|_V=0$ and $\hat{u}|_V=0$, then $u=0$.
\end{corollary}

\begin{proof}
By the inclusion $H^{r, p}(\R^\dimens)\hookrightarrow L^p(\R^\dimens)$ for $r\geq 0$, we can assume $u\in L^p(\R^\dimens)$. If $p=2$, then $\hat{u}\in L^2(\R^\dimens)$. By theorem \ref{thm:uniquecontinuationoffractionallaplacian}, we obtain $\hat{u}=0$ and hence $u=0$. If $2<p<\infty$, then we have that $\hat{u}\in H^{-t}(\R^\dimens)$ where $t>n(1/2-1/p)$ by \cite[Theorem 7.9.3]{HO:analysis-of-pdos}. Again we obtain $\hat{u}=0$ by theorem \ref{thm:uniquecontinuationoffractionallaplacian} and eventually $u=0$.
\end{proof}

Note that if $u$ has compact support, then by the Paley-Wiener theorem the condition $\hat{u}|_V=0$ already implies that $u=0$. 

\subsection{The fractional Poincar\'e inequality}
\label{subsec:poinacareinequality}
This subsection is dedicated to the proofs of a fractional Poincar\'e inequality. It serves the goal of estimating the $L^2$-norm of $u\in \widetilde{H}^s(\Omega)$ with that of its fractional Laplacian $(-\Delta)^{s/2} u$. We give five possible proofs for the fractional Poincar\'e inequality. We believe that giving several proofs will be helpful in subsequent works. This also illustrates some connections between methods which might have been unnoticed before.

The first proof is the most direct one and is based on splitting of frequencies on the Fourier side. The second proof utilizes several estimates (most importantly Hardy-Littlewood-Sobolev inequalities). This proof is motivated by the approach taken in \cite{GSU-calderon-problem-fractional-schrodinger}. Third proof uses a reduction argument to extend the inequality proved in \cite{CLR18} for all powers $s\geq 0$. Fourth proof is based on interpolation of homogeneous Sobolev spaces and it also gives an explicit constant in terms of the classical Poincar\'e constant. Fifth proof uses uncertainty inequalities which are treated in~\cite{FS-the-uncertainty-principle}.

We begin our first proof by dividing the Fourier side into high and low frequencies. We only use simple estimates in the proof. In this approach we also get a control on the Poincar\'e constant. The result is basically the same as \cite[Proposition 1.55]{BCD-fourier-analysis-nonlinear-pde}.

\begin{theorem}[Poincar\'e inequality]
\label{thm:poincarechineseguy}
Let $s\geq 0$, $K\subset\R^\dimens$ compact set and $u\in H^s_K(\R^\dimens)$. There exists a constant $c=c(n, K, s)> 0$ such that
\begin{equation}
\aabs{u}_{L^2(\R^\dimens)}\leq c\aabs{(-\Delta)^{s/2}u}_{L^2(\R^\dimens)}.
\end{equation}
\end{theorem}

\begin{proof}
We divide the integration into high and low frequencies
\begin{align}
\aabs{u}_{L^2(\R^\dimens)}^2=\int_{\abs{\xi}\leq\epsilon}\abs{\hat{u}(\xi)}^2\der\xi+\int_{\abs{\xi}>\epsilon}\abs{\hat{u}(\xi)}^2\der\xi
\end{align}
where $\epsilon>0$ is determined later on. Let us analyze the first part. Since $u\in L^2(\R^\dimens)$ and has support in~$K$, H\"older's inequality implies
$$
\abs{\hat{u}(\xi)}\leq\aabs{u}_{L^1(\R^\dimens)}\leq\abs{K}^{1/2}\aabs{u}_{L^2(\R^\dimens)}.
$$
Thus we have
\begin{align*}
\int_{\abs{\xi}\leq\epsilon}\abs{\hat{u}(\xi)}^2\der\xi\leq\int_{\abs{\xi}\leq\epsilon}\abs{K}\aabs{u}_{L^2(\R^\dimens)}^2\der\xi =\epsilon^\dimens\abs{K}\abs{B(0, 1)}\aabs{u}_{L^2(\R^\dimens)}^2
\end{align*}
where $\abs{K}$ and $\abs{B(0, 1)}$ are the measures of~$K$ and the unit ball~$B(0, 1)$.
For high frequencies we can do the following trick
\begin{align*}
\int_{\abs{\xi}>\epsilon}\abs{\hat{u}(\xi)}^2\der\xi=\int_{\abs{\xi}>\epsilon}\frac{\abs{\xi}^{2s}\abs{\hat{u}(\xi)}^2}{\abs{\xi}^{2s}}\der\xi\leq\epsilon^{-2s}\aabs{(-\Delta)^{s/2}u}_{L^2(\R^\dimens)}^2.
\end{align*}
Now choose $0<\epsilon<(\abs{K}\abs{B(0, 1)})^{-1/n}$. Then one obtains the inequality 
$$
\aabs{u}_{L^2(\R^\dimens)}\leq\frac{\epsilon^{-s}}{\sqrt{1-\epsilon^\dimens\abs{K}\abs{B(0, 1)}}}\aabs{(-\Delta)^{s/2}u}_{L^2(\R^\dimens)}. \qedhere
$$
\end{proof}

\begin{remark}
Choosing $\epsilon=(2\abs{K}\abs{B(0, 1)})^{-1/\dimens}$ one obtains the following inequality in theorem~\ref{thm:poincarechineseguy}
$$
\aabs{u}_{L^2(\R^\dimens)}\leq\sqrt{2}(2\abs{K}\abs{B(0, 1)})^{s/\dimens}\aabs{(-\Delta)^{s/2}u}_{L^2(\R^\dimens)}.
$$
If $K$ is a ball, the constant in this inequality has the same scaling with respect to the diameter of the set as in theorem~\ref{thm:poincareinterpolation}, i.e. $c\approx(\text{diam}(K))^s$. Further, one can use similar method of proof as in theorem~\ref{thm:poincarechineseguy} to show Poincar\'e inequalities for more general pseudodifferential operators on certain manifolds. See~\cite{XI-note-on-fractional-poincare} for details.
\end{remark}

Provided we have the Poincar\'e inequality, we can prove the generalized version of it. See also~\cite[Corollary 1.56]{BCD-fourier-analysis-nonlinear-pde} for a similar inequality when $K$ is a ball. In that case one can take $\tilde{c}\approx(\text{diam}(K))^{s-t}$. The cases $s\geq t\geq 1$ and $s\geq 1\geq t\geq 0$ are also proved for $u\in\widetilde{H}^s(\Omega)$ in theorem~\ref{thm:poincareinterpolation}.

\begin{proof}[Proof of theorem \ref{thm:generalpoincare}]
Since $s\geq t\geq 0$ we have the continuous embeddings $ H^t(\R^\dimens)\hookrightarrow\dot{H}^t(\R^\dimens)$ and $H^s(\R^\dimens)\hookrightarrow H^t(\R^\dimens)$. Using the Poincar\'e inequality in theorem~\ref{thm:poincarechineseguy} we obtain
\begin{align*}
\aabs{(-\Delta)^{t/2}u}_{L^2(\R^\dimens)}=\aabs{u}_{\dot{H}^t(\R^\dimens)}&\leq\aabs{u}_{H^t(\R^\dimens)}\leq\aabs{u}_{H^s(\R^\dimens)}\leq 2^{\frac{s+1}{2}}\Big(\aabs{u}_{L^2(\R^\dimens)}+\aabs{u}_{\dot{H}^s(\R^\dimens)}\Big) \\
&\leq 2^{\frac{s+1}{2}}\Big(c\aabs{(-\Delta)^{s/2}u}_{L^2(\R^\dimens)}+\aabs{(-\Delta)^{s/2}u}_{L^2(\R^\dimens)}\Big) \\
&=\tilde{c}\aabs{(-\Delta)^{s/2}u}_{L^2(\R^\dimens)}
\end{align*}
where the constants $c$ and $\tilde{c}$ do not depend on $u$. In the fourth step we used the elementary inequality $(a+b)^r\leq 2^r (a^r+b^r)$ for $a, b\geq 0$. This concludes the proof.
\end{proof}

We then start preparation for our second proof by stating some known lemmas:\begin{itemize}
    \item lemma \ref{continuity-of-riesz} is the continuity of Riesz potentials, 
    \item lemma \ref{boundedness-of-the-inverse} is the $L^2$ boundedness of inverse of elliptic second order operators,
    \item lemma \ref{from-hormander} is a convolution $L^p$ estimate from below by an inhomogeneous Hölder norm, \item lemma \ref{from-japanese-guy} is a specific form of the Poincar\'e inequality for fractional Laplacians, and \item lemma \ref{grad-commutes-with-laplacian} is a simple commutation property for the gradient and a Fourier multiplier.
    \end{itemize}

\begin{lemma}[Theorem 4.5.3 in \cite{HO:analysis-of-pdos}]\label{continuity-of-riesz}
Let $t\geq 0$, $1<p<\infty$ be such that $n>tp$, and define $q = {\frac{np}{n-tp}}$. Then the Riesz potential $(-\Delta)^{-t/2}: L^p(\R^\dimens) \rightarrow L^q(\R^\dimens)$ is continuous.
\end{lemma}

\begin{lemma}[Section 6 in \cite{Eva10}]\label{boundedness-of-the-inverse}
Let $\Omega\subset \R^n$ be a bounded domain and $f\in L^2(\Omega)$. If $w\in H^1_0(\Omega)$ is the unique solution of the problem \begin{align}\label{eq:Laplace-bndry-vprob}\left\{\begin{array}{rl}
        (-\Delta) w &= f \;\;\text{in}\;  \Omega\\
        w|_{\partial\Omega} &=0
        \end{array}\right.\;,\end{align}
\noindent then there exists a constant $C=C(\Omega)$ such that 
\begin{equation}\label{eq:boundedness-of-inverse}\|w\|_{L^2(\Omega)}\leq C\|f\|_{L^2(\Omega)}\;.\end{equation}
\end{lemma}

\begin{lemma}[Theorem 4.5.10 in \cite{HO:analysis-of-pdos}]\label{from-hormander}
Let $\psi\in C^1(\R^n\setminus \{0\})$ be homogeneous of degree $-n/a$, $p\in[1,\infty]$ and $\gamma = n(1-1/a-1/p)$ be such that $\gamma\in (0,1)$. Then if $v\in L^p(\mathbb R^n)\cap\cdistr(\R^\dimens)$ we have
\begin{equation}\label{eq:hormander-estimate}
\sup_{x\neq y} \left\{ \frac{|(\psi\ast v)(x)-(\psi\ast v)(y)|}{|x-y|^\gamma} \right\} \leq C\|v\|_{L^p(\mathbb R^n)},
\end{equation}
where $C$ does not depend on $w$.
\end{lemma}

\begin{lemma}[Formula (1.3) in \cite{Oz92}]\label{from-japanese-guy}

Let $1<p\leq q < \infty$ and $f\in W^{n/p, p}(\mathbb R^n)$. There is a constant $C=C(n,p)$ such that 
\begin{equation}\label{eq:japanese-estimate}
    \|f\|_{L^q(\mathbb R^n)} \leq C q^{1-1/p} \|(-\Delta)^{n/2p}f\|^{1-p/q}_{L^p(\mathbb R^n)}  \|f\|^{p/q}_{L^p(\mathbb R^n)}\;.
\end{equation}
\end{lemma}
\noindent This estimate is proved using sharp Hardy-Littlewood-Sobolev inequalities.

\begin{lemma}\label{grad-commutes-with-laplacian}

Let $t\geq 0$ and $f\in H^{1+2t}(\R^n)$. Then $[\nabla,(-\Delta)^t]f=0$, that is, the gradient and the fractional Laplacian of exponent $t$ commute.
\end{lemma}
\begin{proof}
The proof is just a trivial computation with Fourier symbols:
\begin{equation*} \fourier(\nabla(-\Delta)^tf) = i\xi |\xi|^{2t} \hat f(\xi) = |\xi|^{2t} i\xi \hat f(\xi) = \fourier((-\Delta)^t(\nabla f))\;.\qedhere\end{equation*}
\end{proof}

We are now ready to state and prove the fractional Poincar\'e inequality.

\begin{theorem}[Poincar\'e inequality]
\label{theorem:poincare}
Let $\Omega\subset\R^\dimens$ be a bounded domain, $s\in [0,\infty)$ and $u\in \widetilde{H}^s(\Omega)$. There exists a constant $c=c(n, \Omega, s)$ such that \begin{equation}\label{eq:poincare-inequality} \|u\|_{L^2(\R^n)} \leq c \|(-\Delta)^{s/2}u\|_{L^2(\R^n)}\;. \end{equation}
\end{theorem}

\begin{proof}
In the inequalities the constants (usually denoted by $c$, $C$, etc.) do not depend on the function which is being estimated and can change from line to line. We let the symbol $s'=s-\floor{s}$ indicate the fractional part of the exponent $s$, with the convention that $s'\in [0,1)$. First observe that by using lemma \ref{continuity-of-riesz} with $p=2$ and H\"older's inequality we get the following useful estimate \begin{equation}\label{eq:riesz-continuity} \|u\|_{L^2(\mathbb R^n)} \leq C_\Omega\|u\|_{L^q(\mathbb R^n)} \leq c \|(-\Delta)^{t/2}u\|_{L^2(\mathbb R^n)} \;\end{equation}
when $u\in\widetilde{H}^t(\Omega)$ where $q$ and $t$ are as in lemma \ref{continuity-of-riesz}. Our proof is divided in several cases.
\vspace{3mm}

\pmb{Case 1: $\quad \floor{s}\in 2\mathbb Z$, $s'=0$}.

\noindent Recall that $\widetilde{H}^{2h}(\Omega)\subset H^{2h}_0(\Omega)$. We show that if $u\in H_0^{2h}(\Omega)$ and $h\in\mathbb N$ then there exists a constant $c=c(n, \Omega, h)$ such that \begin{equation}\label{eq:iteration-estimate} \|(-\Delta)^h u\|_{L^2(\R^n)} \geq c \|u\|_{L^2(\R^n)}\;.\end{equation}
\noindent The estimate \eqref{eq:iteration-estimate} holds trivially if $h=0$, while if $h=1$ then \eqref{eq:iteration-estimate} follows from the boundedness of the inverse lemma \ref{boundedness-of-the-inverse}. Assume now that $h\geq 2$, and by induction that \eqref{eq:iteration-estimate} holds for $h-1$. Then $(-\Delta) u \in  H^{2h-2}_0(\Omega)$, so we can apply \eqref{eq:iteration-estimate} and \eqref{eq:boundedness-of-inverse} to get $$ \|(-\Delta)^h u\|_{L^2(\R^n)} = \|(-\Delta)^{h-1} (-\Delta u)\|_{L^2(\R^n)}\geq c \|(-\Delta) u\|_{L^2(\R^n)} \geq c^{\prime} \|u\|_{L^2(\R^n)}\;. $$

\noindent In the next steps we consider $s\not\in\N$.
\vspace{3mm}

\pmb{Case 2: $\quad\floor{s}\in 2\mathbb Z$, $s'\in(0,1/2) \quad$} or \pmb{$\quad\floor{s}\in 2\mathbb Z$, $s'\in[1/2,1)$, $n\geq 2$}. 

\noindent Now it holds that $n>2s'$, and there exists $k\in\mathbb N$ such that $s\in (2k, 2k+1)$ and we can write $(-\Delta)^{s/2} u = (-\Delta)^{s'/2} (-\Delta)^k u$. Since $(-\Delta)^k u \in \widetilde{H}^{s-2k}(\Omega) = \widetilde{H}^{s'}(\Omega)$, we can apply formula \eqref{eq:riesz-continuity} \begin{equation}\label{eq:first-half-estimate} \|(-\Delta)^ku\|_{L^2(\mathbb R^n)} \leq c \|(-\Delta)^{s/2}u\|_{L^2(\mathbb R^n)} \;.\end{equation} 

\noindent Since $ u\in H^s_0(\Omega) \subset H^{2k}_0(\Omega) $, we can get the result using formula \eqref{eq:iteration-estimate}.  
\vspace{3mm}

\pmb{Case 3: $\quad\floor{s}\in 2\mathbb Z$, $s'\in(1/2,1)$, $n=1$}.

\noindent As in the second case, there exists $k\in\mathbb N$ such that $s\in (2k, 2k+1)$ and we can write $(-\Delta)^{s/2} u = (-\Delta)^{s'/2} (-\Delta)^k u$. However, since now $n<2s'$, we cannot directly use formula \eqref{eq:riesz-continuity}. 

\noindent Assume first that $w\in C^\infty_c(\Omega)$. Then we can take $y_0\in\Omega$ such that $w(y_0)=0$ and $x_0\in\Omega$ such that $w(x_0)=\|w\|_{L^\infty(\Omega)}$. With these choices and for any $\gamma > 0$ we can write
\begin{equation}\label{eq:L2-to-sup-estimate}
\|w\|_{L^2(\mathbb R^n)} \leq C\|w\|_{L^\infty(\Omega)} \leq C\frac{w(x_0)-w(y_0)}{|x_0-y_0|^\gamma}\;.
\end{equation}
\noindent We now let $\gamma=s'-n/2=s'-1/2\in(0,1/2)$, and define $\psi=|x|^{s'-1}$, $v=(-\Delta)^{s'/2}w$. By the mapping properties of the fractional Laplacian and the Mikhlin theorem, we can observe that $v\in L^p(\mathbb R)$ for all $1<p<\infty$ (see \cite[Theorem 7.2]{AB-psidos-and-singular-integrals}). Using the continuity of the Riesz potential in lemma \ref{continuity-of-riesz}, we see that for a constant $c=c(n,s)$ the following holds almost everywhere:
\begin{equation}\label{before-lemma} 
w = (-\Delta)^{-s'/2}((-\Delta)^{s'/2}w) = (-\Delta)^{-s'/2}v = c I_{1-s'}v = c |x|^{s'-1}\ast v = c (\psi \ast v)\;. \end{equation}

Let $\chi_R$ be the characteristic function of the ball $B_R$ of radius $R>0$, and define $w_R = c(\psi\ast (\chi_R v))$, with $c$ as above. We see that
$$ w_R(x) = c  (\psi\ast(\chi_R v))(x) = c \int_{\mathbb R} \psi(x-y)\chi_R(y)v(y)dy\;, $$
\noindent and the integrand is dominated by $|\psi(x-y)v(y)|$. This is an integrable function, since
$$ \int_{\mathbb R}|\psi(x-y)v(y)|dy = \int_{\mathbb R} \psi(x-y)|v(y)|dy = I_{1-s'}(|v|)(x)\;, $$

\noindent and the Riesz potential is well defined almost everywhere on $L^p(\R)$ for any $1<p<1/s'$. Now the dominated convergence theorem gives that $w_R(x) \to w(x)$ as $R \to \infty$ for almost every fixed $x \in \R$.

Let $\epsilon > 0$ and $x_0', y_0'\in \mathbb R$ be such that $|x_0-x_0'|<\epsilon, |y_0-y_0'|<\epsilon$ and $w_R(x_0')$, $w_R(y_0')$ converge to $w(x_0')$, $w(y_0')$ as $R\rightarrow\infty$.  Applying lemma \ref{from-hormander} with $p=2$, $n=1$ and $a=1-s'$, we see that

\begin{align}\label{eq:sup-to-laplacian-estimate}
\frac{w_R(x_0')-w_R(y_0')}{|x_0-y_0|^\gamma} & \leq \sup_{x\neq y}\left\{ \frac{w_R(x)-w_R(y)}{|x-y|^\gamma} \right\} \\ & = c \sup_{x\neq y}\left\{ \frac{(\psi\ast (\chi_Rv))(x)-(\psi\ast (\chi_Rv))(y)}{|x-y|^\gamma} \right\} \\ & \leq C\|\chi_Rv\|_{L^2(\mathbb R)} \leq C\|v\|_{L^2(\mathbb R)} = C\|(-\Delta)^{s'/2}w\|_{L^2(\mathbb R)}\;.
\end{align}

We now first take the limit for $R\rightarrow\infty$ and then the one for $\epsilon \rightarrow 0$. By the smoothness of $w$, this gives

\begin{equation}\label{new-guy} \frac{w(x_0)-w(y_0)}{|x_0-y_0|^\gamma} \leq C \|(-\Delta)^{s'/2}w\|_{L^2(\mathbb R)}\;.
\end{equation}

Combining formulas \eqref{eq:L2-to-sup-estimate} and \eqref{new-guy} we get $\|w\|_{L^2(\mathbb R^n)} \leq C\|(-\Delta)^{s'/2}w\|_{L^2(\mathbb R^n)}$, and the same inequality holds for $w\in  \widetilde{H}^{s'}(\Omega)$ by density. Let now $w:=(-\Delta)^k u \in \widetilde{H}^{s-2k}(\Omega) = \widetilde{H}^{s'}(\Omega)$. The result is then obtained applying formula \eqref{eq:iteration-estimate}.
\vspace{3mm}

\pmb{Case 4: $\quad\floor{s}\in 2\mathbb Z$, $s'=1/2$, $n=1$}.

\noindent Let $w:=(-\Delta)^k u \in \widetilde{H}^{s-2k}(\Omega) = \widetilde{H}^{s'}(\Omega)$. Here we make use of formula \eqref{eq:japanese-estimate} with $p=2$, $q=3$ in order to estimate
\begin{equation}\label{eq:ozawa-estimate}
\|w\|_{L^2(\R^n)} = \|w\|^3_{L^2(\R^n)}\|w\|^{-2}_{L^2(\R^n)}\leq \|w\|^3_{L^3(\R^n)}\|w\|^{-2}_{L^2(\R^n)} \leq C \|(-\Delta)^{n/4}w\|_{L^2(\R^n)}\;.
\end{equation}
Since $n/4$ equals $s'/2$ for $n=1$, the results follows from \eqref{eq:ozawa-estimate} and \eqref{eq:iteration-estimate}. 
\vspace{3mm}

\pmb{Case 5: $\quad\floor{s}\not\in 2\mathbb Z$}.

\noindent Let $u\in C_c^{\infty}(\Omega)$. In this case $s=s'+2k+1$ for some $k\in \mathbb N$, therefore we can calculate
\begin{equation}\begin{split}\label{eq:odd-case}
\|(-\Delta)^{s/2} u\|_{L^2(\mathbb R^n)} & = \|(-\Delta)^{1/2} (-\Delta)^{(s'+2k)/2} u\|_{L^2(\mathbb R^n)} \\ & = \|\nabla (-\Delta)^{(s'+2k)/2} u\|_{L^2(\mathbb R^n)}\\ & = \| (-\Delta)^{(s'+2k)/2} \nabla u\|_{L^2(\mathbb R^n)} \\ & \geq C\| \nabla u\|_{L^2(\mathbb R^n)} \geq C\| u\|_{L^2(\mathbb R^n)}\;.
\end{split}\end{equation}

\noindent The second equality in \eqref{eq:odd-case} is just an $L^2$ property of the gradient and the $(-\Delta)^{1/2}$ operator. The third equality in \eqref{eq:odd-case} follows from lemma \ref{grad-commutes-with-laplacian}. The first inequality in \eqref{eq:odd-case} follows from the even cases, given that $\floor{s'+2k}\in 2\mathbb Z$ and $\nabla u \in \widetilde{H}^{s'+2k}(\Omega)$ componentwise. The last inequality follows from the classical Poincar\'e inequality. The rest follows by approximation.
\end{proof}

\begin{remark}
\label{rmk:poincarecaffarellisilvestre}
Third way to prove the Poincar\'e inequality is using the known result in the case $\dimens\geq 1$ and $s\in (0, 1)$ \cite[Lemma 2.2]{CLR18}. This result is proved using Caffarelli-Silvestre extension. Then one can use similar reduction argument to prove it for all $s\geq 0$ and $u\in C_c^{\infty}(\Omega)$. Namely, one shows using the classical Poincar\'e inequality that the claim holds for all $s\in [0, 2)$. The higher order fractional cases are obtained by splitting the fractional Laplacian as $\fraclaplace=(-\Delta)^k(-\Delta)^{t/2}$ where $t\in (0, 2)$. Boundedness of the inverse and the fractional Poincar\'e inequality for $t\in (0, 2)$ imply the claim for fractional exponents. Integer order exponents are obtained from the boundedness of the inverse as before. The inequality for $u\in\widetilde{H}^s(\Omega)$ follows by approximation.
\end{remark}

For the fourth proof we use the following interpolation lemma of homogeneous Sobolev spaces which is a simple consequence of H\"older's inequality, see \cite[Proposition 1.32]{BCD-fourier-analysis-nonlinear-pde}.

\begin{lemma}
\label{lemma:interpolationhomogeneoussobolev}
Let $s_0\leq r\leq s_1$ and $f\in \dot{H}^{s_0}(\R^\dimens)\cap\dot{H}^{s_1}(\R^\dimens)$. Then $f\in \dot{H}^r(\R^\dimens)$ and
$$
\aabs{f}_{\dot{H}^r(\R^\dimens)}\leq \aabs{f}^{1-\theta}_{\dot{H}^{s_0}(\R^\dimens)}\aabs{f}_{\dot{H}^{s_1}(\R^\dimens)}^{\theta}, \quad r=(1-\theta)s_0+\theta s_1.
$$
\end{lemma}

Using the interpolation lemma and the usual Poincar\'e inequality we can easily prove the following theorem. Note that we also obtain explicit constant from the proof.

\begin{theorem}[Poincar\'e inequality]
\label{thm:poincareinterpolation}
Let $s\geq t\geq 1$ or $s\geq 1\geq t\geq 0$,  $\Omega\subset\R^\dimens$ bounded open set  and $u\in\widetilde{H}^s(\Omega)$. The following inequality holds
\begin{equation}
\label{eq:poincare}
\aabs{(-\Delta)^{t/2} u}_{L^2(\R^\dimens)}\leq C^{s-t}\aabs{(-\Delta)^{s/2}u}_{L^2(\R^\dimens)}
\end{equation}
where $C=C(\dimens, \Omega)$ is the classical Poincar\'e constant.
\end{theorem}

\begin{proof}
Let $s\geq t\geq 1$ and $u\in C^{\infty}_c(\Omega)$. The usual Poincar\'e inequality can be written in terms of the homogeneous Sobolev norm as
$$
\aabs{u}_{L^2(\R^\dimens)}=\aabs{u}_{L^2(\Omega)}\leq C\aabs{\nabla u}_{L^2(\Omega)}= C\aabs{\nabla u}_{L^2(\R^\dimens)}=C\aabs{u}_{\dot{H}^1(\R^\dimens)}
$$
where $C=C(n, \Omega)$. We use the interpolation lemma \ref{lemma:interpolationhomogeneoussobolev} twice. First choose $s_0=0$, $r=1$ and $s_1=t\geq 1$. Now $\theta=1/t$ and by the classical Poincar\'e inequality we obtain
\begin{align*}
\aabs{u}_{\dot{H}^1(\R^\dimens)}\leq\aabs{u}^{1-\theta}_{L^2(\R^\dimens)}\aabs{u}^{\theta}_{\dot{H}^t(\R^\dimens)}
\leq C^{1-\theta}\aabs{u}^{1-\theta}_{\dot{H}^1(\R^\dimens)}\aabs{u}^{\theta}_{\dot{H}^t(\R^\dimens)}.
\end{align*}
From this we get the following inequality  
$$
\aabs{u}_{\dot{H}^1(\R^\dimens)}\leq C^{\frac{1-\theta}{\theta}}\aabs{u}_{\dot{H}^t(\R^\dimens)}
$$
for all $u\in C^{\infty}_c(\Omega)$. Hence
\begin{align*}
\aabs{u}_{L^2(\R^\dimens)}\leq C\aabs{u}_{\dot{H}^1(\R^\dimens)}&\leq C^t\aabs{u}_{\dot{H}^t(\R^\dimens)}.
\end{align*}
Then choose $s_0=0$, $r=t$ and $s_1=s\geq t$ in lemma \ref{lemma:interpolationhomogeneoussobolev}. Now $\theta=t/s$ and by the previous inequality
\begin{align*}
\aabs{u}_{\dot{H}^t(\R^\dimens)}\leq\aabs{u}^{1-\theta}_{L^2(\R^\dimens)}\aabs{u}_{\dot{H}^s(\R^\dimens)}^{\theta}\leq C^{t(1-\theta)}\aabs{u}_{\dot{H}^t(\R^\dimens)}^{1-\theta}\aabs{u}_{\dot{H}^s(\R^\dimens)}^{\theta}.
\end{align*}
From this we obtain 
$$
\aabs{(-\Delta)^{t/2}u}_{L^2(\R^\dimens)}\leq C^{s-t}\aabs{(-\Delta)^{s/2}u}_{L^2(\R^\dimens)}
$$
for $u\in C^{\infty}_c(\Omega)$. 

Let then $s\geq 1\geq t\geq 0$ and $u\in C_c^{\infty}(\Omega)$. First interpolate for $s\geq 1\geq t$ to obtain
$$
\aabs{u}_{\dot{H}^1(\R^\dimens)}\leq \aabs{u}^{1-\theta}_{\dot{H}^t(\R^\dimens)}\aabs{u}^{\theta}_{\dot{H}^s(\R^\dimens)}, \quad \theta=\frac{1-t}{s-t}.
$$
Second, interpolate for $1\geq t\geq 0$ and use the previous inequality and the classical Poincar\'e inequality to get
$$
\aabs{u}_{\dot{H}^t(\R^\dimens)}\leq\aabs{u}_{L^2(\R^\dimens)}^{1-\tilde{\theta}}\aabs{u}_{\dot{H}^1(\R^\dimens)}^{\tilde{\theta}}\leq C^{1-\tilde{\theta}}\aabs{u}_{\dot{H}^t(\R^\dimens)}^{1-\theta}\aabs{u}_{\dot{H}^s(\R^\dimens)}^{\theta}, \quad \tilde{\theta}=t,
$$
which eventually implies the inequality
$$
\aabs{(-\Delta)^{t/2}u}_{L^2(\R^\dimens)}=\aabs{u}_{\dot{H}^t(\R^\dimens)}\leq C^{s-t}\aabs{u}_{\dot{H}^s(\R^\dimens)}=C^{s-t}\aabs{(-\Delta)^{s/2}u}_{L^2(\R^\dimens)}
$$
for all $u\in C^{\infty}_c(\Omega)$.

Then let $u\in\widetilde{H}^s(\Omega)$. By definition there is a sequence $\varphi_k\in C_c^{\infty}(\Omega)$ such that
$$
\varphi_k\rightarrow u \quad \text{in} \ H^s(\R^\dimens).
$$
The continuity of $(-\Delta)^{t/2}$ implies that 
$$
(-\Delta)^{t/2}\varphi_k\rightarrow (-\Delta)^{t/2}u \quad \text{in} \ H^{s-t}(\R^\dimens).
$$
The embedding $H^{s-t}(\R^\dimens)\hookrightarrow L^2(\R^\dimens)$ is continuous and thus
$$
(-\Delta)^{t/2}\varphi_k\rightarrow (-\Delta)^{t/2}u \quad \text{in} \ L^2(\R^\dimens).
$$
By the continuity of the norm and $(-\Delta)^{s/2}$ we finally obtain
\begin{align*}
\aabs{(-\Delta)^{t/2}u}_{L^2(\R^\dimens)}=\lim_k\aabs{(-\Delta)^{t/2}\varphi_k}_{L^2(\R^\dimens)}&\leq C^{s-t}\lim_k\aabs{(-\Delta)^{s/2}\varphi_k}_{L^2(\R^\dimens)} \\ &=C^{s-t}\aabs{(-\Delta)^{s/2}u}_{L^2(\R^\dimens)}. \qedhere
\end{align*}
\end{proof}

We remark that the case $t=0$ and $s=1$ corresponds to the classical Poincar\'e inequality since $\aabs{\nabla u}_{L^2(\R^\dimens)}=\aabs{(-\Delta)^{1/2}u}_{L^2(\R^\dimens)}$. Also the constant $C^{s-t}$ is the expected one. In the usual Poincar\'e inequality we take one derivative and the constant is $C$. In the higher order version we take $t$ and $s$ derivatives and the constant naturally becomes~$C^{s-t}$. The constant~$C$ can be taken to be proportional to the diameter of the set, $C\approx\text{diam}(\Omega)$.

\begin{remark}
Fifth way to prove the Poincar\'e inequality is using uncertainty inequalities. If $u\in L^2(\R^\dimens)$, then there is a constant $c=c(\dimens, s)$ such that
\begin{equation}
\label{eq:uncertainty}
\aabs{u}_{L^2(\R^\dimens)}^2\leq c\aabs{\abs{\cdot}^s u}_{L^2(\R^\dimens)}\aabs{\abs{\cdot}^s \hat{u}}_{L^{2}(\R^\dimens)},
\end{equation}
see the discussion after theorem 4.1 in \cite{FS-the-uncertainty-principle}. We can interpret this inequality as
$$
\aabs{u}_{L^2(\R^\dimens)}^2\leq c\aabs{(-\Delta)^{s/2}(\ifourier (u))}_{L^2(\R^\dimens)}\aabs{(-\Delta)^{s/2} u}_{L^{2}(\R^\dimens)}
$$
whenever the terms on the right hand side of equation \eqref{eq:uncertainty} are finite. If $u$ is supported in some fixed compact set $K$, then one obtains similar inequality as in theorem~\ref{thm:poincarechineseguy}, i.e.
$$
\aabs{u}_{L^2(\R^\dimens)}\leq c^{\prime}\aabs{(-\Delta)^{s/2}u}_{L^{2}(\R^\dimens)}
$$
holds for all $u\in H^s_K(\R^\dimens)$ and for some constant $c^{\prime}=c^{\prime}(\dimens, K, s)$.
\end{remark}

\begin{remark}
The Poincar\'e inequality for the operator $(-\Delta)^{s/2}$ implies also Poincar\'e inequality for the fractional gradient $\nabla^s\colon H^s(\R^\dimens)\rightarrow L^2(\mathbb R^{2n}, \mathbb M^{\floor{s}+1})$ which is defined as
$$
\nabla^s u(x,y) := \frac{\mathcal C_{n,s'}^{1/2}}{\sqrt 2} \frac{ \nabla^{\floor{s}}u(x) - \nabla^{\floor{s}}u(y) }{|y-x|^{n/2 +s'+1}} \otimes (y-x),
$$
see section \ref{sec:magneticschrodinger} for more details. If $s\geq t\geq 0$ and $u\in C^\infty_c(\Omega)$, then
\begin{align}
\aabs{\nabla^t u}_{L^2(\mathbb R^{2n}, \mathbb M^{\floor{s}+1})}=\aabs{(-\Delta)^{t/2}u}_{L^2(\R^\dimens)}\leq\tilde{c}\aabs{(-\Delta)^{s/2}u}_{L^2(\R^\dimens)}=\tilde{c}\aabs{\nabla^s u}_{L^2(\mathbb R^{2n}, \mathbb M^{\floor{s}+1})},
\end{align}
where the constant $\tilde{c}$ does not depend on $u$. By approximation and the continuity of $\nabla^s$ the previous inequality is also true for $u\in\widetilde{H}^s(\Omega)$.
\end{remark}

\section{Applications to integral geometry}
\label{sec:applicationstointegralgeometry}
In this section we discuss how the UCP of Riesz potentials can be used in partial data problems in integral geometry. We follow \cite{HE:integral-geometry-radon-transforms} for the treatment of the $d$-plane transform, theory of X-ray transform and Radon transform can also be found in \cite{NA-mathematics-computerized-tomography, RK-radon-transform-an-local-tomography, SU:microlocal-analysis-integral-geometry}. Let $d\in \{1,\dotso, \dimens-1\}$ and denote by $\mathbf{P}^d$ the space of all $d$-dimensional affine planes in $\R^\dimens$. We define the $d$-plane transform of a function~$f$ to be
$$
\dplane f(A)=\int_{x\in A}f(x)\der m(x)
$$
where $A\in\mathbf{P}^d$ and $m$ is the Hausdorff measure on $A$. The adjoint of $\dplane$ is defined as 
$$
\dplane^* g(x)=\int_{A\ni x}g(A)\der\mu(A)
$$
where $g$ is a function on $\mathbf{P}^d$ and $\mu$ is the associated measure. These transforms are defined for functions such that the integrals exist. The case $d=1$ corresponds to the usual X-ray transform and $d=n-1$ to the Radon transform. The normal operator of the $d$-plane transform $\nod=\dplane^*\dplane$ has an expression $\nod f =c_{\dimens, d}(f\ast\abs{\cdot}^{-(\dimens-d)})$ where $c_{\dimens, d}$ is a constant depending on $\dimens$ and $d$. The normal operator is well defined if $f$ is a function that decreases rapidly enough at infinity \cite{HE:integral-geometry-radon-transforms}. This holds for example if $f\in C_{\infty}(\R^\dimens)$ where $C_\infty(\R
^\dimens)$ is the space of continuous functions which decrease faster than any polynomial at infinity (see section \ref{sec:distributions} for a precise definition). Thus, up to a constant factor, $\nod$ can be represented as a Riesz potential $\nod=\riesz=(-\Delta)^{-d/2}$ where $\alpha=\dimens-d\in \{1,\dotso, \dimens-1\}$.

The transforms $\dplane$ and $\dplane^*$ can be extended to distributions by duality. Let $f\in\cdistr(\R^\dimens)$ and $g\in\distr(\R^\dimens)$. Since $\dplane\colon\csmooth(\R^\dimens)\rightarrow\csmooth(\mathbf{P}^d)$ and $\dplane^*\colon\smooth(\mathbf{P}^d)\rightarrow\smooth(\R^\dimens)$ are continuous \cite{GO-range-of-d-plane-transform}, we can define the following operations
\begin{align*}
\ip{\dplane f}{\psi}&=\ip{f}{\dplane^*\psi}, \quad \psi\in \smooth(\mathbf{P}^d) \\
\ip{\dplane^* g}{\varphi}&=\ip{g}{\dplane\varphi}, \quad \varphi\in\csmooth(\R^\dimens).
\end{align*}
Therefore $\dplane f\in\cdistr(\mathbf{P}^d)$ and $\dplane^* g\in\distr(\R^\dimens)$. This shows that the normal operator $\nod=\dplane^*\dplane\colon\cdistr(\R^\dimens)\rightarrow\distr(\R^\dimens)$ is always defined and $\nod f=c_{n,d}(f\ast\abs{\cdot}^{-(n-d)})$ holds in the sense of distributions. Let $V\subset\R^\dimens$ be a nonempty open set and $f\in\cdistr(\R^\dimens)$. We say that $\dplane f$ vanishes on all $d$-planes intersecting $V$, if $\ip{\dplane f}{\varphi}=0$ for all $\varphi\in C_c^{\infty}(\mathbf{P}^d_V)$ where~$\mathbf{P}^d_V$ is the set of all $d$-planes intersecting~$V$. If $V=B(0, R)$ is a ball, $\varphi\in C_c^{\infty}(\mathbf{P}^d_V)$ means that $\varphi$ is smooth and $\varphi(A)=0$ for all $d$-planes $A$ for which $d(0, A)>r$ for some $r<R$. For more details on the range of the $d$-plane transform and duality in integral geometry, see \cite{GO-range-of-d-plane-transform} and \cite[Chapter II]{HE:integral-geometry-radon-transforms}.

\begin{remark}
\label{remark:ucpofdnormaloperator}
The UCP of Riesz potentials (corollary \ref{cor:uniquecontinuationofriespotential}) immediately implies the UCP of the normal operator of the $d$-plane transform when $d$ is odd (corollary \ref{cor:uniquecontinuationnormaloperator}) since~$\nod\approx I_{\dimens-d}$ and $d/2\not\in\Z$. However, such UCP cannot hold if $d$ is even, which can be shown by contradiction. Assume that corollary \ref{cor:uniquecontinuationnormaloperator} holds when $d$ is even. Take any nonzero $f\in C
^\infty_c(\R^\dimens)$. By the properties of the Fourier transform and Riesz potentials we have $(-\Delta)^{d/2}f=(-\Delta)^{-d/2}((-\Delta)^d f)=\nod(-\Delta)^d f$ up to a constant factor. Since $d$ is even $(-\Delta)^{d/2}$ is a local operator and we obtain $\nod(-\Delta)^d f|_V=(-\Delta)^d f|_V=0$ where $V\subset\R^\dimens$ is an open set outside the support of $f$ and $(-\Delta)^df\in C_c^\infty(\R^\dimens)$. By the assumption we would get that $(-\Delta)
^df=0$, i.e. $f$ is polyharmonic. But this implies $f=0$ by lemma \ref{lemma:polyharmoniccontinuation}, which is a contradiction. Hence the UCP cannot hold for $\nod$ when $d$ even.
\end{remark} 

Using the UCP of~$\nod$ we can now prove corollary \ref{cor:partialdataresult}.

\begin{proof}[Proof of corollary 1.3.]
Consider first $f\in C_{\infty}(\R^\dimens)$. Taking the adjoint, we get the conditions $\nod f|_V=0$ and $f|_V=0$. By corollary \ref{cor:uniquecontinuationnormaloperator} we obtain $f=0$ whenever $d$ is odd. Then let $f\in\cdistr(\R^\dimens)$. We can assume that $V=B(0, R)$ is a ball of radius $R$ centered at the origin. As in \cite{HE:integral-geometry-radon-transforms} we define the ``convolution" 
$$
(g\times\varphi)(A)=\int_{\R^\dimens}g(y)\varphi(A-y)\der y
$$
where $g\in C^{\infty}_c(\R^\dimens)$, $\varphi\in C^{\infty}_c(\mathbf{P}^d)$, $A\in\mathbf{P}^d$ and $A-y$ is a $d$-plane shifted by $y\in\R^\dimens$. Then one can calculate that $\dplane^*(g\times\varphi)=g\ast\dplane^*\varphi$ (see \cite[Proof of theorem 5.4]{HE:integral-geometry-radon-transforms}). Let $j_{\epsilon}\in C_c^{\infty}(\R^\dimens)$ be the standard mollifier and consider the mollifications $f\ast j_{\epsilon}\in C_c^{\infty}(\R^\dimens)$. By the properties of the convolutions
\begin{equation}
\label{eq:dplanemollifier}
\ip{\dplane(f\ast j_{\epsilon})}{\varphi}=\ip{f\ast j_{\epsilon}}{\dplane^*\varphi}=\ip{f}{j_{\epsilon}\ast\dplane^*\varphi}=\ip{f}{\dplane^*(j_{\epsilon}\times\varphi)}=\ip{\dplane f}{j_{\epsilon}\times\varphi}.
\end{equation}
Take $r>0$ and $\epsilon>0$ small enough so that $r+\epsilon<R$. Let $\varphi\in C_c^{\infty}(\mathbf{P}^d)$ such that $\varphi(A)=0$ for all $d$-planes which satisfy $d(0, A)>r$. Then $(j_{\epsilon}\times\varphi)(A)=0$ for all $d$-planes for which $d(0, A)>r+\epsilon$. Thus $j_{\epsilon}\times\varphi\in C_c^{\infty}(\mathbf{P}_V^d)$ and by \eqref{eq:dplanemollifier} it follows that $\dplane(f\ast j_{\epsilon})=0$ for all $d$-planes intersecting $B(0, r)$. We also have $(f\ast j_{\epsilon})|_{B(0, r)}=0$ and the first part of the proof implies the claim for $f\ast j_{\epsilon}$ for small $\epsilon>0$. Since $f\ast j_{\epsilon}\rightarrow f$ in $\cdistr(\R^\dimens)$ when $\epsilon\rightarrow 0$, we obtain the claim for~$f$.
\end{proof}

\begin{remark} 
\label{remark:partialdata}
When $d$ is even, corollary \ref{cor:partialdataresult} does not say that the result is false. It only indicates that we cannot use the UCP of the normal operator in the proof. This boils down to the fact that $\fraclaplace$ does not admit the UCP for $s\in\Z$. However, if $d$ is even, then the function $f$ is determined uniquely in $V$ by its integrals over $d$-planes which intersect $V$. Namely, if $\dplane f=0$ on all $d$-planes intersecting $V$, then $\nod f|_V=0$. Since $\nod\approx (-\Delta)^{-d/2}$, one can invert $\nod f$ by the local operator $(-\Delta)^{d/2}$ to obtain $f|_V=0$. Hence the ROI problem is uniquely solvable in this case without the additional knowledge of $f$ in an open set inside the ROI.
\end{remark}

\begin{remark}
We also note that unlike in the global data case lower dimensional data does not determine higher dimensional data. In other words, $R_k f=0$ for all $k$-planes intersecting~$V$ does not necessarily imply that $\dplane f=0$ for all $d$-planes which intersect~$V$ where $0<k<d<\dimens$. Thus one cannot reduce the partial data problem for $k$-planes to the partial data problem for $d$-planes.
\end{remark}

\section{Higher order fractional Schr\"odinger equation with singular potential}
\label{sec:schrodingerequation}

In this section, we study the fractional Schrödinger equation with higher order fractional Laplacian and singular potential. Let $\Omega\subset\R^\dimens$ be a bounded open set, $s\in\R^+\setminus\Z$ and consider the equation
\begin{align}\label{eq:schrodingerequation} \left\{\begin{array}{rl}
        (\fraclaplace+q)u&=0  \;\;\text{in}\;  \Omega\\
        u|_{\Omega_e} &=f
        \end{array}\right.\end{align}
where $u\in H^s(\R^\dimens)$, $f\in H^s(\R^\dimens)$ is the exterior value of $u$ and $q\in L^{\infty}(\Omega)$ represents the electric potential. When the potential $q$ is more singular one has to interpret the product $qu$ in a suitable way. If $q\in Z_0^{-s}(\R^\dimens)$, then $q$ acts as a multiplier and induces a map $m_q\colon H^s(\R^\dimens)\rightarrow H^{-s}(\R^\dimens)$ defined by $\ip{m_q(u)}{v}_{\R^\dimens}=\ip{q}{uv}_{\R^\dimens}$. Then equation~\eqref{eq:schrodingerequation} becomes
\begin{align}\label{eq:singularschrodingerequation} \left\{\begin{array}{rl}
        \fraclaplace u+m_q(u)&=0  \;\;\text{in}\;  \Omega\\
        u|_{\Omega_e} &=f.
        \end{array}\right.\end{align}
We will prove that the generalized DN map~$\Lambda_q$ for equation~\eqref{eq:singularschrodingerequation} determines the restriction of the potential $q\in Z_0^{-s}(\R^\dimens)$ to~$\Omega$ uniquely from exterior measurements. We also obtain the Runge approximation property for equation~\eqref{eq:singularschrodingerequation}: any function $g\in \widetilde{H}^s(\Omega)$ can be approximated arbitrarily well by solutions of \eqref{eq:singularschrodingerequation}.

Similar results were proved in \cite{RS-fractional-calderon-low-regularity-stability} when $0<s<1$. Our theorems generalize those results for higher order fractional Laplacians. The proofs rely essentially on the same thing: the UCP of the operator $\fraclaplace$ which was proved for $s\in\R^+\setminus\Z$ in section \ref{subsec:uniquecontinuationresults}. Also the higher order Poincar\'e inequality is needed for the well-posedness of the inverse problem. In this section, we provide the basic ideas for the proofs of the lemmas, which are reminescent of the ones in~\cite{RS-fractional-calderon-low-regularity-stability} and \cite{GSU-calderon-problem-fractional-schrodinger}. We will mainly follow the same notation as in those articles.

The strategy to prove theorems \ref{thm:schrodingeruniqueness} and \ref{thm:schrodingerrungeapproximation} is the following. First one constructs a bilinear form and proves that unique weak solutions are obtained in the complement of a countable set of eigenvalues. One also proves that $0$ is not an eigenvalue when \eqref{def:zeronotdirichleteigenvalue} holds. Then one defines the abstract DN map and proves the Alessandrini identity using it. Using the UCP of $\fraclaplace$ one obtains the Runge approximation property for equation~\eqref{eq:singularschrodingerequation}. From the Runge approximation and the Alessandrini identity, one can prove the uniqueness result for the inverse problem.

If $U\subset\R^\dimens$ is open and $u, v\in L^2(U)$, we denote the inner product by
$$
\ip{u}{v}_U=\int_U uv\der x.
$$
We also use the same notation $\ip{\cdot}{\cdot}_U$ for dual pairing.

The following lemma guarantees the existence of unique weak solutions (see \cite[Lemma 2.6]{RS-fractional-calderon-low-regularity-stability}). 

\begin{lemma}
\label{lemma:schrodingerexistenceofsolutions}
Let $\Omega\subset\R^\dimens$ be bounded open set, $s\in\R^+\setminus\Z$ and $q\in Z^{-s}_0(\R^\dimens)$. For $v, w\in H^s(\R^\dimens)$ define the bilinear form $B_q$ as
$$
B_q(v, w)=\ip{(-\Delta)^{s/2}v}{(-\Delta)^{s/2}w}_{\R^\dimens}+ \ip{m_q(v)}{w}_{\R^\dimens}.
$$
The following claims hold:
\begin{enumerate}[(a)]
    \item There is a countable set $\Sigma=\{\lambda_i\}_{i=1}^{\infty}\subset\R$, $\lambda_1\leq\lambda_2\leq\dotso\rightarrow\infty$, with the following property: if $\lambda\notin\Sigma$, then for any $F\in (\widetilde{H}^s(\Omega))^*$ and $f\in H^s(\R^\dimens)$ there is unique $u\in H^s(\R^\dimens)$ satisfying
    $$
    B_q(u, w)-\lambda\ip{u}{w}_{\R^\dimens}=F(w) \ \ \text{for} \ w\in \widetilde{H}^s(\Omega), \quad u-f\in \widetilde{H}^s(\Omega)
    $$
    with the norm estimate
    $$
    \aabs{u}_{H^s(\R^\dimens)}\leq C\Big(\aabs{F}_{(\widetilde{H}^s(\Omega))^*}+\aabs{f}_{H^s(\R^\dimens)}\Big)
    $$
    where $C$ is independent of $F$ and $f$.
    \item The function $u$ in $(a)$ is the unique $u\in H^s(\R^\dimens)$ satisfying
    $$
    (\fraclaplace u+m_q(u)-\lambda u)|_{\Omega}=F
    $$
    in the sense of distributions with $u-f\in\widetilde{H}^s(\Omega)$.
    \item One has $0\notin\Sigma$ if \eqref{def:zeronotdirichleteigenvalue} holds. If $q\in L^{\infty}(\Omega)$ and $q\geq 0$, then $\Sigma\subset (0, \infty)$ and \eqref{def:zeronotdirichleteigenvalue} always holds.
\end{enumerate}
\end{lemma}

\begin{proof}
The constants in the inequalities do not depend on the function $v$ in the proof. It is enough to solve the problem in $(a)$ for $u-f=v\in\widetilde{H}^s(\Omega)$. 
Using the fractional Poincar\'e inequality (theorem \ref{thm:generalpoincare}) we obtain
$$
\aabs{v}^2_{H^s(\R^\dimens)}\leq C \left( \aabs{(-\Delta)^{s/2}v}^2_{L^2(\R^\dimens)} + \aabs{v}^2_{L^2(\R^\dimens)} \right) \leq C^{\prime}\aabs{(-\Delta)^{s/2}v}^2_{L^2(\R^\dimens)}.
$$
Let $0<\epsilon<1/C^{\prime}$ where the constant $C^{\prime}>0$ comes from the previous inequality. Since $q\in Z^{-s}_0(\R^\dimens)$, we can find $q_s\in C_c^{\infty}(\R^\dimens)$ and $q_r\in Z^{-s}(\R^\dimens)$ such that $q=q_s+q_r$ and $\aabs{q_r}_{Z^{-s}(\R^\dimens)}<\epsilon$. When we take $\mu=\aabs{q_s^-}_{L^{\infty}(\R^\dimens)}$ where $q_s^-=-\min\{0, q_s(x)\}$, we obtain the coercivity condition
\begin{align*}
B_q(v, v)+\mu\ip{v} {v}_{\R^\dimens}&\geq\aabs{(-\Delta)^{s/2}v}_{L^2(\R^\dimens)}^2+\ip{q_r}{vv}_{\R^\dimens}\geq \frac{1}{C^{\prime}}\aabs{v}_{H^s(\R^\dimens)}^2-\epsilon\aabs{v}_{H^s(\R^\dimens)}^2.
\end{align*}
Hence $v, w\mapsto B_q(v, w)+\mu\ip{v}{w}_{\R^\dimens}$ defines an equivalent inner product in $\widetilde{H}^s(\Omega)$. The proof is then completed as in \cite{GSU-calderon-problem-fractional-schrodinger}: the Riesz representation theorem implies that for every $\widetilde{F}\in (\widetilde{H}^s({\Omega}))^*$ there is unique $v=G_{\mu}\widetilde{F}\in\widetilde{H}^s(\Omega)$ such that $B_q(v, w)+\mu\ip{v}{w}_{\R^\dimens}=\widetilde{F}(w)$ for all $w\in\widetilde{H}^s(\Omega)$. The map $G_{\mu}\colon(\widetilde{H}^s(\Omega))^*\rightarrow\widetilde{H}^s(\Omega)$ induces a compact, self-adjoint and positive definite operator $\widetilde{G}_{\mu}\colon L^2(\Omega)\rightarrow L^2(\Omega)$ by the compact Sobolev embedding theorem. The spectral theorem for the self-adjoint compact operator $\widetilde{G}_{\mu}$ implies the claim in $(a)$. Part $(b)$ holds since $C_c^{\infty}(\Omega)$ is dense in $\widetilde{H}^s(\Omega)$. The first claim in $(c)$ follows from the Fredholm alternative. The second claim in $(c)$ is essentially the same as in \cite[Lemma 2.3]{GSU-calderon-problem-fractional-schrodinger} and is proved by replacing $\widetilde{H}^s(\Omega)$ with $H_{\overline{\Omega}}(\R^\dimens)$ in the proof of $(a)$.
\end{proof}

Recall the definition of the abstract trace space $X=H^s(\R^\dimens)/\widetilde{H}^s(\Omega)$ which we equip with the quotient norm
$$
\aabs{[f]}_X=\inf_{\varphi\in\widetilde{H}^s(\Omega)}\aabs{f-\varphi}_{H^s(\R^\dimens)}, \quad f\in H^s(\R^\dimens).
$$
The following lemma implies that the DN map is well-defined and continuous. The result follows immediately from the definition of the bilinear form $B_q(\cdot, \cdot)$ and from the continuity of $(-\Delta)^{s/2}\colon H^s(\R^\dimens)\rightarrow L^2(\R^\dimens)$ (see \cite[Lemma 2.4]{GSU-calderon-problem-fractional-schrodinger}).

\begin{lemma}
\label{lemma:dnmap}
Let $\Omega\subset\R^\dimens$ be bounded open set, $s\in\R^+\setminus\Z$ and $q\in Z_0^{-s}(\R^\dimens)$ which satisfies~\eqref{def:zeronotdirichleteigenvalue}. Then the map $\Lambda_q\colon X\rightarrow X^*$, $\ip{\Lambda_q[f]}{ [g]}=B_q(u_f, g)$, is linear and continuous, where $u_f\in H^s(\R^\dimens)$ solves $\fraclaplace u+m_q(u)=0$ in $\Omega$ with $u-f\in\widetilde{H}^s(\Omega)$. One also has the self-adjointness property $\ip{\Lambda_q[f]}{[g])}=\ip{[f]} {\Lambda_q[g]}$ for $f, g\in H^s(\R^\dimens)$.
\end{lemma}

\begin{proof}
 Since $u_f$ is a solution to $\fraclaplace u+m_q(u)=0$ in $\Omega$ with $u_f-f\in\widetilde{H}^s(\Omega)$ and solutions are unique, we obtain $B_q(u_{f+\varphi}, g+\psi)=B_q(u_f, g)$ for $\varphi, \psi\in\widetilde{H}^s(\Omega)$. This implies that~$\Lambda_q$ is well-defined. Further, using continuity of $(-\Delta)^{s/2}$ and the norm estimate for solution $u_f$ from lemma \ref{lemma:schrodingerexistenceofsolutions}, we obtain
\begin{align*}
\abs{\ip{\Lambda_q[f]}{[g]}}&\leq\aabs{(-\Delta)^{s/2}u_f}_{L^2(\R^\dimens)}\aabs{(-\Delta)^{s/2}g}_{L^2(\R^\dimens)}+\aabs{q}_{Z^{-s}(\R^\dimens)}\aabs{u_f}_{H^s(\R^\dimens)}\aabs{g}_{H^s(\R^\dimens)} \\
&\leq C\aabs{f}_{H^s(\R^\dimens)}\aabs{g}_{H^s(\R^\dimens)},
\end{align*}
where $C$ does not depend on $f$ and $g$.
By the definition of the quotient norm $\abs{\ip{\Lambda_q[f]}{[g]}}\leq C\aabs{[f]}_{X}\aabs{[g]}_{X}$, so~$\Lambda_q$ is continuous. Choosing $g=u_g$ we obtain by symmetry of $B_q(\cdot, \cdot)$
$$
\ip{\Lambda_q[f]}{[g]}=B_q(u_f, u_g)=\ip{\Lambda_q[g]}{[u_f]}=\ip{[f]}{\Lambda_q[g]}
$$
where we used the fact that $u_f-f\in\widetilde{H}^s(\Omega)$.
\end{proof}

We immediately obtain the Alessandrini identity from lemma \ref{lemma:dnmap} (see \cite[Lemma 2.7]{RS-fractional-calderon-low-regularity-stability}). We use some abuse of notation and write $f$ instead of $[f]$.

\begin{lemma}[Alessandrini identity]
\label{lemma:schrodingeralessandrini}
Let $\Omega\subset\R^\dimens$ be bounded open set, $s\in\R^+\setminus\Z$ and $q_1, q_2\in Z_0^{-s}(\R^\dimens)$ which satisfy \eqref{def:zeronotdirichleteigenvalue}. For any $f_1, f_2\in X$ one has
$$
\ip{(\Lambda_{q_1}-\Lambda_{q_2})f_1}{f_2}=\ip{m_{q_1-q_2}(u_1)}{u_2}_{\R^\dimens}
$$
where $u_i\in H^s(\R^\dimens)$ solves $\fraclaplace u_i+m_{q_i}(u_i)=0$ in $\Omega$ with $u_i-f_i\in\widetilde{H}^s(\Omega)$.
\end{lemma}

\begin{proof}
Using the self-adjointness of~$\Lambda_q$ and the property $B_q(u_i, g+\psi)=B_q(u_i, g)$ for $\psi\in\widetilde{H}^s(\Omega)$, we obtain
\begin{align*}
\ip{(\Lambda_{q_1}-\Lambda_{q_2})f_1}{f_2}&=\ip{\Lambda_{q_1}f_1}{f_2}-\ip{f_1}{\Lambda_{q_2}f_2}=B_{q_1}(u_1, f_2)-B_{q_2}(u_2, f_1) \\
&=B_{q_1}(u_1, f_2+(u_2-f_2))-B_{q_2}(u_2, f_1+(u_1-f_1)) \\
&= B_{q_1}(u_1, u_2)-B_{q_2}(u_1, u_2)=\ip{m_{q_1-q_2}(u_1)}{u_2}_{\R^\dimens}
\end{align*}
which gives the claim.
\end{proof}

From the UCP of $\fraclaplace$ (theorem \ref{thm:uniquecontinuationoffractionallaplacian}), we obtain the following approximation result (see \cite[Lemma 8.1]{RS-fractional-calderon-low-regularity-stability}).

\begin{lemma}
\label{lemma:schrodingerrunge}
Let $\Omega\subset\R^\dimens$ be bounded open set, $s\in\R^+\setminus\Z$ and $q\in Z_0^{-s}(\R^\dimens)$ which satisfies~\eqref{def:zeronotdirichleteigenvalue}. Denote by $P_q\colon X\rightarrow H^s(\R^\dimens)$, $P_q f=u_f$, where $u_f\in H^s(\R^\dimens)$ is the unique solution to $\fraclaplace u +m_q(u)=0$ in $\Omega$ with $u-f\in \widetilde{H}^s(\Omega)$ given by lemma~\ref{lemma:schrodingerexistenceofsolutions}. Let $W\subset\Omega_e$ be any open set and define the set 
$$
\mathcal{R}=\{P_q f-f: f\in C_c^{\infty}(W)\}.
$$
Then $\mathcal{R}$ is dense in $\widetilde{H}^s(\Omega)$.
\end{lemma}

\begin{proof}
By the Hahn-Banach theorem it is enough to show that if $F\in (\widetilde{H}^s(\Omega))^*$ and $\ip{F}{v}=0$ for all $v\in\mathcal{R}$, then $F=0$. Let $F\in (\widetilde{H}^s(\Omega))^*$ and assume that 
$$
\ip{F}{P_qf-f}=0, \quad f\in C_c^{\infty}(W).
$$
Let $\varphi\in\widetilde{H}^s(\Omega)$ be the solution to 
$$
\fraclaplace\varphi+m_q(\varphi)=F \ \text{in} \ \Omega, \quad \varphi|_{\Omega_e}=0
$$ 
which exists by lemma \ref{lemma:schrodingerexistenceofsolutions}. This means that $B_q(\varphi, w)=\ip{F}{w}$ for all $w\in\widetilde{H}^s(\Omega)$. Let $u_f=P_q f\in H^s(\R^\dimens)$ where $u_f-f\in\widetilde{H}^s(\Omega)$. Now
$$
\ip{F}{P_q f-f}=B_q(\varphi, u_f-f)=-B_q(\varphi, f)
$$
since $u_f$ is a solution to $\fraclaplace u+m_q(u)=0$ in $\Omega$ and $\varphi\in\widetilde{H}^s(\Omega)$. Thus $B_q(\varphi, f)=0$ for all $f\in C_c^{\infty}(W)$. Using the fact that $\spt(\varphi)$ and $\spt(f)$ are disjoint, we obtain
$$
0=\ip{(-\Delta)^{s/2}\varphi}{(-\Delta)^{s/2}f}_{\R^\dimens}=\ip{\fraclaplace\varphi}{f}_{\R^\dimens}.
$$
Here we used that $\ip{(-\Delta)^{s/2}u}{ (-\Delta)^{s/2}v}_{\R^\dimens}=\ip{\fraclaplace u}{ v}_{\R^\dimens}$ for $u, v\in\schwartz(\R^\dimens)$ and the equality holds also in $H^s(\R^\dimens)$ by density. Hence $\varphi|_W=\fraclaplace\varphi|_W=0$ and theorem \ref{thm:uniquecontinuationoffractionallaplacian} implies $\varphi=0$ and eventually $F=0$.
\end{proof}

We remark that exactly the same proof gives the density of $r_\Omega\mathcal{R}$ in $L^2(\Omega)$ where $r_\Omega$ is the restriction to $\Omega$ (see \cite[Lemma 5.1]{GSU-calderon-problem-fractional-schrodinger}). Now it is easy to prove theorems \ref{thm:schrodingeruniqueness} and \ref{thm:schrodingerrungeapproximation}.

\begin{proof}[Proof of theorem \ref{thm:schrodingeruniqueness}.]
Since we can always shrink the sets $W_i$, we can assume without loss of generality that $\overline{W}_1\cap\overline{W}_2=\varnothing$ and $(\overline{W}_1\cup\overline{W}_2)\cap\overline{\Omega}=\varnothing$. Using the Alessandrini identity (lemma \ref{lemma:schrodingeralessandrini}), we obtain 
$$
\ip{m_{q_1-q_2}(u_1)}{u_2}_{\R^\dimens}=0
$$
for any $u_i\in H^s(\R^\dimens)$ which solves $\fraclaplace u_i+m_{q_i}(u_i)=0$ in $\Omega$ with exterior values in $C_c^{\infty}(W_i)$. Let $v_1, v_2\in\widetilde{H}^s(\Omega)$. By lemma \ref{lemma:schrodingerrunge} there are sequences of exterior values $f_1^k\in C_c^{\infty}(W_1)$ and $f_2^k\in C_c^{\infty}(W_2)$ with sequences of solutions $u_1^k, u_2^k\in H^s(\R^\dimens)$ such that
\begin{itemize}
\item $\fraclaplace u_i^k+m_{q_i}(u_i^k)=0$ in $\Omega$
\item $u_i^k-f_i^k\in\widetilde{H}^s(\Omega)$
\item $u_i^k=f_i^k+v_i+r_i^k$ where $r_i^k\xrightarrow{k\rightarrow\infty}0$ in $\widetilde{H}^s(\Omega)$.
\end{itemize}
When we insert the solutions $u_i^k$ into the Alessandrini identity, use the support conditions and take the limit $k\rightarrow\infty$, we obtain
$$
\ip{m_{q_1-q_2}(v_1)}{v_2}_{\R^\dimens}=0.
$$
Let $\varphi\in C_c^{\infty}(\Omega)$. Choose $v_1=\varphi$ and $v_2\in C_c^{\infty}(\Omega)$ such that $v_2=1$ in a neighborhood of $\spt(\varphi)$. This implies
$$
0=\ip{m_{q_1-q_2}(v_1)}{v_2}_{\R^\dimens}=\ip{q_1-q_2}{v_1 v_2}_{\R^\dimens}=\ip{q_1-q_2}{\varphi}_{\R^\dimens}
$$
which is equivalent to that $q_1|_{\Omega}=q_2|_{\Omega}$ as distributions.
\end{proof}

\begin{proof}[Proof of theorem \ref{thm:schrodingerrungeapproximation}.]
Since $\text{int}(\Omega_1\setminus\Omega)\neq\varnothing$, there is open set $W\subset\Omega_e$ such that $\overline{W}\subset\Omega_1\setminus\overline{\Omega}$. By lemma \ref{lemma:schrodingerrunge}, the set $\mathcal{R}$ is dense in $\widetilde{H}^s(\Omega)$. Hence, we can approximate any $g\in \widetilde{H}^s(\Omega)$ arbitrarily well by solutions $u\in H^s(\R^\dimens)$ to the equation $\fraclaplace u+m_q(u)=0$ in $\Omega$ with $u-f\in\widetilde{H}^s(\Omega)$. Since $f\in C_c^{\infty}(W)$ we especially have $\spt(u)\subset\overline{\Omega}_1$.
\end{proof}

\section{Higher order fractional magnetic Schr\"odinger equation}
\label{sec:magneticschrodinger}
In this section we will be dealing with the definition of the FMSE, as well as with the proof of the injectivity result for the corresponding inverse problem. For the sake of simplicity, let us fix the convention throughout this section that the symbol $\langle\cdot,\cdot\rangle$ indicates both the scalar product (duality pairing) on $L^2(\mathbb R^n)$ and the one on $L^2(\mathbb R^{2n})$, the distinction between the two being always possible by checking the number of free variables inside the brackets. We also let the norms $\aabs{\cdot}_{L^2}$, $\aabs{\cdot}_{H^s}$ etc. to denote the norms over the whole~$\R^\dimens$ or $\R^{2\dimens}$ when the base set is not specified.

\subsection{High order bivariate functions} 

Let $l,n\in \mathbb N$, and consider a family $A$ of scalar two-point functions indexed over the set $\{ 1, ..., n \}^l$. A generic member of the family is determined by a vector $(i_1, ..., i_l)$ such that $i_j \in \{ 1, ..., n \}$ for all $j\in \{ 1, ..., l \}$, and it is a function $A_{i_1, ..., i_l}: \mathbb R^{2n}\rightarrow \mathbb R$. We call such family $A$ a \emph{bivariate function of order} $l$. One can see the family $A$ as a function $A: \R^{2n} \rightarrow \mathbb M^l$, where $\mathbb M^l$ is the set of all $n\times...\times n = n^l$ arrays of real numbers, i.e. tensors of order $l$. 
\vspace{3mm}

Let $a, b \in \mathbb N$, and let $A, B$ be bivariate functions of orders $a$ and $b$ respectively, in the variables $x_1, x_2$. The \emph{tensor product} of $A$ and $B$ is the bivariate function of order $a+b$ given by $$ (A \otimes B)_{i_1, ..., i_a, j_1, ..., j_b} (x_1, x_2) := A_{i_1, ..., i_a}(x_1, x_2) B_{j_1, ..., j_b}(x_1, x_2)\;. $$ In particular, for a vector $\xi\in\R^\dimens$ we let $\xi^{\otimes 0}=0$, $\xi^{\otimes 1}=\xi$ and recursively $\xi^{\otimes m}=\xi^{\otimes (m-1)}\otimes\xi $.

\noindent Let $A,B$ as before, but assume now that $a\geq b$. The \emph{contraction} of $A$ and $B$ is the bivariate function of order $a-b$ given by $$ (A \cdot B)_{i_1, ..., i_{a-b}} (x_1, x_2) := \sum_{j_1, ..., j_b = 1}^n A_{i_1, ..., i_{a-b}, j_1, ..., j_b}(x_1, x_2) B_{j_1, ..., j_b}(x_1, x_2)\;. $$

\noindent If $A=B$, then of course $a=b$, so that $A\cdot A$ is a scalar function of the variables $(x_1, x_2)$. One sees that $|A|:=(A\cdot A)^{1/2}$ defines a norm for fixed $x_1$ and $x_2$, and that this coincides with the usual one when $A$ is a vector function.
\vspace{3mm}

\begin{lemma}\label{prodlem}
Let $a, b \in \mathbb N$, and let $A, B, v$ be bivariate functions of orders $a,b$ and $1$ respectively, in the variables $x_1, x_2$. Assume that $a\geq b+1$; then \begin{equation*}A\cdot (B\otimes v) = (A\cdot v)\cdot B\;.\end{equation*}
\end{lemma}

\begin{proof}
The proof is just a simple computation:
\begin{align*}
[A\cdot (B\otimes v)]_{i_1, ..., i_{a-b-1}}& = \sum_{j_1, ..., j_{b+1} = 1}^n A_{i_1, ..., i_{a-b-1}, j_1, ..., j_{b+1}} (B\otimes v)_{j_1, ..., j_{b+1}} \notag\\ & = \sum_{j_1, ..., j_{b+1} = 1}^n A_{i_1, ..., i_{a-b-1}, j_1, ..., j_{b+1}} B_{j_1, ..., j_b} v_{j_{b+1}} \notag\\ & = \sum_{j_1, ..., j_{b} = 1}^n B_{j_1, ..., j_b} \sum_{j_{b+1}=1}^n A_{i_1, ..., i_{a-b-1}, j_1, ..., j_{b+1}} v_{j_{b+1}} \notag\\ & = \sum_{j_1, ..., j_{b} = 1}^n B_{j_1, ..., j_b} (A\cdot v)_{i_1, ..., i_{a-b-1}, j_1, ..., j_{b}} \notag\\ & = [(A\cdot v)\cdot B]_{i_1, ..., i_{a-b-1}}\;. \qedhere
\end{align*}
\end{proof}

Let $A$ be a bivariate function of any order. Following \cite{CO-magnetic-fractional-schrodinger}, we recall the definitions of the \emph{symmetric} and \emph{antisymmetric parts of $A$ with respect to the variables $x$ and $y$} and the $L^2$ \emph{norms of} $A$ \emph{with respect to the first and second variable at point} $x$:
$$ A_s(x,y):= \frac{A(x,y)+A(y,x)}{2}\,, \;\;\;\;\;\; A_a(x,y) := A(x,y)-A_s(x,y)\;, $$ $$ \mathcal{J}_1 A(x) := \left( \int_{\mathbb R^n} |A(y,x)|^2 \,dy \right)^{1/2}\;\;, \;\;\;\; \mathcal{J}_2 A(x) := \left( \int_{\mathbb R^n} |A(x,y)|^2 \,dy \right)^{1/2}\;.$$  

\noindent It is easily seen that $A\in L^2$ implies $A_s, A_a \in L^2$, since \begin{equation}\label{symest} \|A_s\|_{L^2} = \left\| \frac{A(x,y)+A(y,x)}{2} \right\|_{L^2} \leq \|A\|_{L^2}\;,\;\;\; \|A_a\|_{L^2} = \left\| \frac{A(x,y)-A(y,x)}{2} \right\|_{L^2} \leq \|A\|_{L^2}\;. \end{equation}

A bivariate function $A$ of any order will be called \emph{symmetric} if $A=A_s$ almost everywhere, and \emph{antisymmetric} if $A=A_a$ almost everywhere. 

\begin{lemma}\label{antisymlem}
Let $A\in L^1(\mathbb R^{2n}, \mathbb M^l)$ be an antisymmetric bivariate function of order $l$ for some $l\in \mathbb N$. Then $\int_{\mathbb R^{2n}} A(x,y) \, dydx =0\;.$
\end{lemma}

\begin{proof}
Let $D^+, D^-$ be the sets respectively above and under the diagonal $D:=\{(x,y)\in \mathbb R^{2n} : x=y\}$ of $\mathbb R^{2n}$. Since $\int_{D^\pm} A(x,y) \, dydx \leq \int_{D^\pm} |A(x,y)| \, dydx \leq \|A\|_{L^1} <\infty$, we can decompose the integral as $\int_{\mathbb R^{2n}} A(x,y) \, dydx = \int_{D^+} A(x,y) \, dydx + \int_{D^-} A(x,y) \, dydx$. Given the symmetry of the sets $D^+$ and $D^-$, this can be rewritten as $ \int_{\mathbb R^{2n}} A(x,y) \, dydx = \int_{D^+} (A(x,y) + A(y,x)) \, dydx \;, $ which vanishes by virtue of the antisymmetry of $A$. \qedhere
\end{proof}

\subsection{Fractional operators}
Let $s\in \mathbb R^+\setminus \mathbb Z$, $u\in C^{\infty}_c(\mathbb R^n)$ and $x,y \in \mathbb R^n$. Let $\floor{s} := \sup\{n\in \mathbb N : n< s\}$ and $s' := s-\floor{s}$, so that by definition $s'\in (0,1)$. The \emph{fractional gradient of order $s$ of $u$ at points $x$ and $y$} is the following symmetric bivariate function of order $\floor{s}+1$: 
$$\nabla^s u(x,y) := \frac{\mathcal C_{n,s'}^{1/2}}{\sqrt 2} \frac{ \nabla^{\floor{s}}u(x) - \nabla^{\floor{s}}u(y) }{|y-x|^{n/2 +s'+1}} \otimes (y-x)\;.$$

\noindent Observe that this definition coincides with the usual one for $s\in(0,1)$, since in this case $\floor{s}=0$ and $s'=s$. One can compute
\begin{align}
\|\nabla^s u\|^2_{L^2(\mathbb R^{2n}, \mathbb M^{\floor{s}+1})} & = \frac{\mathcal{C}_{n,s'}}{2} \int_{\mathbb R^n}\int_{\mathbb R^n} \frac{|\nabla^{\floor{s}}u(x) - \nabla^{\floor{s}}u(y)|^2}{|x-y|^{n+2s'}} dx\,dy \notag\\ & = \frac{\mathcal{C}_{n,s'}}{2} [\nabla^{\floor{s}}u]^2_{\dot {H}^{s'}(\mathbb R^n)}=\aabs{(-\Delta)^{s'/2}\nabla^{\floor{s}}u}_{L^2(\R^\dimens)}^2 \notag\\ &= \| |\xi|^{s'} \xi^{\otimes \floor{s}} \hat u(\xi) \|^2_{L^2(\mathbb R^n)} = \| |\xi|^{s}\hat u(\xi) \|^2_{L^2(\mathbb R^n)} \notag\\ & = \left\| (-\Delta)^{s/2} u \right\|^2_{L^2(\mathbb R^n)} \leq \|u\|^2_{H^s(\mathbb R^n)}\label{L2Hsest}\;.
\end{align}

\noindent Thus, by the density of $C^{\infty}_c$ in $H^s$, $\nabla^s$ can be extended to a continuous operator $\nabla^s : H^s(\mathbb R^n) \rightarrow L^2(\mathbb R^{2n}, \mathbb M^{\floor{s}+1})$. One sees by density that the formula given for $\nabla^s u$ in the case $u\in C^\infty_c(\R^\dimens)$ still holds almost everywhere for $u\in H^s(\R^\dimens)$.
Thus if $u,v\in H^s$, by the above computation, $$ \langle\nabla^s u,\nabla^s u\rangle = \| (-\Delta)^{s/2} u \|^2_{L^2} = \langle (-\Delta)^{s/2} u, (-\Delta)^{s/2} u \rangle = \langle (-\Delta)^s u, u \rangle\;,$$

\noindent so that by the polarization identity and the self-adjointness of $(-\Delta)^s$,
\begin{align*}
\langle\nabla^s u,\nabla^s v\rangle & = \frac{\langle\nabla^s (u+v),\nabla^s (u+v)\rangle - \langle\nabla^s u,\nabla^s u\rangle - \langle\nabla^s v,\nabla^s v\rangle}{2} \\ & = \frac{\langle(-\Delta)^s (u+v),u+v\rangle - \langle(-\Delta)^s u, u\rangle - \langle(-\Delta)^s v, v\rangle}{2} \\ & = \frac{\langle(-\Delta)^s u, v\rangle + \langle(-\Delta)^s v, u\rangle}{2} = \langle(-\Delta)^s u, v\rangle\;.
\end{align*}
\noindent This proves that if the \emph{fractional divergence} $(\nabla\cdot)^s : L^2(\mathbb R^{2n}, \mathbb M^{\floor{s}+1}) \rightarrow H^{-s}(\mathbb R^n)$ is defined as the adjoint of $\nabla^s$, then weakly $(\nabla\cdot)^s\nabla^s = (-\Delta)^s$ for $s\in \mathbb R^+\setminus \mathbb Z$. This result was already proved in \cite{Co18}, but only for the case $s\in (0,1)$. If we define the antisymmetric bivariate vector function $$\alpha(x,y) := \frac{\mathcal C_{n,s'}^{1/2}}{\sqrt 2} \frac{y-x}{|y-x|^{n/2 +s'+1}}$$ then for $u\in H^s$ the identity \begin{align}\label{gradalpha} \nabla^s u(x,y) = ( \nabla^{\floor{s}}u(x) - \nabla^{\floor{s}}u(y) )\otimes \alpha(x,y) \end{align} \noindent holds almost everywhere.
\vspace{3mm}

We now define the magnetic versions of the above operators. Fix $p>\max\{1, n/2s\}$, and let $A$ be a bivariate function of order $\floor{s}+1$ such that 

\begin{enumerate}[label=(a\arabic*)]
\item \label{assumption1} $\mathcal{J}_2A \in L^{2p}(\mathbb R^n)$ 
\item \label{assumption2} $\spt(A) \subset \Omega\times\Omega$.
\end{enumerate}

\noindent With such choice of $p$, the embedding $H^s \times L^{2p}\hookrightarrow L^2$ always holds by \cite[Theorem 6.1]{BH2017}, recall that $W^r(\R^\dimens)=H^r(\R^\dimens)$ with equivalent norms when $r\in\R$ and $W^r(\R^\dimens)$ is the $L^2$ Sobolev-Slobodecki space \cite{BH2017, ML-strongly-elliptic-systems}. Therefore, if $u\in H^s$,
\begin{align*}
\| A(x,y)u(x) \|_{L^2(\mathbb R^{2n}, \mathbb M^{\floor{s}+1})} & = \left(\int_{\mathbb R^n} \abs{u(x)}^2 \int_{\mathbb R^n} |A(x,y)|^2 dy\, dx\right)^{1/2} \\ & = \left(\int_{\mathbb R^n} \abs{u(x)}^2 \,\abs{\mathcal{J}_2A(x)}^2 \,dx\right)^{1/2} = \|u \,\mathcal{J}_2A\|_{L^2(\mathbb R^n)} \\ & \leq c \|u\|_{H^s} \|\mathcal{J}_2A\|_{L^{2p}}<\infty,
\end{align*}
where $c$ does not depend on $u$ and $A$.
\noindent This allows the definition of $\nabla^s_A u(x,y) := \nabla^s u(x,y) + A(x,y)u(x)$ and its adjoint $(\nabla\cdot)^s_A$ just as in \cite{CO-magnetic-fractional-schrodinger}, in such a way that $\nabla^s_A : H^s(\mathbb R^n) \rightarrow L^2(\mathbb R^{2n}, \mathbb M^{\floor{s}+1})$ and $(\nabla\cdot)^s_A : L^2(\mathbb R^{2n}, \mathbb M^{\floor{s}+1}) \rightarrow H^{-s}(\mathbb R^n)$. By definition, the \emph{magnetic fractional Laplacian} $(-\Delta)^s_A : H^s \rightarrow H^{-s}$ will be the composition $(\nabla\cdot)^s_A\nabla^s_A$. Let now $q$ be a scalar field such that

\begin{enumerate}[label=(a\arabic*)]
\setcounter{enumi}{2}
\item \label{assumption3}$q \in L^{p}(\Omega)$.
\end{enumerate}

\noindent By \cite[Theorem 8.3]{BH2017} we have the embedding $H^s \times L^p \hookrightarrow H^{-s}$ and hence $qu\in H^{-s}$ holds for $u\in H^s$. We can thus define the magnetic Schr\"odinger operator $(-\Delta)^s_A + q : H^s \rightarrow H^{-s}$ and the fractional magnetic Schr\"odinger equation (FMSE) $$ (-\Delta)^s_Au + qu =0\;.$$

In the next Lemma we write $(-\Delta)^s_A$ in a more convenient form. To this scope, we introduce the bivariate function of order $\floor{s}$ given by $S(x,y):= A(x,y)\cdot\alpha(x,y)$, for which we assume that

\begin{enumerate}[label=(a\arabic*)]
\setcounter{enumi}{3}
\item \label{assumption4} $|S(x,y)| \leq \tilde S(y)$ for a.e. $x,y\in\mathbb R^n$, with $\tilde S \in L^2$,
\item \label{assumption5}  $S(x,y) \in H^{\floor{s}}(\mathbb R^{2n},\mathbb M^{\floor{s}} )$.
\end{enumerate}

\begin{remark}
Assumption \ref{assumption4} is really relevant only when $\floor{s}\neq 0$, as it will be clear from the proof; in the case $s\in (0,1)$, this assumption can be reduced. We refer to \cite{CO-magnetic-fractional-schrodinger} for a set of sufficient conditions in that regime. Moreover, with a more careful analysis, one could reduce the exponent of the space to which $\tilde S$ belongs. However, we decided to keep $L^2$ for the sake of simplicity.   
\end{remark}

\begin{lemma}\label{expanded} Let $A$ be a bivariate function of order $\floor{s}+1$ satisfying conditions \ref{assumption1}, \ref{assumption2} \ref{assumption4}, \ref{assumption5}, and let $u\in H^s$. There exist linear operators $\mathfrak N, \mathfrak M_\beta$ acting on bivariate functions of order $\floor{s}$, with $\beta$ a multi-index of length $|\beta|\leq \floor{s}$, such that the equation \begin{align*}
(-\Delta)^s_A u(x) = & (-\Delta)^s u(x) + \sum_{|\beta|\leq\floor{s}}\partial^\beta u(x) (\mathfrak{M}_\beta(S))(x) + \\ \;\;\;& + \int_{\mathbb R^n} u(y)  (\mathfrak{N}(S))(x,y)\,dy +  u(x)\int_{\mathbb R^n} |A(x,y)|^2 dy 
\end{align*} 
\noindent holds in weak sense.
\end{lemma}
 
\begin{proof}
If $v\in H^s$, then in weak sense \begin{equation}\label{maglap1} \langle(-\Delta)^s_A u, v\rangle = \langle\nabla^s u, \nabla^sv\rangle + \langle \nabla^s u, Av \rangle + \langle \nabla^s v, Au \rangle + \left\langle Au ,Av\right\rangle\;, \end{equation} \noindent where all the terms on the right hand side are finite, as observed above. 

\para{Step 1} Let us start by computing the third term on the right hand side of \eqref{maglap1}. The bivariate function $\nabla^s v(x,y) [A(x,y)u(x)]_a$ is antisymmetric, and by Cauchy-Schwartz and formula~\eqref{symest} we have $\|\nabla^s v\, (Au)_a\|_{L^1} \leq \|\nabla^s v\|_{L^2} \|(Au)_a\|_{L^2} \leq \|v\|_{H^s} \|Au\|_{L^2}< \infty$. Therefore Lemma \ref{antisymlem} gives $\langle \nabla^s v, (Au)_a \rangle =0$, and we can use Lemma \ref{prodlem} to write \begin{align}\notag \langle \nabla^s v, Au \rangle & = \langle \nabla^s v, Au \rangle - \langle \nabla^s v, (Au)_a \rangle = \langle \nabla^s v, (Au)_s \rangle \\ & \notag = \langle (\nabla^{\floor{s}} v(x)- \nabla^{\floor{s}} v(y)) \otimes \alpha, (Au)_s \rangle \\ & \label{maglap2} = \langle \nabla^{\floor{s}} v(x)- \nabla^{\floor{s}} v(y), (Au)_s\cdot \alpha \rangle \\ & \notag = \langle \nabla^{\floor{s}} v(x)- \nabla^{\floor{s}} v(y), (A\cdot\alpha u)_a \rangle \\ & \notag = \langle \nabla^{\floor{s}} v(x)- \nabla^{\floor{s}} v(y), (S u)_a \rangle\;.\end{align}

\noindent The bivariate function $[\nabla^{\floor{s}} v(x)+ \nabla^{\floor{s}} v(y)][S(x,y)u(x)]_a$ is antisymmetric, and we can estimate its $L^1$ norm by means of the triangle inequality as $$ \|(\nabla^{\floor{s}} v(x)+ \nabla^{\floor{s}} v(y))(Su)_a\|_{L^1} \leq  \|(\nabla^{\floor{s}} v(x)- \nabla^{\floor{s}} v(y))(Su)_a\|_{L^1} +  \|2\nabla^{\floor{s}} v(x)(Su)_a\|_{L^1}\;.$$

\noindent The first term on the right hand side equals $\|\nabla^s v\, (Au)_s\|_{L^1}$ by computation \eqref{maglap2}, so that it is finite by $\|\nabla^s v\, (Au)_s\|_{L^1} \leq \|\nabla^s v\|_{L^2} \|(Au)_s\|_{L^2} \leq \|v\|_{H^s} \|Au\|_{L^2}< \infty$. We estimate the other term again by triangular inequality as
\begin{equation}\label{maglap3} \|2\nabla^{\floor{s}} v(x)(Su)_a\|_{L^1} \leq \|\nabla^{\floor{s}} v(x) S(x,y)u(x)\|_{L^1} + \|\nabla^{\floor{s}} v(x)S(y,x)u(y)\|_{L^1}\;.\end{equation}
\noindent The estimation of the second term on the right hand side of \eqref{maglap3} can be done as follows, and similarly for the other one:
\begin{align}
\|\nabla^{\floor{s}} v(x)S(y,x)u(y)\|_{L^1} & = \int_{\mathbb R^n} |\nabla^{\floor{s}} v(x)| \int_{\mathbb R^n} |S(y,x)|\,|u(y)|\, dydx \notag\\ & \leq \int_{\mathbb R^n} |\nabla^{\floor{s}} v(x)|\tilde S (x) \int_{\Omega} |u(y)|\, dydx \notag\\ & \leq c \|u\|_{L^2}\int_{\mathbb R^n} |\nabla^{\floor{s}} v(x)|\tilde S (x) dx \label{maglap5}\\ & \leq c \|u\|_{L^2} \|\nabla^{\floor{s}} v(x)\|_{L^2} \|\tilde S\|_{L^2} \notag\\ & \leq c \|u\|_{H^s} \|v\|_{H^s} \|\tilde S\|_{L^{2}} < \infty,
\end{align}
where the constant $c$ can change from line to line and does not depend on $v, u$ and $S$.

\noindent Thus we have proved that $\|2\nabla^{\floor{s}} v(x)(Su)_a\|_{L^1} < \infty$, which in turn implies that $\|(\nabla^{\floor{s}} v(x)+ \nabla^{\floor{s}} v(y))(Su)_a\|_{L^1} < \infty$. Now we can use again Lemma \ref{antisymlem} to conclude that $\langle\nabla^{\floor{s}} v(x)+ \nabla^{\floor{s}} v(y), (Su)_a \rangle=0$. From this fact and formula~\eqref{maglap2}, integrating by parts,
\begin{align*}
\langle \nabla^s v, Au \rangle & = \langle \nabla^{\floor{s}} v(x)- \nabla^{\floor{s}} v(y), (S u)_a \rangle + \langle\nabla^{\floor{s}} v(x)+ \nabla^{\floor{s}} v(y), (Su)_a \rangle \\ & = 2\langle \nabla^{\floor{s}} v(x), (S u)_a \rangle = \langle \nabla^{\floor{s}}v(x) , S(x,y) u(x) - S(y,x) u(y) \rangle \\ & = \langle \nabla^{\floor{s}}v(x) , S(x,y) u(x) \rangle - \langle \nabla^{\floor{s}}v(x) , S(y,x) u(y) \rangle \\ &   = (-1)^{\floor{s}} \left\langle v , (\nabla\cdot)_x^{\floor{s}}\left( u(x) \int_{\mathbb R^n} S(x,y) dy \right) \right\rangle \\ & \;\;\;\; - (-1)^{\floor{s}} \left\langle v , (\nabla\cdot)_x^{\floor{s}} \int_{\mathbb R^n} S(y,x)u(y) dy \right\rangle\;.
\end{align*}

\noindent In the last term the derivatives can pass under the integral sign by means of the dominated convergence theorem, since $|S(x,y)u(y)| \leq \tilde S(y)|u(y)|$, and $\int_{\mathbb R^n} \tilde S(y)|u(y)| dy \leq \|\tilde S\|_{L^2}\|u\|_{L^2} < \infty$. Eventually, 
\begin{align}\label{maglap4} \langle \nabla^s v, Au \rangle & = (-1)^{\floor{s}} \left\langle v , (\nabla\cdot)_x^{\floor{s}}\left( u(x) \int_{\mathbb R^n} S(x,y) dy \right) \right\rangle \\ & \;\;\;\; + (-1)^{\floor{s}+1} \left\langle v , \int_{\mathbb R^n} u(y) (\nabla\cdot)_x^{\floor{s}} S(y,x) dy \right\rangle\;. \notag\end{align}

\para{Step 2} Next we compute the second term on the right hand side of \eqref{maglap1}. With a computation similar to \eqref{maglap2}, we obtain $\langle \nabla^s u, Av \rangle = \langle \nabla^{\floor{s}}u(x) - \nabla^{\floor{s}}u(y) , S(x,y) v(x) \rangle$; moreover, we have estimates similar to the ones in \eqref{maglap5}, and so we can split the integral. Eventually, we integrate by parts and get
\begin{align}
\langle \nabla^s u, Av \rangle & = \langle \nabla^{\floor{s}}u(x), S(x,y) v(x) \rangle - \langle \nabla^{\floor{s}}u(y) , S(x,y) v(x) \rangle \notag\\ & = \left\langle v(x), \nabla^{\floor{s}}u(x) \cdot \int_{\mathbb R^n} S(x,y) dy \right\rangle - \left\langle v(x), \int_{\mathbb R^n}  \nabla^{\floor{s}}u(y) \cdot S(x,y) dy \right\rangle \label{maglap6}\\ & = \left\langle v(x), \nabla^{\floor{s}}u(x) \cdot \int_{\mathbb R^n} S(x,y) dy \right\rangle + (-1)^{\floor{s}+1} \left\langle v(x), \int_{\mathbb R^n}  u(y) (\nabla\cdot)_y^{\floor{s}} S(x,y) dy \right\rangle\;.\notag
\end{align}

\para{Step 3} The properties $\langle (-\Delta)^s u, v\rangle = \langle \nabla^s u, \nabla^s v \rangle$ and $\langle Au, Av \rangle = \left\langle v, u \int_{\mathbb R^n} |A(x,y)|^2dy\right\rangle$ hold, as proved in \cite{CO-magnetic-fractional-schrodinger}. Using this information and formulas~\eqref{maglap4}, \eqref{maglap6} we can write the fractional magnetic Schr\"odinger operator as
\begin{align*}
&\langle (-\Delta)^s u, v\rangle + \left\langle \nabla^{\floor{s}}u(x) \cdot \int_{\mathbb R^n}S(x,y)dy + (-1)^{\floor{s}}(\nabla\cdot)_x^{\floor{s}}\left( u(x)\int_{\mathbb R^n}S(x,y)dy\right) ,v\right\rangle + \\ & + (-1)^{\floor{s}+1} \left\langle \int_{\mathbb R^n} u(y) \left( (\nabla\cdot)_x^{\floor{s}} S(y,x) + (\nabla\cdot)_y^{\floor{s}} S(x,y) \right) dy  ,v\right\rangle + \left\langle u \int_{\mathbb R^n} |A|^2 dy ,v\right\rangle.
\end{align*}

\noindent Let us compute the left hand side of the second bracket and collect the resulting terms according to the order of their derivatives of $u$. For every multi-index $\beta$ such that $|\beta|\leq \floor{s}$ we can find a linear operator $\mathfrak M_\beta$ such that $$ \nabla^{\floor{s}}u(x) \cdot \int_{\mathbb R^n}S(x,y)dy + (-1)^{\floor{s}}(\nabla\cdot)_x^{\floor{s}}\left( u(x)\int_{\mathbb R^n}S(x,y)dy\right) = \sum_{|\beta|\leq\floor{s}}\partial^\beta u(x) \mathfrak M_\beta(S)\;. $$ \noindent We can also define the following linear operator: $$ \mathfrak N(S) = (-1)^{\floor{s}+1}\left((\nabla\cdot)_x^{\floor{s}} S(y,x) + (\nabla\cdot)_y^{\floor{s}} S(x,y)\right)\;. $$ \noindent With these new definitions, we can rewrite the fractional magnetic Schr\"odinger operator as in the statement of the Lemma. \qedhere
\end{proof}

\subsection{The bilinear form and the DN map}

For every $s\in \mathbb R^+\setminus \mathbb Z$ and $u,v\in H^s$ we define the bilinear form $B^s_{A,q} : H^s \times H^s \rightarrow \mathbb R$ as in \cite{CO-magnetic-fractional-schrodinger}: \begin{equation*}
B^s_{A,q} (u,v) = \int_{\mathbb R^n}\int_{\mathbb R^n} \nabla^s_A u \cdot \nabla^s_A v \,dy dx + \int_{\mathbb R^n} quv \,dx\;.
\end{equation*}

\begin{lemma}
There are constants $\mu', k' > 0$ such that, for all $u\in H^s$, $$ B^s_{A,q} (u,u) + \mu' \langle u,u \rangle \geq k'\|u\|^2_{H^s}\;. $$
\end{lemma}

\begin{proof} 
\noindent The formula we want to prove is called \emph{coercivity estimate}. Using \eqref{maglap1}, we can write
\begin{align}
B^s_{A,q} (u,u) & =  \int_{\mathbb R^n}\int_{\mathbb R^n} \nabla^s_A u \cdot \nabla^s_A u \,dy dx + \int_{\mathbb R^n} qu^2 \,dx \notag\\ & = \int_{\mathbb R^n} u(-\Delta)^s_A u \,dx + \int_{\mathbb R^n} qu^2 \,dx = \langle (-\Delta)^s_A u, u \rangle + \langle qu,u\rangle \notag\\ & = \langle(-\Delta)^s u, u\rangle + 2\langle \nabla^s u, Au \rangle + \left\langle \left(q + \int_{\mathbb R^n}|A(x,y)|^2 dy\right)u,u\right\rangle \notag\\ & = \langle(-\Delta)^s u, u\rangle + 2\left\langle \int_{\mathbb R^n}\nabla^s u\cdot A \,dy, u \right\rangle + \left\langle Qu,u\right\rangle\label{bilin} \;,
\end{align}
\noindent where $Q(x) := q(x) + \int_{\mathbb R^n}|A(x,y)|^2 dy$ belongs to $L^p$ since Cauchy-Schwartz and assumptions \ref{assumption1} and \ref{assumption3} imply the embedding $L^{2p}\times L^{2p}\hookrightarrow L^p$. Since we always have $L^{p}\times H^s \hookrightarrow H^{-s}$, we get $\left\langle Qu,u\right\rangle \leq \|u\|_{H^s}\|Qu\|_{H^{-s}} \leq \|Q\|_{L^{p}}\|u\|^2_{H^s}$. For the second term on the right hand side of \eqref{bilin} we first perform an estimate by means of the Young inequality
\begin{align*}
2\left\langle \int_{\mathbb R^n}\nabla^s u\cdot A\, dy, u \right\rangle \leq \epsilon^{-1} \|u\|^2_{L^2} + \epsilon \left\| \int_{\mathbb R^n} \nabla^s u \cdot A \, dy \right\|^2_{L^2}\;,
\end{align*}
\noindent then estimate the second term with the Cauchy-Schwartz inequality, in light of \ref{assumption4}:
\begin{align*}
\epsilon \left\| \int_{\mathbb R^n} \nabla^s u \cdot A \, dy \right\|^2_{L^2} & = \epsilon \left\| \int_{\mathbb R^n} \left((\nabla^{\floor{s}} u(x) - \nabla^{\floor{s}} u(y))\otimes \alpha \right)\cdot A \, dy \right\|^2_{L^2} \\ & = \epsilon \left\| \int_{\mathbb R^n} (\nabla^{\floor{s}} u(x) - \nabla^{\floor{s}} u(y))\cdot (A\cdot\alpha) \, dy \right\|^2_{L^2} \\ & = \epsilon \left\| \int_{\Omega} (\nabla^{\floor{s}} u(x) - \nabla^{\floor{s}} u(y))\cdot S(x,y) \, dy \right\|^2_{L^2(\Omega)} \\ & \leq \epsilon \left\| \left(\int_{\Omega} |\nabla^{\floor{s}} u(x) - \nabla^{\floor{s}} u(y)|^2 dy\right)^{1/2} \left( \int_{\Omega} |S(x,y)|^2 \, dy\right)^{1/2} \right\|^2_{L^2(\Omega)} \\ & = \epsilon \int_{\Omega} \left(\int_{\Omega} |\nabla^{\floor{s}} u(x) - \nabla^{\floor{s}} u(y)|^2 dy\,\int_{\Omega} |S(x,y)|^2 \, dy\right)dx \\ & \leq  \epsilon \int_{\Omega} \left(\int_{\Omega} (|\nabla^{\floor{s}} u(x)|+ |\nabla^{\floor{s}} u(y)|)^2 dy\,\int_{\Omega} \tilde S^2(y) \, dy\right)dx \\ & = \epsilon \|\tilde S\|^2_{L^2(\Omega)} \int_{\Omega}\int_{\Omega} (|\nabla^{\floor{s}} u(x)|+ |\nabla^{\floor{s}} u(y)|)^2 dydx \\ & \leq 2 \epsilon \|\tilde S\|^2_{L^2(\Omega)} \int_{\Omega}\int_{\Omega} (|\nabla^{\floor{s}} u(x)|^2+ |\nabla^{\floor{s}} u(y)|^2) dydx \\ & \leq 4 |\Omega| \epsilon \|\tilde S\|^2_{L^2(\Omega)} \|\nabla^{\floor{s}} u\|^2_{L^2} \leq c\, \epsilon  \|\nabla^{\floor{s}} u\|^2_{H^{s'}} \leq c\, \epsilon  \|u\|^2_{H^s},
\end{align*}
where the constant $c$ can change from line to line and does not depend on $u$.

\noindent Eventually 
$$2\left\langle \int_{\mathbb R^n}\nabla^s u\cdot A\, dy, u \right\rangle \leq \epsilon^{-1} \|u\|^2_{L^2} + c\, \epsilon  \|u\|^2_{H^s},$$ which leads to \begin{equation}\label{precoerc}B^s_{A,q} (u,u) \geq B^s_{0,Q}(u,u) - \epsilon^{-1} \|u\|^2_{L^2} - c\, \epsilon  \|u\|^2_{H^s}\;.\end{equation}
\noindent Since $C^{\infty}_c(\Omega)$ is dense in $ L^p(\Omega)$, for every $\delta>0$ we can find functions $Q_s, Q_r$ such that $ Q_s \in C^{\infty}_c(\Omega)$, $\|Q_r\|_{L^p(\Omega)}\leq \delta$ and $Q=Q_s+Q_r$. Also, if $\phi_j\in C^{\infty}_c(\Omega)$ and $\|\phi_j\|_{H^s} = 1$ for $j=1,2$, then $|\langle Q_r \phi_1, \phi_2 \rangle| \leq c \|\phi_1\|_{H^s} \|\phi_2\|_{H^s} \|Q_r\|_{L^p} \leq c\delta$ by the embedding $L^p\times H^s \hookrightarrow H^{-s}$. Therefore, $$ \|Q_r\|_{Z^{-s}} = \sup_{\|\phi_j\|_{H^s}=1} \{ |\langle Q_r \phi_1, \phi_2 \rangle| \} \leq c\delta \;,$$ \noindent which means that $Q$ belongs to the closure of $C^\infty_c(\Omega)$ in $Z^{-s}(\mathbb R^n)$, that is $Q\in Z^{-s}_0(\mathbb R^n)$. Now by Lemma \ref{lemma:schrodingerexistenceofsolutions} we know the coercivity estimate for the non-magnetic high exponent case; this lets us write \eqref{precoerc} as $$B^s_{A,q} (u,u) +  (\mu+\epsilon^{-1}) \langle u,u \rangle \geq (k-c\, \epsilon)\|u\|^2_{H^s}\;,$$
\noindent which is the coercivity estimate for $B^s_{A,q}$ as soon as $\epsilon$ is fixed small enough and $\mu' := \mu+\epsilon^{-1}$, $k' := k-c\, \epsilon$ are defined.  \qedhere
\end{proof}

By means of the lemma above, if we assume $0$ is not an eigenvalue for the equation, we can proceed as in the proof of Lemma 2.6 from \cite{RS-fractional-calderon-low-regularity-stability} and get the well-posedness of the direct problem for FMSE. This can be stated as follows: if $F \in (\widetilde H^s(\Omega))^*$, there exists unique solution $u \in H^s(\R^\dimens)$ to $B^s_{A,q}(u,v)=F(v)$ for all $v\in \widetilde H^s(\Omega)$, i.e. unique $u\in H^s(\R^\dimens)$ such that $(-\Delta)^s_A u +qu =F$ in $\Omega$, $u|_{\Omega_e}=0$. This is also true for non-vanishing exterior value $f\in H^s(\R^\dimens)$ (see \cite{Co18} and \cite{GSU-calderon-problem-fractional-schrodinger}), and the following estimate holds:   
\begin{equation}\label{estdirect}\|u\|_{H^s(\mathbb R^n)} \leq c(\|F\|_{(\widetilde H^s(\Omega))^*} + \|f\|_{H^s(\mathbb R^n)}),\end{equation}
where $c$ does not depend on $F$ and $f$.

\noindent One can prove (see Lemma 3.11 from \cite{CO-magnetic-fractional-schrodinger}) that $B^s_{A,q}$ also enjoys these properties:

\begin{enumerate}
\item $B^s_{A,q}(v,w) = B^s_{A,q}(w,v)\;$, for all $v,w\in H^s$, 
\item $|B^s_{A,q}(v,w)| \leq c\|v\|_{H^s(\mathbb R^n)}\|w\|_{H^s(\mathbb R^n)}\;$ for all $v,w\in H^s$, where $c$ does not depend on $v$ and $w$.
\item $B^s_{A,q}(u_1,e_2) = B^s_{A,q}(u_2,e_1)\;$, for $u_j\in H^s$ solution to the direct problem for FMSE with exterior value $f_j\in H^s(\Omega_e)$ and $e_j$ any extension of $f_j$ to $H^s$, $j=1,2$.
\end{enumerate} 

\begin{lemma}\label{DNLem}
Let $X=H^s(\mathbb R^n)/\widetilde H^s(\Omega)$ be the abstract quotient space, and let $u_1\in H^s$ be the solution to the direct problem for FMSE with exterior value $f_1\in H^s(\Omega_e)$. Then $$\langle \Lambda_{A,q}^s [f_1],[f_2] \rangle = B^s_{A,q}(u_1,f_2), \;\;\;\;\;\;\; f_j\in H^s,\; j=1,2$$ \noindent defines a bounded, linear, self-adjoint map $\Lambda_{A,q}^s : X\rightarrow X^*$. We call $\Lambda_{A,q}^s$ the \emph{DN map}. 
\end{lemma}  

\begin{proof} The proof follows trivially from properties (1)-(3) of $B_{A, q}^s$ and \eqref{estdirect}. \qedhere \end{proof}

\subsection{The gauge}

\noindent Consider two couples of potentials $(A_1, q_1)$ and $(A_2, q_2)$. We say that $(A_1, q_1)\sim(A_2, q_2)$ if and only if the following conditions are met:
\begin{itemize}
\item $\mathfrak N(S_1-S_2)=0$ for almost every $x,y\in\mathbb R^n$
\item $\mathfrak M_{(0,...,0)}(S_1-S_2) + \int_{\mathbb R^n} (|A_1|^2-|A_2|^2) dy + (q_1-q_2) =0$ for almost every $x\in\mathbb R^n$
\item $\mathfrak M_\beta(S_1-S_2)=0$ for all $1\leq|\beta|\leq\floor{s}$ and almost every $x\in\mathbb R^n.$
\end{itemize}

\noindent It is clear from the linearity of $\mathfrak N$ and $\mathfrak M_\alpha$ that $\sim$ is an equivalence relation, and so the set of all couples of potentials is divided into equivalence classes by $\sim$. We call these \emph{gauge classes}, and if $(A_1, q_1)\sim(A_2, q_2)$ we say that $(A_1, q_1)$ and $(A_2, q_2)$ \emph{are in gauge}.

\noindent Observe that this gauge $\sim$ coincides with the one defined in \cite{CO-magnetic-fractional-schrodinger} if $s\in(0,1)$, although it looks quite different. Since in this case $\floor{s}=0$, there is no third condition. In the language of that paper, the first condition reads
\begin{align*}
0 & = -\mathfrak N(S_1-S_2) = S_1(y,x) + S_1(x,y) - S_2(y,x) - S_2(x,y) \\ & = (A_1(x,y)-A_2(x,y))\cdot\alpha(x,y) + (A_1(y,x)-A_2(y,x))\cdot\alpha(y,x)  \\ & = (A_1(x,y)-A_1(y,x) - A_2(x,y)+A_2(y,x))\cdot\alpha(x,y) \\ & = 2(A_1-A_2)_a\cdot \alpha = 2 (A_1-A_2)_{a\parallel}\cdot\alpha \;, 
\end{align*}
\noindent which is equivalent to $(A_1)_{a\parallel} = (A_2)_{a\parallel}$, since the two vectors in the last scalar product have the same direction. Given this fact, for any $v\in H^s$ the first term in the second condition weakly is
\begin{align*}
\langle\mathfrak M_{(0,...,0)}&(S_1-S_2), v \rangle = 2 \langle S_1-S_2, v\rangle = 2 \langle \alpha\cdot(A_1-A_2), v\rangle = 2 \langle \alpha\cdot(A_1-A_2)_\parallel, v\rangle \\ & = 2 \langle \alpha\cdot(A_1-A_2)_{s\parallel}, v\rangle = 2 \langle \alpha v, (A_1-A_2)_{s\parallel}\rangle = 2 \langle (\alpha v)_s, (A_1-A_2)_{s\parallel}\rangle \\ & = \langle \alpha(x,y) v(x) + \alpha(y,x) v(y), (A_1-A_2)_{s\parallel}\rangle \\ & = \langle \alpha(x,y)( v(x) -v(y)), (A_1-A_2)_{s\parallel}\rangle \\ & = \langle \nabla^s v, (A_1-A_2)_{s\parallel}\rangle = \langle v, (\nabla\cdot)^s((A_1-A_2)_{s\parallel}) \rangle\;,
\end{align*}
\noindent which lets us rewrite the second condition as $$(\nabla\cdot)^s(A_1)_{s\parallel} + \int_{\mathbb R^n} |A_1|^2 dy + q_1 = (\nabla\cdot)^s(A_2)_{s\parallel} + \int_{\mathbb R^n} |A_2|^2 dy + q_2\;.$$

\begin{remark}
Observe that the gauge enjoyed by the FMSE is quite different from the one holding for the MSE. For the sake of simplicity, we shall compare the classical case with the fractional one in the regime $s\in(0,1)$, following section 3 in \cite{CO-magnetic-fractional-schrodinger}.

Given lemma \ref{expanded}, one sees that the following is an equivalent definition for the gauge $\sim$ above:
$$ (A_1,q_1)\sim(A_2,q_2) \quad \Leftrightarrow \quad (-\Delta)^s_{A_1} u + q_1 u = (-\Delta)^s_{A_2} u + q_2 u\;,$$

\noindent for all $u\in H^s(\mathbb R^n)$. One may also define the accessory gauge $\approx$ as
$$ (A_1,q_1)\approx(A_2,q_2) \quad \Leftrightarrow \quad \exists\phi\in G : (-\Delta)^s_{A_1} (u\phi) + q_1 u\phi = \phi((-\Delta)^s_{A_2} u + q_2 u)\;,$$

\noindent for all $u\in H^s(\mathbb R^n)$, where $G:=\{ \phi\in C^\infty(\mathbb R^n) : \phi > 0, \phi|_{\Omega_e}=1 \}$. These definitions can be extended to the MSE in the natural way. It was proved in lemmas 3.9 and 3.10 of \cite{CO-magnetic-fractional-schrodinger} that the FMSE enjoys the gauge $\sim$, but not $\approx$. In the same discussion, it was argued that the opposite holds for MSE. The reason for this surprising discrepancy should be looked for in the nonlocal structure of the FMSE. As apparent in formula (10) in \cite{CO-magnetic-fractional-schrodinger}, the coefficient of the gradient term in FMSE is not related to the whole vector potential $A$ itself, but only to its antisymmetric part $A_a$. It is such antisymmetry requirement what eventually does not allow the FMSE to enjoy $\approx$ as the MSE. As a result, the scalar potential $q$ can not be in general uniquely determined as in the classical case.  
\end{remark}

\subsection{Main result} \begin{remark}\label{WUCPrem} Assume $W \subseteq \Omega_e$ is an open set and $u\in H^s$ satisfies $u=0$ and $(-\Delta)^s_A u + qu = 0$ in $W$. We say that the fractional magnetic Schr\"odinger operator enjoys the weak unique continuation property (WUCP) if we can deduce that $u=0$ in $\Omega$. This was proved in \cite{CO-magnetic-fractional-schrodinger} by using the UCP of the fractional Laplacian for $s\in(0,1)$; since we know by Theorem \ref{thm:uniquecontinuationoffractionallaplacian} that UCP still holds for $(-\Delta)^s$ in the regime $s\in \mathbb R^+\setminus \mathbb Z$, we can deduce WUCP for $(-\Delta)^s_A + q$ by the same proof. \end{remark}    

\begin{proof}[Proof of theorem \ref{thm:FMSE}.]
\para{Step 1} Without loss of generality, let $W_1 \cap W_2 = \emptyset$. Let $f_i \in C^\infty_c(W_i)$, and let $u_i \in H^s(\mathbb R^n)$ solve $(-\Delta)^s_{A_i} u_i + q_i u_i = 0$ with $u_i - f_i \in \widetilde H^s(\Omega)$ for $i=1,2$. Knowing that the DN maps computed on $f\in C^\infty_c(W_1)$ coincide when restricted to $W_2$, using Lemmas \ref{expanded} and \ref{DNLem} we write this integral identity
\begin{align*}
0 & = \langle (\Lambda^s_{A_1,q_1} - \Lambda^s_{A_2,q_2}) f_1, f_2 \rangle = B^s_{A_1,q_1}(u_1,u_2) - B^s_{A_2,q_2}(u_1,u_2) \\ & = \left\langle u_2, \sum_{|\beta|\leq\floor{s}}\partial^\beta u_1 \mathfrak M_\beta(S_1-S_2)\right\rangle + \left\langle u_2, \int_{\mathbb R^n} u_1(y) \mathfrak N(S_1-S_2)\,dy \right\rangle + \\ & \;\;\;\;\;\; +  \left\langle u_2, u_1 \left( \int_{\mathbb R^n} (|A_1|^2-|A_2|^2) dy + (q_1-q_2)\right)\right\rangle.
\end{align*}

\noindent Since if $x\not\in\Omega$ or $y\not\in\Omega$ we have $A_1(x,y) = A_2(x,y)$ and $q_1(x)=q_2(x)$, we can restrict $u_1$, $u_2$ and $\partial^\beta u_1$ over $\Omega$ in the previous formula; it is also true that $(\partial^\beta u_1)|_\Omega = \partial^\beta (u_1|_\Omega)$, and therefore  
\begin{align*}
0 & = \left\langle u_2|_\Omega, \sum_{|\beta|\leq\floor{s}}\partial^\beta (u_1|_\Omega) \mathfrak M_\beta(S_1-S_2)\right\rangle + \left\langle u_2|_{\Omega}, \int_{\mathbb R^n} u_1|_{\Omega}(y) \mathfrak N(S_1-S_2)\,dy \right\rangle + \\ & \;\;\;\;\;\; +  \left\langle u_2|_{\Omega}, u_1|_{\Omega} \left( \int_{\mathbb R^n} (|A_1|^2-|A_2|^2) dy + (q_1-q_2)\right)\right\rangle.
\end{align*}
\noindent This is the Alessandrini identity, which now we will test with certain solutions in order to obtain information about the potentials. The appropriate test solutions will be produced by means of the Runge approximation property (RAP), which holds for the FMSE because of Remark \ref{WUCPrem} and Lemma 3.15 in \cite{CO-magnetic-fractional-schrodinger}. This property says that the set $\mathcal R = \{ u_f|_{\Omega}: f \in C^\infty_c(W) \} \subset L^2(\Omega)$ of the restrictions to $\Omega$ of those functions $u_f$ solving FMSE for some smooth exterior value $f$ supported in $W$ is dense in $L^2(\Omega)$.

\para{Step 2} Given any $f\in L^2(\Omega)$, by the RAP we can find a sequence of solutions $(u_2)_k \rightarrow f$ in $L^2$ sense as $k\rightarrow\infty$. Substituting these in the Alessandrini identity and taking limits, by the arbitrarity of $f$ we can deduce that
\begin{align*}
0 & = \sum_{|\beta|\leq\floor{s}} \partial^\beta (u_1|_{\Omega}) \mathfrak M_\beta(S_1-S_2)  + \int_{\mathbb R^n} u_1|_{\Omega}(y) \mathfrak N(S_1-S_2)\,dy + \\ & \;\;\;\;\;\; +  u_1|_{\Omega} \left( \int_{\mathbb R^n} (|A_1|^2-|A_2|^2) dy + (q_1-q_2)\right)
\end{align*}
\noindent holds for every solution $u_1\in H^s$ and almost every point $x\in\Omega$. Fix $x\in\Omega$. Consider now any $\psi\in C^\infty_c(\Omega)$ and let $g(y):= e^{-1/|x-y|}\psi(y)$, $g(x)=0$. Since $e^{-1/|x-y|}$ is smooth, it is easy to see that $g\in C^\infty_c(\Omega) \subset L^2(\Omega)$; also, by the properties of $e^{-1/|x-y|}$ one has that $\partial^\beta g(x)=0$ for all multi-indices $\beta$. By the RAP we can find a sequence of solutions $(u_1)_k \rightarrow g$ in $L^2$ sense as $k\rightarrow\infty$. Substituting these in the above identity and taking limits, we get 
$$  \int_{\mathbb R^n} e^{-1/|x-y|}\psi(y) \mathfrak N(S_1-S_2)\,dy=0 \;,$$
\noindent which by the arbitrarity of $\psi$ and the positivity of the exponential now implies $\mathfrak N(S_1-S_2)=0$ for almost all $x,y\in\Omega$. We can now return to the above equation with this new information: for every solution $u_1\in H^s$ and almost every $x\in\Omega$, 
\begin{align*}
0 & = \sum_{|\beta|\leq\floor{s}} \partial^\beta (u_1|_{\Omega}) \mathfrak M_\beta(S_1-S_2) +  u_1|_{\Omega} \left( \int_{\mathbb R^n} (|A_1|^2-|A_2|^2) dy + (q_1-q_2)\right)\;.
\end{align*}
\noindent For every multi-index $\beta$ we can consider the function $h_\beta (x)= x^\beta=x_1^{\beta_1}\dotso x_n^{\beta_n}$, which belongs to $L^2(\Omega)$. Let $(h_\beta)_k$ be a sequence of solutions approximating $h_\beta$ in $L^2$, which exists by the RAP. We will first substitute $(h_{(0,...,0)})_k$ into the last formula, take limits and deduce
$$ \mathfrak M_{(0,...,0)}(S_1-S_2) + \int_{\mathbb R^n} (|A_1|^2-|A_2|^2) dy + (q_1-q_2) =0 \;,$$
which has the effect of reducing the equation to 
$$ \sum_{1\leq|\beta|\leq\floor{s}} \partial^\beta (u_1|_{\Omega}) \mathfrak M_\beta(S_1-S_2)=0.$$ If $\floor{s}\geq 1$, we will repeat the last steps with each $h_\beta$ such that $|\beta|=1$, deducing $\mathfrak M_\beta(S_1-S_2)=0$ for every such $\beta$, and subsequently $$\sum_{2\leq|\beta|\leq\floor{s}} \partial^\beta (u_1|_{\Omega}) \mathfrak M_\beta(S_1-S_2)=0.$$ Repeating this process for a total of $\floor{s}$ times eventually leads to $$\mathfrak M_\beta(S_1-S_2)=0 \quad\forall \ 1\leq|\beta|\leq\floor{s}\;,$$
\noindent which proves the theorem by the definition of the gauge $\sim$.
\end{proof}

\section{Possible generalizations and applications beyond this article} 
\label{sec:possiblegeneralizations}

We discuss some possible directions for the future research on higher order fractional inverse problems, fractional Poincar\'e inequalities and unique continuation properties. It seems that now it would be the most natural to reconsider many of the recent developments in fractional inverse problems for higher order operators. We outline here some problems which we would like to see solved in the future.

We have split this section in three in order to emphasize some open problems which we find especially interesting. We do not claim that answers to all questions are positive and it would be interesting to see why and where the greatest difficulties, or even counterexamples, would show up. We first list the most natural directions to continue our work on higher order fractional Calder\'on problems. One could study for example the following cases: \begin{enumerate}[(i)]
    \item Is reconstruction from a single measurement \cite{Co18,GRSU-fractional-calderon-single-measurement} possible also in the higher order cases?
    \item Is there stability \cite{RS-fractional-calderon-low-regularity-stability} in the higher order cases?
    \item Is there exponential instability \cite{RS18} in the higher order cases?
    \item Is there uniqueness for the Calder\'on problem for fractional semilinear Schrödinger equations \cite{LL19, LL-fractional-semilinear-problems} in the higher order cases?
    \item Do the monotonicity methods \cite{HL19a,HL19b} extend to the higher order cases?
    \item Is there uniqueness for the conductivity type fractional Calder\'on problems \cite{CLL19,Co18} in the higher order cases?
    \item Could recent results on fractional heat equations \cite{LLR19,ruland2019quantitative} be generalized to the higher order cases?
    \item Does the higher regularity Runge approximation in \cite{CLR18, GSU-calderon-problem-fractional-schrodinger} generalize to higher order cases?
\end{enumerate}

\subsection{Unique continuation problems}

We state here some unique continuation problems, which do not follow directly from the earlier results and the techniques that we have developed for this article.

\begin{question}[UCP for Bessel potentials]\label{prob:besselUCP} Let $s \in \R^+ \setminus \Z$, $p \in [1,\infty)$ and $r \in \R$. Let $V \subset \R^n$ be an open set. Suppose that $f \in H^{r,p}(\R^n)$, $f|_V = 0$ and $\fraclaplace f|_V = 0$. Show that $f \equiv 0$ or give a counterexample.
\end{question}

The positive answer to question \ref{prob:besselUCP} is known when $p \in [1,2]$ (see corollary \ref{cor:besselucp}). If $f$ has compact support, then the answer is positive for all $p\in [1, \infty)$ (see corollary \ref{cor:stronguniquecontinuation}). Question \ref{prob:besselUCP} is also open for the exponents $s \in (0,1)$ when $p>2$. See section \ref{subsec:uniquecontinuationresults} for details.

\begin{question}[Measurable UCP]\label{prob:measUCP} Let $s \in \R^+ \setminus \Z$ and $r \in \R$. Let $V \subset \R^n$ be an open set and $E \subset V$ a measurable set with positive measure. Suppose that $f \in H^{r}(\R^n)$, $f|_E = 0$ and $\fraclaplace f|_V = 0$. Show that $f \equiv 0$ or give a counterexample.
\end{question}

The positive answer to question \ref{prob:measUCP} is known when $s \in (0,1)$ \cite{GRSU-fractional-calderon-single-measurement}. Question \ref{prob:measUCP} with a potential $q$ from a suitable class of functions is also an interesting and more challenging problem. See \cite[Proposition 5.1]{GRSU-fractional-calderon-single-measurement} for more details.

\begin{question}[Alternative strong UCP]\label{prob:strongUCP}
Let $s\in\R^+\setminus\Z$ and $r\in\R$. Let $V\subset\R^\dimens$ be an open set. Suppose that $f\in H^{r}(\R^\dimens)$, $f|_V=0$ and $\partial^{\beta}(\fraclaplace f)(x_0)=0$ for some $x_0\in V$ and all $\beta\in \N_0^\dimens$. Show that $f\equiv 0$ or give a counterexample.
\end{question}

Question \ref{prob:strongUCP} can be seen as a version of the strong unique continuation property (see e.g. \cite{FF-unique-continuation-fractional-ellliptic-equations, GR-fractional-laplacian-strong-unique-continuation, RU-unique-continuation-scrodinger-rough-potentials}) with interchanged decay conditions. When $f$ has compact support, the answer to question \ref{prob:strongUCP} is positive for $s\in(-\dimens/2, \infty)\setminus\Z$ (see corollary \ref{cor:stronguniquecontinuation}).

The problems posed in questions \ref{prob:besselUCP}--\ref{prob:strongUCP} for the fractional Laplacian are interesting mathematical problems on their own right, but they also have important applications in inverse problems. The UCPs can be used to show Runge approximation properties for nonlocal equations such as the fractional Schr\"odinger equation (see theorem \ref{thm:schrodingerrungeapproximation}), which in turn can be used to show uniqueness for the corresponding nonlocal inverse problem (see theorem \ref{thm:schrodingeruniqueness}). The UCPs have also applications in integral geometry, where the uniqueness of the ROI problem for the $d$-plane transform can be reduced to a unique continuation problem for the fractional Laplacian (see remark \ref{remark:ucpofdnormaloperator} and corollaries \ref{cor:uniquecontinuationnormaloperator} and \ref{cor:partialdataresult}).


\subsection{Fractional Poincar\'e inequality for $L^p$-norms}
In section \ref{subsec:poinacareinequality} we prove the fractional Poincar\'e inequality for $L^2$-norms in multiple ways. The inequality is needed for the well-posedness of the inverse problem for the fractional Schr\"odinger equation. One could try to extend the Poincar\'e inequality for general $L^p$-norms. This suggests the following natural question which is also interesting from the pure mathematical point of view.

\begin{question}
\label{prob:poincare}
Let $s\geq 0$, $1\leq p<\infty$, $K\subset\R^\dimens$ compact set and $u\in H^{s, p}(\R^\dimens)$ such that $\spt(u)\subset K$. Show that there exists a constant $c=c(n, K, s, p)$ such that
\begin{equation}
\label{eq:poincarebessel}
\aabs{u}_{L^p(\R^\dimens)}\leq c\aabs{(-\Delta)^{s/2}u}_{L^p(\R^\dimens)}
\end{equation}
or give a counterexample.
\end{question}

Since we have presented several proofs for the Poincar\'e inequality in the case $p=2$, one could try some of our methods to solve question \ref{prob:poincare}. However,  some of our proofs are heavily based on Fourier analysis and those approaches might be difficult to generalize to the~$L^p$-case when $p\neq 2$. Like in theorem~\ref{thm:generalpoincare} and in theorem~\ref{thm:poincareinterpolation}, another interesting question is whether one can replace~$u$ in the left-hand side of equation~\eqref{eq:poincarebessel} with $(-\Delta)^{t/2}u$ when $0\leq t\leq s$, and whether the constant~$c$ in equation \eqref{eq:poincarebessel} can be expressed in terms of the classical Poincar\'e constant when $s\geq 1$.

\subsection{The Calder\'on problem for determining a higher order PDO}

In this discussion, we try to make as simple assumptions as possible. The whole point is to introduce a new inverse problem that we think is a very natural and interesting one, at least from a pure mathematical point of view. Therefore the optimal regularity in the statement of the problem is not as important. Let $\Omega$ be a domain with smooth boundary. Suppose that $P(x,D) = \sum_{\abs{\alpha}\leq m} a_\alpha(x) D^\alpha$ is a partial differential operator (PDO) of order $m$ with smooth coefficients on $\Omega$. We argue in section \ref{subsec:uniquecontinuationresults} that the operator $(-\Delta)^s + P(x,D)$ admits the UCP (in open sets).

It is shown in the seminal work of Ghosh, Uhlmann and Salo \cite{GSU-calderon-problem-fractional-schrodinger} that if $P(x,D)$ is of order $m = 0$, then one can determine the zeroth order coefficient (i.e. the potential $q$) from the associated DN map. It was then later shown in \cite{CLR18} that if $P(x,D)$ is of order $m = 1$, then one can also determine the coefficients (i.e. the potential $q$ and the magnetic drift $b$) from the associated DN map whenever the order of $\fraclaplace$ is large enough, namely when $2s > 1$. This and our work on higher order Calder\'on problems motivate the following inverse problem.

\begin{question} Suppose that $\Omega \subset \R^n$ is a bounded open domain with smooth boundary. Let $P_j(x,D)$, $j = 1,2$, be smooth PDOs of order $m \in \N$ in $\Omega$. Let $s\in\R^+\setminus\Z$ be such that $2s > m$. Given any two open sets $W_1,W_2 \subset \Omega_e$, suppose that the DN maps $\Lambda_{P_i}$ for the equations
\[(\fraclaplace + P_j(x,D))u_j = 0 \;\;\mbox{in}\; \Omega\]
satisfy $\Lambda_{P_1}f|_{W_2}=\Lambda_{P_2}f|_{W_2}$ for all $f \in C_c^\infty(W_1)$. Show that $P_1(x,D) = P_2(x,D)$ or give a counterexample.
\end{question}

Another interesting question is whether the strong UCP \cite{GR-fractional-laplacian-strong-unique-continuation} can be extended to higher order PDOs.

\bibliography{refs} 

\begin{thebibliography}{10}

\bibitem{AB-psidos-and-singular-integrals}
H.~{Abels}.
\newblock {Pseudodifferential and Singular Integral Operators}.
\newblock {De Gruyter}, {First} edition, 2012.

\bibitem{A11}
A.~Abouelaz.
\newblock The {$d$}-plane {R}adon transform on the torus {$\Bbb T^n$}.
\newblock {\em Fract. Calc. Appl. Anal.}, 14(2):233--246, 2011.

\bibitem{VMRTM-nonlocal-diffusion-problems}
F.~{Andreu-Vaillo}, J.~M. {Maz\'on}, J.~D. {Rossi}, and J.~J. {Toledo-Melero}.
\newblock {Nonlocal Diffusion Problems}.
\newblock {American Mathematical Society}, {{F}irst} edition, 2010.

\bibitem{BCD-fourier-analysis-nonlinear-pde}
H.~{Bahouri}, J.-Y. {Chemin}, and R.~{Danchin}.
\newblock {Fourier Analysis and Nonlinear Partial Differential Equations}.
\newblock {Springer}, {First} edition, 2011.

\bibitem{BH2017}
A.~Behzadan and M.~Holst.
\newblock Multiplication in {Sobolev} spaces, revisited.
\newblock 2017.
\newblock arXiv:1512.07379v2.

\bibitem{BL-interpolation-spaces}
J.~Bergh and J.~L\"ofstr\"om.
\newblock {Interpolation Spaces, An Introduction}.
\newblock Springer-Verlag, {First} edition, 1976.

\bibitem{BV-nonlocal-diffusion-applications}
C.~{Bucur} and E.~{Valdinoci}.
\newblock {Nonlocal Diffusion and Applications}.
\newblock {Springer}, {{F}irst} edition, 2016.

\bibitem{CS-nonlinera-equations-fractional-laplacians}
X.~Cabr\'e and Y.~Sire.
\newblock {Nonlinear equations for fractional Laplacians, I: Regularity,
  maximum principles, and Hamiltonian estimates}.
\newblock {\em Annales de l'I.H.P. Analyse Non Lin\'eaire}, 31(1):23--53, 2014.

\bibitem{CS-extension-problem-fractional-laplacian}
L.~{Caffarelli} and L.~{Silvestre}.
\newblock {An Extension Problem Related to the Fractional Laplacian}.
\newblock {\em Communications in Partial Differential Equations}, 32, 2006.

\bibitem{CLL19}
X.~Cao, Y.-H. Lin, and H.~Liu.
\newblock Simultaneously recovering potentials and embedded obstacles for
  anisotropic fractional {S}chr\"{o}dinger operators.
\newblock {\em Inverse Probl. Imaging}, 13(1):197--210, 2019.

\bibitem{CLR18}
M.~Cekic, Y.-H. Lin, and A.~Ruland.
\newblock {The Calder{\'o}n problem for the fractional Schr{\"o}dinger equation
  with drift}.
\newblock {\em Calculus of Variations and Partial Differential Equations},
  59(3):91, 2020.

\bibitem{CWHM-sobolev-spaces-on-non-lipchtiz-domains}
S.~N. Chandler-Wilde, D.~P. Hewett, and A.~Moiola.
\newblock {Sobolev Spaces on Non-Lipschitz Subsets of $\mathbb{R}^n$ with
  Application to Boundary Integral Equations on Fractal Screens}.
\newblock {\em Integral Equations and Operator Theory}, 87(2):179--224, 2017.

\bibitem{CNDK-solving-interior-problem-ct-with-apriori-knowledge}
M.~Courdurier, F.~Noo, M.~Defrise, and H.~Kudo.
\newblock Solving the interior problem of computed tomography using \textit{a
  priori} knowledge.
\newblock {\em Inverse Problems}, 24(6):065001, 2008.

\bibitem{CO-magnetic-fractional-schrodinger}
G.~Covi.
\newblock {An inverse problem for the fractional Schr\"odinger equation in a
  magnetic field}.
\newblock {\em Inverse Problems}, 36(4):045004, 2020.

\bibitem{Co18}
G.~Covi.
\newblock Inverse problems for a fractional conductivity equation.
\newblock {\em Nonlinear Analysis}, 193:111418, 2020.
\newblock Nonlocal and Fractional Phenomena.

\bibitem{DE03}
A.~D'Agnolo and M.~Eastwood.
\newblock Radon and {F}ourier transforms for $\mathcal{D}$-modules.
\newblock {\em Adv. Math.}, 180(2):452--485, 2003.

\bibitem{DSV-all-functions-are-s-harmonic}
S.~Dipierro, O.~Savin, and E.~Valdinoci.
\newblock {All functions are locally s-harmonic up to a small error}.
\newblock {\em Journal of the European Mathematical Society}, 19(4):957--966,
  2017.

\bibitem{DGLZ2012}
Q.~Du, M.~Gunzburger, R.~B. Lehoucq, and K.~Zhou.
\newblock Analysis and {Approximation} of {Nonlocal} {Diffusion} {Problems}
  with {Volume} {Constraints}.
\newblock {\em SIAM Rev.}, 54, No. 4:667--696, 2012.

\bibitem{DGLZ2013}
Q.~Du, M.~Gunzburger, R.~B. Lehoucq, and K.~Zhou.
\newblock {A nonlocal vector calculus, nonlocal volume-constrained problems,
  and nonlocal balance laws}.
\newblock {\em {Math. Models Methods Appl. Sci.}}, 23, No. 3:493--540, 2013.

\bibitem{Es01}
G.~Eskin.
\newblock {Global uniqueness in the inverse scattering problem for the
  Schr\"odinger operator with external Yang-Mills potentials}.
\newblock {\em Communications in Mathematical Physics}, 222(3):503--531, 2001.

\bibitem{Eva10}
L.~C. Evans.
\newblock {\em Partial differential equations}, volume~19 of {\em Graduate
  Studies in Mathematics}.
\newblock American Mathematical Society, Providence, RI, second edition, 2010.

\bibitem{FF-unique-continuation-fractional-ellliptic-equations}
M.~M. Fall and V.~Felli.
\newblock Unique continuation property and local asymptotics of solutions to
  fractional elliptic equations.
\newblock {\em Comm. Partial Differential Equations}, 39(2):354--397, 2014.

\bibitem{FF-unique-continuation-higher-laplacian}
V.~{Felli} and A.~{Ferrero}.
\newblock {Unique continuation principles for a higher order fractional Laplace
  equation}.
\newblock {\em Nonlinearity}, 33(8):4133--4190, 2020.

\bibitem{FS-the-uncertainty-principle}
G.~B. Folland and A.~Sitaram.
\newblock {The Uncertainty Principle: A Mathematical Survey}.
\newblock {\em Journal of Fourier Analysis and Applications}, 3(3):207--238,
  1997.

\bibitem{FQ16}
J.~Frikel and E.~T. Quinto.
\newblock Limited data problems for the generalized {R}adon transform in {$\Bbb
  R^n$}.
\newblock {\em SIAM J. Math. Anal.}, 48(4):2301--2318, 2016.

\bibitem{GR-fractional-laplacian-strong-unique-continuation}
M.-{\'A}. {Garc{\'\i}a-Ferrero} and A.~{R{\"u}land}.
\newblock Strong unique continuation for the higher order fractional
  {L}aplacian.
\newblock {\em Mathematics in Engineering}, 1(4):715--774, 2019.

\bibitem{GRSU-fractional-calderon-single-measurement}
T.~Ghosh, A.~R\"uland, M.~Salo, and G.~Uhlmann.
\newblock {Uniqueness and reconstruction for the fractional Calder\'on problem
  with a single measurement}.
\newblock {\em Journal of Functional Analysis}, 279(1):108505, 2020.

\bibitem{GSU-calderon-problem-fractional-schrodinger}
T.~{Ghosh}, M.~{Salo}, and G.~{Uhlmann}.
\newblock The {C}alder{\'o}n problem for the fractional {S}chr{\"o}dinger
  equation.
\newblock {\em Anal. PDE 13(2):455-475}, 2020.

\bibitem{G17}
F.~O. Goncharov.
\newblock An iterative inversion of weighted radon transforms along
  hyperplanes.
\newblock {\em Inverse Problems}, 33(12):124005, 20, 2017.

\bibitem{GN18b}
F.~O. Goncharov and R.~G. Novikov.
\newblock An example of non-uniqueness for {R}adon transforms with continuous
  positive rotation invariant weights.
\newblock {\em J. Geom. Anal.}, 28(4):3807--3828, 2018.

\bibitem{GN18a}
F.~O. Goncharov and R.~G. Novikov.
\newblock An example of non-uniqueness for the weighted {R}adon transforms
  along hyperplanes in multidimensions.
\newblock {\em Inverse Problems}, 34(5):054001, 6, 2018.

\bibitem{GO-range-of-d-plane-transform}
F.~B. Gonzalez.
\newblock {On the Range of the Radon $d$-Plane Transform and Its Dual}.
\newblock {\em Transactions of the American Mathematical Society},
  327(2):601--619, 1991.

\bibitem{HL19a}
B.~Harrach and Y.-H. Lin.
\newblock Monotonicity-based inversion of the fractional {S}chr\"{o}dinger
  equation {I}. {P}ositive potentials.
\newblock {\em SIAM Journal on Mathematical Analysis}, 51(4):3092--3111, 2019.

\bibitem{HL19b}
B.~Harrach and Y.-H. Lin.
\newblock {Monotonicity-based inversion of the fractional Schrödinger equation
  II. General potentials and stability}.
\newblock {\em SIAM Journal on Mathematical Analysis}, 52(1):402--436, 2020.

\bibitem{HLW06}
H.~Heck, X.~Li, and J.-N. Wang.
\newblock {Identification of Viscosity in an Incompressible Fluid}.
\newblock {\em Indiana University Mathematics Journal}, 56(5):2489--2510, 2007.

\bibitem{HE:integral-geometry-radon-transforms}
S.~{Helgason}.
\newblock Integral {G}eometry and {R}adon {T}ransforms.
\newblock Springer, {F}irst edition, 2011.

\bibitem{HH17}
A.~Homan and H.~Zhou.
\newblock Injectivity and stability for a generic class of generalized {R}adon
  transforms.
\newblock {\em J. Geom. Anal.}, 27(2):1515--1529, 2017.

\bibitem{HO:analysis-of-pdos}
L.~{H\"ormander}.
\newblock {The Analysis of Linear Partial Differential Operators I}.
\newblock {Springer-Verlag}, {Second} edition, 1990.

\bibitem{HO-topological-vector-spaces}
J.~{Horv\'ath}.
\newblock {Topological Vector Spaces and Distributions}.
\newblock volume~{I}. {Addison-Wesley}, 1966.

\bibitem{I15}
J.~Ilmavirta.
\newblock On {R}adon transforms on tori.
\newblock {\em J. Fourier Anal. Appl.}, 21(2):370--382, 2015.

\bibitem{IM-unique-continuation-riesz-potential}
J.~{Ilmavirta} and K.~{M\"onkk\"onen}.
\newblock Unique continuation of the normal operator of the x-ray transform and
  applications in geophysics.
\newblock {\em Inverse Problems}, 36(4):045014, 2020.

\bibitem{KKW-stability-of-interior-problems}
E.~Katsevich, A.~Katsevich, and G.~Wang.
\newblock Stability of the interior problem with polynomial attenuation in the
  region of interest.
\newblock {\em Inverse Problems}, 28(6):065022, 2012.

\bibitem{KEQ-wavelet-methods-ROI-tomography}
E.~Klann, E.~T. Quinto, and R.~Ramlau.
\newblock Wavelet methods for a weighted sparsity penalty for region of
  interest tomography.
\newblock {\em Inverse Problems}, 31(2):025001, 22, 2015.

\bibitem{KQ-microlocal-analysis-in-tomography}
V.~P. Krishnan and E.~T. Quinto.
\newblock {Microlocal Analysis in Tomography}.
\newblock In O.~Scherzer, editor, {\em Handbook of Mathematical Methods in
  Imaging}, pages 847--902. Springer, New York, 2015.

\bibitem{KR-all-functions-are-s-harmonic}
N.~Krylov.
\newblock {All functions are locally s-harmonic up to a small error}.
\newblock {\em Journal of Functional Analysis}, 277(8):2728 -- 2733, 2019.

\bibitem{KWA-ten-definitions-fractional-laplacian}
M.~Kwa\'snicki.
\newblock {Ten equivalent definitions of the fractional Laplace operator}.
\newblock {\em Fractional Calculus and Applied Analysis}, 20, 2015.

\bibitem{LL19}
R.-Y. Lai and Y.-H. Lin.
\newblock Global uniqueness for the fractional semilinear {S}chr\"{o}dinger
  equation.
\newblock {\em Proc. Amer. Math. Soc.}, 147(3):1189--1199, 2019.

\bibitem{LL-fractional-semilinear-problems}
R.-Y. {Lai} and Y.-H. {Lin}.
\newblock {Inverse problems for fractional semilinear elliptic equations}.
\newblock 2020.
\newblock arXiv:2004.00549.

\bibitem{LLR19}
R.-Y. Lai, Y.-H. Lin, and A.~Rüland.
\newblock {The Calderón problem for a space-time fractional parabolic
  equation}.
\newblock {\em SIAM Journal on Mathematical Analysis}, 52(3):2655--2688, 2020.

\bibitem{La00}
N.~Laskin.
\newblock {Fractional Quantum Mechanics and L\'evy Path Integrals}.
\newblock {\em Physics Letters A}, 268(4):298--305, 2000.

\bibitem{LA-fractional-quantum-mechanics}
N.~{Laskin}.
\newblock {Fractional Quantum Mechanics}.
\newblock {World Scientific}, {{F}irst} edition, 2018.

\bibitem{LI-fractional-magnetic-2}
L.~Li.
\newblock {A semilinear inverse problem for the fractional magnetic Laplacian}.
\newblock {\em arXiv:2005.06714}, 2020.

\bibitem{LI-fractional-magnetic-3}
L.~Li.
\newblock {Determining the magnetic potential in the fractional magnetic
  Calder{\'{o}}n problem}.
\newblock {\em arXiv:2006.10150}, 2020.

\bibitem{LI-fractional-magnetic}
L.~Li.
\newblock {The Calder{\'{o}}n problem for the fractional magnetic operator}.
\newblock {\em Inverse Problems}, 36(7):075003, 2020.

\bibitem{MS-theory-of-sobolev-multipliers}
V.~G. {Maz'ya} and T.~O. {Shaposhnikova}.
\newblock {Theory of Sobolev Multipliers}.
\newblock {Springer}, {First} edition, 2009.

\bibitem{Mc00}
S.~R. McDowall.
\newblock {An electromagnetic inverse problem in chiral media}.
\newblock {\em {Trans. Amer. Math. Soc. 352}}, 2000.

\bibitem{ML-strongly-elliptic-systems}
W.~{McLean}.
\newblock {Strongly Elliptic Systems and Boundary Integral Equations}.
\newblock {Cambridge University Press}, {{F}irst} edition, 2000.

\bibitem{KM-random-walks-quide-anomalous-diffusion}
R.~Metzler and J.~Klafter.
\newblock The random walk's guide to anomalous diffusion: a fractional dynamics
  approach.
\newblock {\em Physics Reports}, 339(1):1--77, 2000.

\bibitem{MI:distribution-theory}
D.~{Mitrea}.
\newblock Distributions, {P}artial {D}ifferential {E}quations, and {H}armonic
  {A}nalysis.
\newblock Springer, New York, {F}irst edition, 2013.

\bibitem{NSU95}
G.~Nakamura, Z.~Sun, and G.~Uhlmann.
\newblock {Global identifiability for an inverse problem for the Schr\"odinger
  equation in a magnetic field}.
\newblock {\em {Matematische Annalen, 303(1):377-388}}, 1995.

\bibitem{NT00}
G.~Nakamura and T.~Tsuchida.
\newblock {Uniqueness For An Inverse Boundary Value Problem For Dirac
  Operators}.
\newblock {\em Communications in Partial Differential Equations},
  25(7-8):557--577, 1999.

\bibitem{NU94}
G.~Nakamura and G.~Uhlmann.
\newblock {Global uniqueness for an inverse boundary problem arising in
  elasticity}.
\newblock {\em {Invent. Math.}}, 118, 1994.

\bibitem{NA-mathematics-computerized-tomography}
F.~{Natterer}.
\newblock {The Mathematics of Computerized Tomography}.
\newblock SIAM, Philadelphia, 2001.
\newblock Reprint.

\bibitem{Oz92}
T.~Ozawa.
\newblock {On critical cases of Sobolev inequalities}.
\newblock {\em {Hokkaido University, series 154}}, 1992.

\bibitem{QU-singularities-x-ray-transform-limited-data}
E.~Quinto.
\newblock {Singularities of the X-Ray Transform and Limited Data Tomography in
  $\mathbb{R}^2$ and $\mathbb{R}^3$}.
\newblock {\em SIAM Journal on Mathematical Analysis}, 24(5):1215--1225, 1993.

\bibitem{QU-artifacts-and-singularities-limited-tomography}
E.~Quinto.
\newblock Artifacts and {V}isible {S}ingularities in {L}imited {D}ata {X}-{R}ay
  {T}omography.
\newblock {\em Sensing and Imaging}, 18, 2017.

\bibitem{RA-periodic-radon-transform}
J.~Railo.
\newblock {Fourier Analysis of Periodic Radon Transforms}.
\newblock {\em Journal of Fourier Analysis and Applications}, 26(4):64, 2020.

\bibitem{RK-radon-transform-an-local-tomography}
A.~G. {Ramm} and A.~I. {Katsevich}.
\newblock {The Radon Transform and Local Tomography}.
\newblock CRC Press, Boca Raton, {F}irst edition, 1996.

\bibitem{R15}
T.~Reichelt.
\newblock {A comparison theorem between {R}adon and {F}ourier-{L}aplace
  transforms for D-modules}.
\newblock {\em Ann. Inst. Fourier (Grenoble)}, 65(4):1577--1616, 2015.

\bibitem{RI-liouville-riemann-integrals-potentials}
M.~Riesz.
\newblock Int{\'e}grales de {R}iemann-{L}iouville et potentiels.
\newblock {\em Acta Sci. Math. Szeged}, 9(1-1):1--42, 1938.

\bibitem{RO-nonlocal-elliptic-equations-bounded-domains}
X.~{Ros-Oton}.
\newblock Nonlocal elliptic equations in bounded domains: a survey.
\newblock {\em Publicacions Matem\'atiques}, 60:3 -- 26, 2015.

\bibitem{RU-unique-continuation-scrodinger-rough-potentials}
A.~R\"{u}land.
\newblock Unique continuation for fractional {S}chr\"{o}dinger equations with
  rough potentials.
\newblock {\em Comm. Partial Differential Equations}, 40(1):77--114, 2015.

\bibitem{RS18}
A.~R\"{u}land and M.~Salo.
\newblock Exponential instability in the fractional {C}alder\'{o}n problem.
\newblock {\em Inverse Problems}, 34(4):045003, 21, 2018.

\bibitem{ruland2019quantitative}
A.~R{\"u}land and M.~Salo.
\newblock Quantitative approximation properties for the fractional heat
  equation.
\newblock {\em Mathematical Control \& Related Fields}, pages 233--249, 2019.

\bibitem{RS-fractional-calderon-low-regularity-stability}
A.~{R\"uland} and M.~{Salo}.
\newblock {The fractional Calder\'on problem: Low regularity and stability}.
\newblock {\em Nonlinear Analysis}, 2019.

\bibitem{Sa07}
M.~Salo.
\newblock {Recovering first order terms from boundary measurements}.
\newblock {\em { J. Phys.: Conf. Ser.}}, 73, 2007.

\bibitem{SA:calderon-problem}
M.~{Salo}.
\newblock Calder\'on problem.
\newblock 2008.
\newblock Lecture notes.

\bibitem{SA:fourier-analysis-distributions}
M.~{Salo}.
\newblock Fourier analysis and distribution theory.
\newblock 2013.
\newblock Lecture notes.

\bibitem{Sal17}
M.~Salo.
\newblock {The fractional Calderón problem}.
\newblock {\em Journées équations aux dérivées partielles}, Exp. No.(7),
  2017.

\bibitem{SI-regularity-obstacle-problem}
L.~Silvestre.
\newblock {Regularity of the obstacle problem for a fractional power of the
  Laplace operator}.
\newblock {\em Communications on Pure and Applied Mathematics}, 60:67 -- 112,
  2007.

\bibitem{SU:microlocal-analysis-integral-geometry}
P.~{Stefanov} and G.~{Uhlmann}.
\newblock Microlocal {A}nalysis and {I}ntegral {G}eometry (working title).
\newblock 2018.
\newblock Draft version.

\bibitem{TRE:topological-vector-spaces-distributions}
F.~{Tr\`eves}.
\newblock Topological {V}ector {S}paces, {D}istributions and {K}ernels.
\newblock Academic Press, {F}irst edition, 1967.

\bibitem{UH-inverse-problems-seeing-the-unseen}
G.~Uhlmann.
\newblock Inverse problems: seeing the unseen.
\newblock {\em Bulletin of Mathematical Sciences}, 4(2):209--279, 2014.

\bibitem{XI-note-on-fractional-poincare}
L.~Xiaojun.
\newblock {A Note On Fractional Order Poincar\'es Inequalities}.
\newblock 2012.

\bibitem{YYJW-high-order-TV-minimization}
J.~Yang, H.~Yu, M.~Jiang, and G.~Wang.
\newblock High-order total variation minimization for interior tomography.
\newblock {\em Inverse Problems}, 26(3):035013, 2010.

\bibitem{YA-higher-order-laplacian}
R.~{Yang}.
\newblock {On higher order extensions for the fractional {L}aplacian}.
\newblock 2013.
\newblock arXiv:1302.4413.

\bibitem{YYW-interior-reconstruction-limited-angle-data}
Y.~Ye, H.~Yu, and G.~Wang.
\newblock {Exact Interior Reconstruction from Truncated Limited-Angle
  Projection Data}.
\newblock {\em International Journal of Biomedical Imaging}, vol. 2008, 2008.

\bibitem{YW-compressed-interior-tomography}
H.~Yu and G.~Wang.
\newblock {Compressed sensing based interior tomography}.
\newblock {\em Physics in Medicine and Biology}, 54(9):2791--2805, 2009.

\end{thebibliography}
\bibliographystyle{abbrv}

\end{document}